\documentclass[a4paper,11pt]{article}
\usepackage{amsmath,amsthm,amsfonts,amssymb,bbm}
\usepackage{graphicx,psfrag,subfigure,color,cite}
\usepackage{enumerate}
\usepackage[UKenglish]{babel}
\usepackage{bm}
\usepackage{changes}
\usepackage[colorlinks]{hyperref}
\usepackage{stmaryrd}
\usepackage[percent]{overpic}
\usepackage{epstopdf}
\numberwithin{equation}{section}

\newcommand{\leftcontour}[2]{{\tiny #1}\,\raisebox{-0.25em}{\includegraphics[scale=0.035]{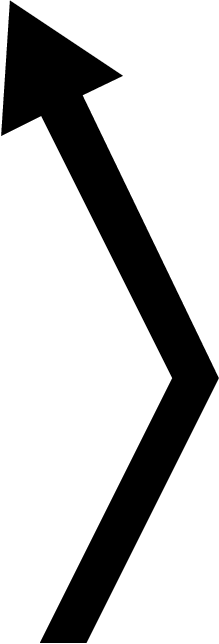}}\,{\tiny #2}}
\newcommand{\rightcontour}[2]{{\tiny #1}\,\raisebox{-0.25em}{\includegraphics[scale=0.035]{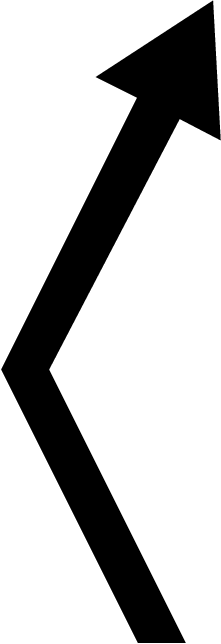}}\,{\tiny #2}}

\usepackage[margin=1.2in]{geometry}

\newcommand{\Pb}{\mathbb{P}}
\newcommand{\dx}{\mathrm{d}}
\newcommand{\E}{\mathbb{E}}
\newcommand{\I}{\mathrm{i}}
\newcommand{\R}{\mathbb{R}}

\newcommand{\Z}{\mathbb{Z}}
\newcommand{\C}{\mathbb{C}}
\newcommand{\Id}{\mathbbm{1}}
\newcommand{\Ai}{\mathrm{Ai}}
\newcommand{\kai}{K_{{\rm Ai}}}

\newcommand{\Or}{{\mathcal{O}}}
\newcommand{\e}{\varepsilon}
\newcommand{\Exp}{{\rm Exp}}
\newcommand{\Geom}{{\rm Geom}}
\newcommand{\inner}[1]{\langle #1\rangle}

\newcommand{\bfa}{\mathbf{a}}
\newcommand{\bfb}{\mathbf{b}}
\newcommand{\bfr}{\mathbf{r}}
\newcommand{\bfR}{\mathbf{R}}
\newcommand{\bfx}{\mathbf{x}}
\newcommand{\bfy}{\mathbf{y}}

\newcommand{\bma}{\bm\alpha}
\newcommand{\bmb}{\bm\beta}
\newcommand{\lpp}{L^{\bma,\bmb}}

\newcommand{\wt}[1]{\widetilde{#1}}
\newcommand{\wh}[1]{\widehat{#1}}

\newcommand{\mci}{\mathcal I}
\newcommand{\kgeo}{K^{{\rm geo}}}
\newcommand{\kexp}{K^{{\rm exp}}}
\newcommand{\kexpresc}{K^{{\rm exp, resc}}}

\newtheorem{prop}{Proposition}[section]
\newtheorem{thm}[prop]{Theorem}
\newtheorem{lem}[prop]{Lemma}
\newtheorem{defin}[prop]{Definition}
\newtheorem{cor}[prop]{Corollary}
\newtheorem{rem}[prop]{Remark}

\title{Critical fluctuations of last passage\\ percolation with thick boundaries}
\author{Elnur Emrah\thanks{University College Dublin, School of Mathematics and Statistics, Dublin, Ireland.\\ Email: {\tt elnur.emrah@ucd.ie}}
\and Patrik L. Ferrari\thanks{Bonn University, Institute for Applied Mathematics, Bonn, Germany. Email: {\tt ferrari@uni-bonn.de}}
\and Min Liu\thanks{Bonn University, Institute for Applied Mathematics, Bonn, Germany. Email: {\tt liu@iam.uni-bonn.de}}}

\date{July 20, 2026}

\begin{document}

\maketitle

\begin{abstract}
We consider the exponential last passage percolation (LPP) with thick two-sided boundary that consists of a few inhomogeneous columns and rows.
Ben Arous and Corwin previously studied the limit fluctuations in this model except in a critical regime, for which they predicted that the limit distribution exists and depends only on the most dominant columns and rows. In this article, we prove their conjecture and identify the limit distribution explicitly in terms of a
Fredholm determinant of a $2 \times 2$ matrix kernel. This result leads in particular to an explicit variational formula for the one-point
marginal of the KPZ fixed point for a new class of initial conditions. Our limit distribution is also a novel generalization of and provides a
new numerically efficient representation for the Baik--Rains distribution.
\end{abstract}

\maketitle

\section{Introduction}

\subsection{Exponential LPP with thick two-sided boundary}

Our first motivation in this article is to address a conjecture of Ben Arous and Corwin concerning the limit fluctuations of the exponential last passage percolation (LPP) with thick two-sided boundary conditions (see (51) of~\cite{BC09}). To describe the model, we consider two vectors $\bma = (\alpha_1, \dotsc, \alpha_\ell) \in \R^\ell$ and $\bmb = (\beta_1, \dotsc, \beta_k) \in \R^k$ with components of the form
\begin{align}
\label{E:Param}
\alpha_i = \wt{\alpha}_i + \frac{x_i}{(16N)^{1/3}} \quad \text{ and } \quad \beta_j = \wt{\beta}_j + \frac{y_j}{(16N)^{1/3}},
\end{align}
where $k, \ell \in \Z_{>0}$, $\wt{\alpha}_i, \wt{\beta}_j > -1/2$, $x_i, y_j \in \R$ are fixed, and the parameter $N \in \Z_{>0}$ is large enough so that $\alpha_i, \beta_j > -1/2$ as well. Let $(\omega_{i,j})$ be a family of independent random weights indexed by the lattice $\{-\ell+1, -\ell+2, \dotsc\} \times \{-k+1, -k+2, \dotsc\}$ with
\begin{equation}
\label{E:ThickBd}
		\omega_{i,j} \sim
		\begin{cases}
			\Exp(1) & \text{if } i > 0 \text{ and } j > 0, \\
			\Exp\left(\frac{1}{2} + \alpha_{i+\ell}\right) & \text{if } i \leq 0 \text{ and } j  > 0, \\
			\Exp\left(\frac{1}{2} + \beta_{j+k} \right) & \text{if } i > 0 \text{ and } j \leq 0, \\
			0 & \text{if } i \leq 0 \text{ and } j\leq 0.
		\end{cases}
\end{equation}
In particular, the weights on the $\ell \times k$ corner block $\{-\ell+1, \dotsc, 0\} \times \{-k+1, \dotsc, 0\}$ are zero. We think of the first integer quadrant $\Z_{>0}^2$ as the \emph{bulk}, and  the first $k$ rows and the first $\ell$ columns of the above lattice as the \emph{thick boundary} where the exponential rates can differ from the bulk rate $1$, see Figure~\ref{F:2Slpp}.
We then define, for each $(i,j),(m,n)\in\Z^2$ with $m \ge i > -\ell$ and $n \ge j > -k$, the associated last passage time by
\begin{equation}\label{estp}
		L^{\bma, \bmb}_{(i,j) \to (m,n)} := \max_{\pi:(i,j)\to(m,n)} \sum_{(r,s) \in \pi} \omega_{r,s},
	\end{equation}
where the maximum runs over all up-right paths from $(i, j)$ to $(m, n)$. In the case $(i,j) = (-k+1,-\ell+1)$, we abbreviate the left-hand side of \eqref{estp} as $L^{\bma, \bmb}_{m,n}$.

\begin{figure}
		\centering
\begin{overpic}[scale = 0.8]{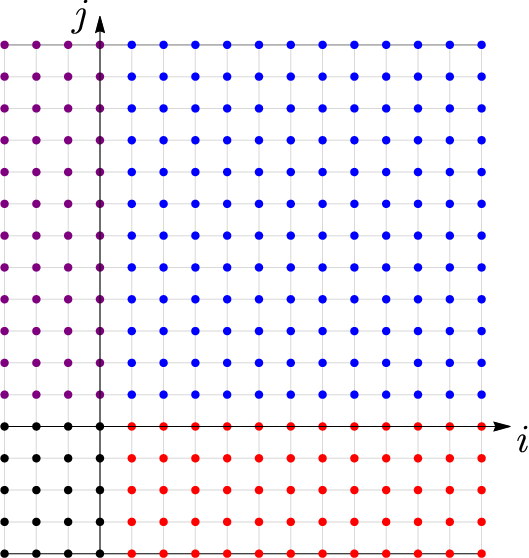}
\end{overpic}
\caption{Illustration of the weights of the exponential LPP with thick two-sided boundary conditions for $k = 5$ and $\ell = 4$. The bulk points (blue) have $ \Exp(1)$ weights. The points on the $\ell \times k$ corner block (black) have zero weights. The remaining points on the thick boundary have exponentially distributed weights with inhomogeneous rates that equal  $1/2+\alpha_i$ on column $i-\ell$ (purple) and $1/2+\beta_j$ on row $j-k$ (red).}
		\label{F:2Slpp}
\end{figure}

\subsection{Background: Pr\"ahofer--Spohn conjecture}

To explain the origin of the preceding model, we revisit the Pr\"ahofer--Spohn conjecture \cite{PS01} that was eventually settled in \cite{BC09}. Let us consider for now the basic case of the model with $k = \ell = 1$ and $x_1 = y_1 = 0$. We also assume that $\alpha = \alpha_1 \le 1/2$ and $\beta = \beta_1 \le 1/2$. Then the last passage time $L_{m, n}^{\alpha, \beta}$ is closely related in distribution to the current in the totally asymmetric simple exclusion process (TASEP) with two-sided Bernoulli initial conditions \cite{PS01, BCS06, BC09}. More specifically, at time zero, each site $i \in \Z$ is occupied independently with probability $\rho^- = 1/2 + \alpha$ if $i \le 0$ and with probability $\rho^+ = 1/2 -\beta$ if $i > 0$. This implies that the initial particle density is $\rho^-$ to the left of the origin and $\rho^+$ to the right. When $\rho^- = \rho^+ = \rho$, one obtains the stationary initial condition for the TASEP.  In the article \cite{PS01}, Pr\"{a}hofer and Spohn predicted the limiting current fluctuations for each choice of $\rho^-$ and $\rho^+$, relying in part on the analogous results of Baik and Rains \cite{BR00} for the polynuclear growth model (PNG), and also on a result of Johansson \cite{Jo00b} for the step initial condition where $\rho^- = 1$ and $\rho^+ = 0$. Translated to the LPP language, the Pr\"ahofer--Spohn conjecture asserts the limit distribution of suitably rescaled $L_{m, n}^{\alpha, \beta}$ as the point $(m, n)$ tends to infinity in a fixed direction $r$, which can be defined as the limit of the ratio $m/n$. The limit exhibited an interesting phase diagram with respect to the parameters $\rho^+$, $\rho^-$ and $r$, and it became a significant problem to verify this picture rigorously due to the role of the TASEP as a prototypical nonreversible interacting particle system.

The proof of the Pr\"ahofer--Spohn conjecture spanned several works. In the stationary case and for any direction $r \neq r_c = \rho^{-2}(1-\rho)^{2}$, it was already known that the fluctuations are governed by the central limit theorem \cite{FF94}. In this case, one side of the boundary completely determine the fluctuation behavior. For $\rho^- = 1$ and $\rho^+ = 0$, there is a rarefaction fan and, for any $r$, the limit is the Tracy--Widom GUE distribution~\cite{TW94} as proved in \cite{Jo00b}. For the stationary case, in the critical direction $r_c$, both the bulk and the two boundaries have non-trivial contributions in the limit distribution, which is the Baik--Rains distribution~\cite{BR00} as shown in \cite{FS05a}. In a subsequent work \cite{BBP06}, Baik, Ben Arous and P\'{e}ch\'{e} resolved the conjecture for $\rho^+ = 0$ and arbitrary $\rho^-$; see also \cite{NS04}. Using the results of \cite{BBP06} as inputs, Ben Arous and Corwin completed the proof of the conjecture in \cite{BC09}.

\subsection{A conjecture of Ben Arous and Corwin}

As the authors noted, the probabilistic approach in \cite{BC09} works more generally to also identify the limit fluctuations for the exponential LPP with thick boundaries. The basic idea is that, through definitions \eqref{E:ThickBd} and \eqref{estp}, one can write
\begin{align}
\label{E:maxId}
L_{m, n}^{\bma, \bmb} = \max \{L_{(-\ell+1, 0) \to (m, n)}^{\bma, \bmb}, L_{(0, -k+1) \to (m, n)}^{\bma, \bmb}\}
\end{align}
because the weights are zero on the $\ell \times k$ corner block. The two last passage times on the right-hand side depend only on the one side of the thick boundary (either the purple or the red part in Figure \ref{F:2Slpp}). Therefore, the results of \cite{BBP06} immediately yield the limit distribution for each term of the maximum in \eqref{E:maxId}. From this knowledge, the article \cite{BC09} determined the limit distribution of $L_{m, n}^{\bma, \bmb}$ in all regimes except the critical one, which generalizes taking $r = r_c$ in the stationary case where $\rho^\pm = \rho$ and $k = \ell = 1$. When the critical direction is along the diagonal ($r_c = 1$), the critical regime is precisely when $\wt{\alpha}_i = \wt{\beta}_j = 0$ for some $i$ and $j$ in \eqref{E:Param}. The definition is the same for other choices of $r_c$ except that $1/2$ in the second and third lines of \eqref{E:ThickBd} need to be modified as $1/(1+\sqrt{r_c})$ and $\sqrt{r_c}/(1+\sqrt{r_c})$, respectively. The approach of \cite{BC09} breaks down for this regime because then neither term in the maximum \eqref{E:maxId} dominates the other, and these terms are significantly correlated. For $k = \ell = 1$, the critical regime was handled in \cite{FS05a} through a \emph{shift argument} (see also~\cite{SI04}), which has since become a standard technique to establish convergence to the Baik--Rains distribution. However, the shift argument does not generalize to the thick boundary setting.

It was conjectured in (51) of \cite{BC09} that, with appropriate rescaling, $L_{m, n}^{\bma, \bmb}$ converges in distribution in the critical regime, and the limit distribution depends only on those $x_i$ and $y_j$ parameters for which $\wt{\alpha}_i = \wt{\beta}_j = 0$. One of our main result verifies the preceding conjecture for $r_c = 1$ with an explicit identification of the limit, see Theorem~\ref{thm:AC_conj}. The choice of the diagonal direction is only for notational simplicity; a version of our result for any other direction can be formulated and proved with minor modifications.

The proof of Theorem~\ref{thm:AC_conj} begins with a derivation of the distribution of $L_{m, n}^{\bma, \bmb}$ for arbitrary vectors $\bma$ and $\bmb$ with components $\alpha_i, \beta_j  > -1/2$ but not necessarily of the form \eqref{E:Param}. This intermediate result, which is recorded as Theorem~\ref{thm:expKernel}, is also new and might be of independent interest. In its proof, we first assume that $\alpha_i,  \beta_j > 0$ for all $i, j$. Then we can compute the joint distribution of the two terms of the maximum in \eqref{E:maxId} through the associated Schur process \cite{OR01} and a limit transition from the geometric to the exponential LPP. Doing so yields the distribution of $L_{m, n}^{\bma, \bmb}$ as a Fredholm determinant of an explicit $2 \times 2$ kernel with entries expressed as contour integrals. Through analytic continuation of these integrals, we then show that the distributional formula remains valid even without the positivity restriction on the parameters $\alpha_i$ and $\beta_j$. Such an extension is essential for our treatment of the critical regime where $\wt{\alpha}_i = \wt{\beta}_j = 0 $ and so $\alpha_i$ and $\beta_j$ can be nonpositive. We then apply Theorem~\ref{thm:expKernel} choosing $\bma$ and $\bmb$ as in \eqref{E:Param} and $(m, n) = (N, N) + \tau  (2N)^{2/3} (-1, 1)$ where $\tau$ is a real parameter. For the asymptotics as $N \to \infty$, we perform a saddle-point, steep-descent analysis along the lines of \cite{BFP09}, which leads to Theorem~\ref{thm:AC_conj}.

\subsection{A new representation of the Baik--Rains distribution}

The novel limit distribution that appears in Theorem~\ref{thm:AC_conj} is given by a Fredholm determinant of a $2 \times 2$ limit kernel, see Definition~\ref{def:gBR}. To connect it to the Baik--Rains distribution, recall that the case $k = \ell = 1$ and $\alpha = \beta = 0$ of our setting corresponds to the stationary TASEP with density $\rho = 1/2$. It was proved in \cite{FS05a} that the last passage time $L_{m, n}^{0, 0}$ under the same scaling as in Theorem~\ref{thm:AC_conj} converges to the Baik--Rains distribution (with parameter $\tau$). Therefore, our limit distribution recovers the Baik--Rains distribution in the above special case. However, it is not straightforward to see this directly because our representation is significantly far from the original definition of the Baik--Rains distribution \cite{BR00}, which is in terms of a Painlev\'e transcendent. In Section~\ref{sec:BR}, we provide a direct proof that the two definitions are indeed the same.

In Corollary~\ref{cor:toBR}, we rewrite our limit kernel for the Baik--Rains distribution in terms of the Airy functions and the Airy kernel. This yields a new Fredholm determinant formula for the Baik--Rains distribution that is simpler than the earlier such representation obtained in \cite{FS05a}, which we recall in Appendix~\ref{AppBR}. Our representation is particularly well-suited for efficient numerical computation via Bornemann's method \cite{Born08}, which provides an alternative to the earlier evaluations of the Baik--Rains distribution via Painlev\'e transcendents \cite{PS02b,PSKPZ}. In Figure~\ref{fig:compare} ahead, we employ both methods to produce matching plots of the Baik--Rains distribution for two choices of the $\tau$ parameter.

\subsection{A generalization of the Baik--Rains distribution}

We next return to the setting of thick boundary, and assume that $\alpha_i = \beta_j = 0$ for all $i, j$. For this case, we show in Theorem~\ref{thm:var} that our limit distribution also admits an explicit variational description involving the Airy$_2$ process and $k+\ell$ independent Brownian motions. For $k = \ell = 1$, our result specializes to the variational formula in Corollary 2.4 of~\cite{CFS16} for the Baik--Rains distribution.

Variational formulas are also available for other initial conditions, see~\cite{Jo03,QR13,QR16,BL13,CLW16,CFS16,FO18}. However, for the generic case, the kernel is expressed in terms of hitting time of the hypograph of the initial condition \cite{MQR17} and, thus, is not fully explicit. Indeed, the cases of variational formulas with explicit kernels essentially reduce to the initial conditions scaling to wedge, flat, Brownian initial conditions and their mixtures, for which the joint distributions are also known~\cite{PS02,Jo03b,SI03,BFPS06,BFS07,BFP06,BFP09}. With our result, we obtain a new family of initial conditions for which the one-point marginal of the KPZ fixed point has explicit determinantal formulas.

In the proof of Theorem~\ref{thm:var}, a main step is to control the exit point where the geodesic from the origin leaves the thick boundary. We achieve this with an exponential tail bound in Proposition~\ref{p:exit} by adapting the approach in \cite{BHA20} for the case $k = \ell = 1$.

Our final result concerns the following variation of the model in \eqref{E:ThickBd}. First, setting $x_1 = y_1 = 0$ and $\wt{\alpha}_i = \wt{\beta}_j = 0$ for all $i, j$, we assume that $x_i + y_j > 0$ for all $i, j$ except when $i = j = 1$.  We then introduce independent weights $(\wh{\omega}_{i, j})$ such that
\begin{align}
\label{E:StThickBd}
\wh{\omega}_{i,j} \sim
		\begin{cases}
			\Exp(1), & \text{if } i > 0 \text{ and } j > 0, \\
			\Exp\left(\frac{1}{2} + \alpha_{i+\ell}\right), & \text{if } i \leq 0 \text{ and } j  > 0, \\
			\Exp\left(\frac{1}{2} + \beta_{j+k} \right), & \text{if } i > 0 \text{ and } j \leq 0, \\
			\Exp\left(\alpha_{i+\ell} + \beta_{j+k}\right) &\text{if } i \leq 0 \text{ and } j\leq 0 \text{ and } (i, j) \neq (-\ell+1, -k+1), \\
			0, & \text{if } i = -\ell+1 \text{ and } j = -k + 1.
		\end{cases}
\end{align}
In other words, we replace the zero weights on the $\ell \times k$ corner block (shown in black in Figure~\ref{F:2Slpp}) except the one at the corner point $(-\ell+1, -k+1)$ with independent rate $\alpha_{i+\ell} + \beta_{j+k} = (16N)^{-1/3} (x_{i+\ell}+y_{j+k}) > 0$ exponentials. We define the last passage time $\wh{L}_{m, n}^{\bma, \bmb}$ from the origin as in \eqref{estp} using the $\wh{\omega}$ weights in place of $\omega$.

One can view \eqref{E:StThickBd} as an intermediate model between \eqref{E:ThickBd} and the LPP model introduced by Borodin and P\'ech\'e \cite{BP07} where the weight at $(-\ell+1, -k+1)$ is also an independent exponential with rate $\alpha_{1}+\beta_{1} = (16N)^{-1/3} (x_{1}+y_{1})$, assuming further that $x_1 + y_1 > 0$. Alternatively, one can think of \eqref{E:StThickBd} as the basic case $k = \ell = 1$ of \eqref{E:ThickBd} with a few inhomogeneous columns and rows.

Our interest in \eqref{E:StThickBd} stems from the fact that the process $\wh{L}^{\bma, \bmb}$ has stationary increments \cite{E16} and, thus, generalizes the stationary LPP process $L^{0, 0}$. Given that the  Baik--Rains distribution governs the limit fluctuations of $L^{0, 0}$ along the critical direction, it is natural to wonder what form the preceding statement takes for $\wh{L}^{\bma, \bmb}$. We address this question in Proposition~\ref{P:gBR2} and find that the associated limit distribution is a variation of the distribution in Definition~\ref{def:gBR}. We will only briefly indicate how to obtain this result since its proof is very similar to that of Theorem~\ref{thm:AC_conj}.

\subsection{Further related literature}

Under mild assumptions, LPP with general IID weights is expected to belong to the Kardar--Parisi--Zhang (KPZ) universality class \cite{KPZ86}, which has been a subject of intense research in the past several decades. For a detailed introduction, we refer to the review articles and lecture notes \cite{FS10,Cor11,QS15,BG16,Qua11,Fer10b,Tak16, Zyg18}. Briefly, the KPZ class is a conjectural collection of stochastic models with similar large-scale fluctuations, and is believed to include various growth models, interacting particle systems and random matrix ensembles. In the appropriate scaling limit, the fluctuations of any KPZ class model are supposedly described by each of the two closely related universal objects, the KPZ fixed point and the directed landscape. The KPZ fixed point is a Markov process on the space of upper-semicontinuous functions on $\R$, and was originally constructed in \cite{MQR17} as a limit of the TASEP. The directed landscape is a random directed metric on $\R^4$ and was first obtained in \cite{DOV22} as the limit of the Brownian LPP. The relation between the two is that the KPZ fixed point can be represented through a variational formula in terms of its initial condition and the directed landscape \cite{NQR20}. For several integrable and adjacent models, the convergence to the directed landscape and the KPZ fixed point has been established \cite{DV21, ACH24, QS20, Wu23, Vir20}, giving rigorous evidence for the universality of these limits. Integrability here means the possibility of deriving exact distributional formulas that can be studied through asymptotic techniques to access limit statistics. For the exponential LPP and TASEP in particular, a source of integrability is the connection to the Schur measures and processes \cite{Jo00b, Ok01, OR01}.

The Baik--Rains distribution is believed to describe the limit fluctuations along the critical direction in any KPZ model under stationary initial conditions. A distributional convergence statement to this effect was first proved for the PNG \cite{BR00} where the Baik--Rains distribution was introduced. For the TASEP and exponential LPP, the analogous result was obtained in \cite{FS05a} and subsequently extended to multi-point convergence in \cite{BFP09}.  The latter type of result was then derived for the Brownian LPP in \cite{FSW15}. By now, convergence to the Baik--Rains distribution has also been established in several positive-temperature models including the KPZ equation \cite{BCFV14}, the asymmetric simple exclusion process (ASEP) and the stochastic six vertex model \cite{Agg16}, $q$-TASEP \cite{SI17}, O'Connell--Yor polymer \cite{SI17b}, and the stochastic higher spin six vertex model \cite{IMS19}. The half-space analogue of the Baik--Rains distribution and the associated multi-point process were introduced in the articles \cite{BFO19} and \cite{BFO20}, respectively. In the setting of stationary TASEP on a ring, a periodic version of the Baik--Rains distribution was found by \cite{L18}, see also \cite{Prol16}. More recently,  the two-time extension of the Baik--Rains distribution was described in \cite{Rah25}. We also mention the very recent work \cite{Ze25} that obtained the two-parameter counterpart of the half-space Baik--Rains distribution under the critical scaling limit of the stationary half-space geometric LPP.

Inhomogeneous extensions of integrable KPZ models have also attracted significant research attention lately, thanks in part to featuring rich limit statistics and interesting phase transitions, a prominent example being the Baik--Ben Arous--P\'ech\'e (BBP) transition \cite{BBP06}. In the exponential LPP, varying the exponential rates in the form $a_i+b_j$ at each site $(i, j)$ retains integrability. Consequently, many aspects of the exponential LPP with such inhomogeneous rates have been studied extensively. A nonexhaustive list of prior works covers exact distributional identities \cite{BP07, DW08, Def08, AvMW13, JR22}, hydrodynamics and limit shape \cite{BFC96, SK99, E16, EJS21}, large deviations \cite{EJ17}, limit fluctuations \cite{BBP06, BP07, Kar07, Joh08, E16PhD, BC09, BV13, BV16, CLW16, Dim26, JR22} and infinite geodesics \cite{EJS25}. See also the further related works on other integrable inhomogeneous LPP \cite{GTW02, GTW02b, IS07, JR22, KLO21, DW07, FAvM08} and the generalizations of the TASEP with particlewise, spatial or temporal inhomogeneity \cite{KPS19, RS06, Arai20, BLSZ23, Assi20, Petr20}.

\subsection{Notation}

For $k\in\Z_{\geq 1}$, we write $\llbracket k \rrbracket := \{1,2,\ldots,k\}$. Bold font denotes vectors. Given two vectors $\bma := (\alpha_1, \ldots, \alpha_\ell) \in \R^{\ell}$ and $\bmb := (\beta_1, \ldots, \beta_k) \in \R^k$, we denote their concatenation by $ \bma, \bmb$,
that is,
	$\bma, \bmb := (\alpha_1, \ldots, \alpha_\ell, \beta_1, \ldots, \beta_k).$
We also denote the indices of the zero components of $\bma$ by $\mathcal{I}(\bma) := \{ i \mid \alpha_i = 0 \}$
  and for $i\in\llbracket\ell\rrbracket$, we define $\bma_{[i:\ell]}:=\{\alpha_i,\alpha_{i+1},\ldots,\alpha_\ell\}.$

For a meromorphic function $f$ and a finite subset $I$ of the complex plane, we will write
$\Gamma_I$ for any simple anticlockwise oriented contour that encloses $I$ but not any of the poles of $f$ in the complement of $I$. Following the convention in~\cite{BFO20}, we also introduce notation for the standard Airy contours: For subsets $I, J \subset \mathbb{C}$, we denote by $\leftcontour{I}{J} $ (resp.\ $\rightcontour{I}{J}$) an upward-oriented contour from $e^{-2\pi \I /3} \infty$ (resp.\ $e^{-\pi\I /3} \infty$) to $e^{2\pi\I /3} \infty$ (resp.\ $ e^{\pi \I/3}\infty $) such that all points in $I$ are to its left and all points in $J$ are to its right.
Also, we denote by $B_r$ the anticlockwise-oriented circle centered at the origin and of radius $r > 0$.

The Fredholm determinant of a $2\times 2$ matrix kernel $K$ can be defined by its Fredholm series expansion as
\begin{equation*}
\det\left(\Id-P_sK P_s\right)_{L^2(\R)} = \sum_{\ell = 0}^\infty \frac{(-1)^\ell}{\ell!}\sum_{k_1,\ldots,k_\ell=1}^2\int_s^\infty \dx x_1 \cdots \int_s^\infty \dx x_\ell \det_{1\leq i,j\leq \ell}[K_{k_i k_j}(x_i,x_j)].
\end{equation*}

\subsection{Acknowledgments}
We are grateful to Michael Pr\"ahofer for kindly providing the numerical data for the Painlev\'e transcendents and for explaining how to use it to compute the Baik--Rains distribution. E.\ Emrah would also like to thank G\'erard Ben Arous, Ivan Corwin and Herbert Spohn for their helpful comments at the preliminary stages of this project.
E.\ Emrah was partially supported by the EPSRC grant EP/W032112/1 and by the grant KAW 2015.0270 from the Knut and Alice Wallenberg Foundation. P.L.\ Ferrari and M.\ Liu were partly funded by the Deutsche Forschungsgemeinschaft (DFG, German Research Foundation) by the CRC 1720 – 539309657 and under Germany’s Excellence Strategy - EXC 2047/1 – 390685813.

\section{Main results}\label{sectResults}

In this section, we state our main results.

\subsection{One-point distribution}

Our first result gives the distribution of the last passage time $L^{\bma, \bmb}_{m,n}$ for any $\bma, \bmb > -1/2$ componentwise. In its statement, we use the function $E =E_{m, n}$ defined by
	\begin{equation} \label{E:P}
		E_{m,n}(x, w; y, v) := \left(\frac{1 / 2-w}{1 / 2-v}\right)^m\left(\frac{1/2 + v}{1/2 + w} \right)^{n} e^{w x - v y} \quad \text{ for } x, w, y, v \in \C.
	\end{equation}
	
	\begin{thm}[Distribution of Exponential LPP]\label{thm:expKernel}
		For any $\bma,\bmb>-\dfrac12$ and $m,n\in \Z_{>0}$,
		\begin{equation}\label{right}
			\Pb\left(L^{\bma, \bmb}_{m,n}\leq s\right)=\det\left(\Id-P_sK_{m,n}^{\exp}P_s\right)_{L^2(\R)},
		\end{equation}
        where $P_s$ is the projection operator onto $(s,\infty)$ and
		 $K_{m,n}^{\exp}$ is a $2\times 2$ matrix kernel with entries:
		\begin{align}
				K_{m,n,11}^{\exp}(x, y) &:= \frac{-1}{(2\pi\I)^2} \oint_{\Gamma_{-1/2,-\bmb}} \dx w\oint_{\Gamma_{1/2}} \dx v  \,
				\frac{E(x, w; y, v)}{v-w} \prod_{j=1}^{k} \frac{v + \beta_j}{w + \beta_j}, \nonumber\\		
				K_{m,n,22}^{\exp}(x, y) &:= \frac{-1}{(2\pi\I)^2}\oint_{\Gamma_{-1/2}} \dx w \oint_{\Gamma_{1/2,\bma}} \dx v  \,
				\frac{E(x, w; y, v)}{v-w} \prod_{i=1}^{\ell} \frac{w - \alpha_i}{v - \alpha_i}, \nonumber\\ 		
				K_{m,n,21}^{\exp}(x, y) &:= \frac{-1}{(2\pi\I)^2} \oint_{\Gamma_{-1/2}} \dx w \oint_{\Gamma_{1/2}} \dx v \,
				\frac{E(x, w; y, v)}{v-w} \prod_{j=1}^{k} (v + \beta_j) \prod_{i=1}^{\ell} (w - \alpha_i), \\ 	
				K_{m,n,12}^{\exp}(x, y) &:= \frac{-\Id_{x>y}}{(2\pi\I)^2}
				\oint_{\Gamma_{-1/2,-\bmb}} \dx w\oint_{\Gamma_{1/2,\bma,w}} \dx v
				\,\frac{E(x, w; y, v)}{v-w}\frac{1}{ \prod_{j=1}^{k} (w + \beta_j) \prod_{i=1}^{\ell} (v - \alpha_i) }\nonumber\\
&+\frac{-\Id_{x\leq y}}{(2\pi\I)^2}
				\oint_{\Gamma_{1/2,\bma}} \dx v\oint_{\Gamma_{-1/2,-\bmb,v}} \dx w
				\,\frac{E(x, w; y, v)}{v-w}\frac{1}{ \prod_{j=1}^{k} (w + \beta_j) \prod_{i=1}^{\ell} (v - \alpha_i) }.\nonumber
		\end{align}
\end{thm}
This theorem is proved in Section~\ref{SectFiniteN}. The well-definedness of the Fredholm determinant on the right hand side of~\eqref{right} is proved in Appendix~\ref{Sec:Well}.

\subsection{Limit fluctuations in the critical regime}

We next choose the parameters $\bma$ and $\bmb$ as in \eqref{E:Param}, that is,
\begin{equation} \label{eq2.11}
				\bma = \tilde\bma + (16N)^{-1/3} \bfx, \quad
				\bmb = \tilde\bmb + (16N)^{-1/3} \bfy,
		\end{equation}
		where $\tilde\bma, \bfx \in \R^\ell$ and $\tilde\bmb,\bfy \in \R^k$ are such that
		\begin{equation}
			-\frac12<\tilde\bma + (16N)^{-1/3} \bfx,\tilde\bmb + (16N)^{-1/3} \bfy.
		\end{equation}
We also take the endpoint of the last passage time $L_{m, n}^{\bma, \bmb}$ to be
\begin{equation}\label{eq2.10}
(m, n) = (N - \tau(2N)^{2/3},\, N + \tau(2N)^{2/3})
\end{equation}
for some parameter $\tau \in \R$. Our aim is to identify the limit fluctuations of $L_{m, n}^{\bma, \bmb} = L_{N,\tau}^{\tilde\bma, \tilde\bmb, \bfx, \bfy}$ as $N \to \infty$.
				The scaling of $\bfx$ and $\bfy$ by $N^{-1/3}$ will ensure that these parameters contribute nontrivially in the scaling limit if the corresponding $\tilde\alpha_i$, $\tilde\beta_i$ are  $0$.
				
As mentioned in the introduction, the asymptotic behavior of the last passage time $L^{\tilde\bma, \tilde\bmb, \bfx, \bfy}_{N,\tau}$ was analyzed in~\cite{BC09} for various regimes of the parameters.  We summarize below some results from \cite{BC09} and refer to that article for further details.
		\begin{thm}[Theorem~2.1 of~\cite{BC09}]$ $
			\begin{enumerate}[(i)]
				\item If
                $-1/2<\min\left\{\tilde\alpha_i,\tilde\beta_j \mid i\in\llbracket \ell \rrbracket,\, j\in\llbracket k \rrbracket \right\} < 0,$
				the fluctuation of $L^{\tilde\bma, \tilde\bmb, \bfx, \bfy}_{N,\tau}$ is of order $\Or(N^{1/2})$ with known limiting distribution;
				\item If $\tilde\bma,\tilde\bmb\geq 0$ with $|\mci(\tilde\bma)|=|\mci(\tilde\bmb)|=0$, then the fluctuation of $L^{\tilde\bma, \tilde\bmb, \bfx, \bfy}_{N,\tau}$ is of order $\Or(N^{1/3})$ and the limiting distribution is GUE Tracy-Widom distribution.
				\item If $\tilde\bma,\tilde\bmb\geq 0$ with $|\mci(\tilde\bma)|=0$ and $|\mci(\tilde\bmb)|>0$, then the fluctuation of $L^{\tilde\bma, \tilde\bmb, \bfx, \bfy}_{N,\tau}$ is of order $\Or(N^{1/3})$ and the limiting distribution is $F_{|\mci(\tilde\bmb)|,\bfy_{\mci(\tilde\bmb)}}$ introduced in~\cite{BBP06}.
			\end{enumerate}
		\end{thm}
		
		In the case where $\tilde\bma, \tilde\bmb \geq 0$ with $\min\{|\mci(\tilde\bma)|, |\mci(\tilde\bmb)|\} > 0$, Ben Arous and Corwin conjectured the existence of a limit distribution $F_{|\mci(\tilde\bma)|, |\mci(\tilde\bmb)|, \bfx_{\mci(\tilde\bma)}, \bfy_{\mci(\tilde\bmb)}}$ that governs the fluctuations of $L_{N,\tau}^{\tilde\bma, \tilde\bmb, \bfx, \bfy}$ under the KPZ scaling in the large-$N$ limit (see (51) of~\cite{BC09}). Our next result establishes this conjecture and provides an explicit formula for the limit distribution. This new distribution, which generalizes the Baik--Rains distribution, can be introduced as follows.
		
	\begin{defin}[Generalized Baik--Rains distribution]\label{def:gBR}
		Let $k, \ell \in \mathbb{Z}_{\geq 1}$, $\bfx \in \R^{\ell}$, $\bfy \in \R^{k}$. We define the CDF of the generalized Baik–Rains distribution by
		\begin{equation}
			F_{\ell,k,\bfx,\bfy}(s)=\det\left(\Id-P_s K^{\bfx,\bfy}P_s\right)_{L^2(\R)},
		\end{equation}
        where $P_s$ is the projection operator onto $(s,\infty)$ and $K^{\bfx,\bfy}$ the $2\times 2$ matrix kernel with entries
\begin{align}
		K_{11}^{\bfx,\bfy}(\xi, \zeta) &:= \frac{1}{(2\pi \I)^2}
		\int_{\; \leftcontour{-\bfy}{}} \dx w
		\int_{\; \rightcontour{-\bfy,w}{}} \dx v \,
		\frac{e^{v^3/3 - v\zeta}}{e^{w^3/3 - w\xi}}  \frac{1}{v - w}
		\prod_{i=1}^{k}\frac{v + y_i}{w + y_i} , \nonumber \\
		K_{22}^{\bfx,\bfy}(\xi, \zeta) &:= \frac{1}{(2\pi \I)^2}\int_{\; \rightcontour{}{\bfx}} \dx v
		\int_{\; \leftcontour{}{\bfx,v}} \dx w\,
		\frac{e^{v^3/3 - v\zeta}}{e^{w^3/3 - w\xi}}  \frac{1}{v - w}
		\prod_{i=1}^{\ell} \frac{w - x_i}{v - x_i}, \label{E:GBR} \\
		K_{21}^{\bfx,\bfy}(\xi, \zeta) &:= \frac{1}{(2\pi \I)^2}
		\int_{\leftcontour{}{0}} \dx w
		\int_{\rightcontour{w}{}} \dx v \,
		\frac{e^{v^3/3 - v\zeta}}{e^{w^3/3 - w\xi}}  \frac{1}{v - w}
		\prod_{i=1}^{k} (v + y_i) \prod_{j=1}^{\ell} (w - x_j), \nonumber \\
		K_{12}^{\bfx,\bfy}(\xi,\zeta)&:=\frac{\Id_{\xi>\zeta}}{\left(2\pi \I\right)^2}\int_{\;\rightcontour{}{-\bfy,\bfx}}\dx v \bigg[ \int_{\;\leftcontour{}{-\bfy, \bfx,v}}\dx w f_{12}(\xi,w;\zeta,v)+\oint_{ \Gamma_{-\bfy}} \dx w f_{12}(\xi,w;\zeta,v)\bigg] \nonumber \\
		&\quad+\frac{\Id_{\xi\leq \zeta}}{\left(2\pi \I\right)^2}
		\int_{\leftcontour{-\bfy,\bfx}{}}\dx w\bigg[\int_{\rightcontour {-\bfy,\bfx,w}{}}\dx v f_{12}(\xi,w;\zeta,v)- \oint_{\Gamma_{\bfx}} \dx v f_{12}(\xi,w;\zeta,v)\bigg], \nonumber
\end{align}
where
\begin{equation}\label{E:f12}
	f_{12}(\xi,w;\zeta,v):=\frac{e^{v^3/3 - v\zeta}}{e^{w^3/3 - w\xi}}  \frac{1}{v-w}\frac{1}{\prod_{i=1}^{k}\left(w+y_i\right)\prod_{j=1}^{\ell}\left(v-x_j\right)}.
\end{equation}
In particular, the contours for $w,v$ are chosen not to intersect.
	\end{defin}
	
 	\begin{thm}\label{thm:AC_conj}
		Let $\tilde\bma \in \R^{\ell}_{\geq 0}$ and $\tilde\bmb \in \R^{k}_{\geq 0}$ with $\min\{|\mathcal{I}(\tilde\bma)|,\, |\mathcal{I}(\tilde\bmb)|\} > 0$. Then, for all $s \in \R$, the following convergence holds:
		\begin{equation} \label{E:AC_Conj}
			\lim_{N \to \infty} \Pb \left(L^{\tilde\bma, \tilde\bmb, \bfx, \bfy}_{N,\tau} \leq 4N + (16N)^{1/3} s \right)
			= F_{|\mathcal{I}(\tilde\bma)|,\, |\mathcal{I}(\tilde\bmb)|,\, \bfx_{\mathcal{I}(\tilde\bma)} - \tau,\, \bfy_{\mathcal{I}(\tilde\bmb)} + \tau}(s + \tau^2),
		\end{equation}
		where the distribution on the right-hand side is the generalized Baik–Rains distribution introduced in Definition~\ref{def:gBR}. With the notation $\bfx_{\mathcal{I}(\tilde \bma)}$ we mean the subset of $x_i$ such that $i\in \mathcal{I}(\tilde \bma)$, and similarly for $\bfy_{\mathcal{I}(\tilde \bmb)}$.
	\end{thm}
	
This theorem is proved in Section~\ref{sectAsymptotics}.

	\subsection{New determinantal formula for the Baik--Rains distribution}
		
We now connect Definition~\ref{def:gBR} to the classical Baik--Rains distribution and then record a new representation for the latter. In the special case $k=\ell=1$, it is known that
	\begin{equation}\label{E:brcvg}
		\lim_{N \to \infty} \mathbb{P}\left( L_{N,\tau}
		\leq 4N + (16N)^{1/3} s\right) = F_{{\rm BR},\tau}(s),
	\end{equation}
	where  $L_{N,\tau} = L_{N, \tau}^{0, 0, 0, 0}$ and $F_{{\rm BR},\tau}$ is the (CDF of the) Baik--Rains distribution. The distribution $F_{{\rm BR},\tau}$ was first discovered as a limiting law in the study of the LPP with Poisson weights, where it was expressed in terms of the solution to a certain Painlev\'e equation~\cite{BR00}. The analogous limit for the exponential LPP is \eqref{E:brcvg}, which was initially conjectured by Prähofer and Spohn~\cite{PS01}, and proved in~\cite{FS05a}.
	It follows from \eqref{E:brcvg} and Theorem~\ref{thm:AC_conj} that
	\begin{equation}\label{E:gT}
		F_{{\rm BR},\tau}(s)=F_{1,1,-\tau,\tau}(s+\tau^2).
	\end{equation}
	In Section~\ref{sec:BR}, we give a direct analytic proof of the preceding identity for $\tau=0$.

	In view of \eqref{E:gT}, using the integral representation of Airy function together with~\eqref{E:GBR}, we obtain a new representation of $F_{{\rm BR},\tau}$ as a Fredholm determinant of a $2\times 2$ kernel given in Corollary~\ref{cor:toBR} below. Our formula is simpler than the one available in~\cite{FS05a}; see Appendix~\ref{AppBR} for comparison.
	
	\begin{cor}\label{cor:toBR}
		For any $s, \tau \in \R$, we have
		\begin{equation}\label{E:oldAndNew}
			F_{{\rm BR},\tau}(s) = \det(\Id-P_{s+\tau^2} K^{{\rm BR},\tau} P_{s+\tau^2})_{L^2(\R)}
		\end{equation}
		where $K^{{\rm BR},\tau}$ is a $2\times 2$ matrix kernel with entries
		\begin{align}\label{E:brtau}
				K^{{\rm BR},\tau}_{11}(\xi, \zeta) &:=\kai(\xi,\zeta)-\Ai(\zeta)\int_0^\infty\dx\lambda\Ai(\xi+\lambda)e^{-\tau\lambda}+e^{-\frac{\tau^3}{3}+\tau \xi}\Ai(\zeta)\nonumber\\
				K^{{\rm BR},\tau}_{22}(\xi, \zeta) &:=\kai(\xi,\zeta)-\Ai(\xi)\int_0^\infty\dx\lambda\Ai(\zeta+\lambda)e^{\tau\lambda}+e^{\frac{\tau^3}{3}-\tau \zeta}\Ai(\xi)\nonumber\\
				K^{{\rm BR},\tau}_{21}(\xi, \zeta) &:=\tau^2\kai(\xi,\zeta)-\tau(\partial_\xi\kai(\xi,\zeta)-\partial_\zeta\kai(\xi,\zeta))-\partial_\xi\partial_\zeta \kai(\xi,\zeta)\\
				K^{{\rm BR},\tau}_{12}(\xi, \zeta) &:=-\int_0^\infty\dx\lambda\int_0^\infty\dx\gamma\; \kai(\xi+\gamma,\zeta+\lambda)e^{\tau(\lambda-\gamma)}\nonumber\\
				&\quad+e^{\frac{\tau^3}{3}-\tau\zeta}\int_0^\infty\dx\gamma\Ai(\xi+\gamma)\gamma e^{-\tau\gamma}+e^{-\frac{\tau^3}{3}+\tau\xi}\int_0^\infty\dx\lambda\Ai(\zeta+\lambda)\lambda e^{\tau\lambda}\nonumber\\
				&\quad-e^{\tau(\xi-\zeta)}(\tau^2-\min\{\xi,\zeta\}),\nonumber
		\end{align}
		where $\kai$ is the Airy kernel
		\begin{equation}\label{E:airk}
			\kai(\xi,\zeta):=\int_0^\infty\dx\lambda\Ai(\xi+\lambda)\Ai(\zeta+\lambda).
		\end{equation}
	\end{cor}
We will derive Corollary~\ref{cor:toBR} in Section~\ref{subsectBR}.

Relative to the classical representation of the Baik--Rains distribution, which is typically evaluated using Painlev\'e transcendents~\cite{PS02b,PSKPZ}, the alternative expression on the right-hand side of~\eqref{E:oldAndNew} enables a more efficient numerical computation via Bornemann's method~\cite{Born08}. Figure~\ref{fig:compare} illustrates the graph of $F_{{\rm BR},\tau}$ computed using both approaches.

 	\begin{figure}
 		\centering
 		\includegraphics[scale=1.3]{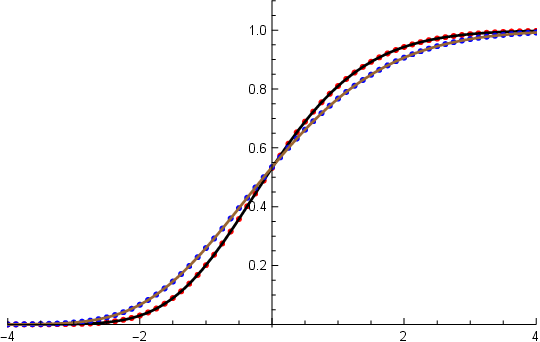}
 		\caption{Numerical comparison of both sides of~\eqref{E:oldAndNew} for $\tau \in \{1/2, 1\}$. The red (resp.\ blue) dots represent the right-hand side of~\eqref{E:oldAndNew} with $\tau = 1/2$ (resp.\ $\tau = 1$), computed using Bornemann's method~\cite{Born08}. The black (resp.\ purple) curve corresponds to the left-hand side of~\eqref{E:oldAndNew} for $\tau = 1/2$ (resp.\ $\tau = 1$), evaluated via the Painlev\'e transcendents. The corresponding code can be found in the BonnData repository~\cite{FK2/5PUB1P_2025}.	
 		}\label{fig:compare}
 	\end{figure}

\subsection{Variational formula}
	We next deduce a variational formula for the generalized Baik--Rains distribution. Let $k,\ell\in\Z_{>0}$, $t\geq 0$ and $B_1,\ldots,B_k,\tilde B_1,\ldots,\tilde B_\ell$ be independent standard Brownian motions. Define two independent Brownian last passage percolation~\cite{OCY01,Bar01,GTW00}
\begin{equation}\label{E:blpp}
	\begin{aligned}
		&\mathcal B_k(t):=\sup _{0\leq u_1\leq \ldots\leq u_{k-1}\leq t} \sum_{r=1}^{k}[B_r(u_r)-B_r(u_{r-1})],\\
		&\wt{\mathcal B}_\ell(t):=\sup _{0\leq v_1\leq \ldots\leq v_{\ell-1}\leq t} \sum_{r=1}^\ell[\tilde B_r(v_r)-\tilde B_r(v_{r-1})],
	\end{aligned}
\end{equation}
where we identify $u_0=v_0=0$, $u_k=v_\ell=t$.

Also, we denote by $\mathcal A_2(t)$ the Airy$_2$ process~\cite{PS02}. For the case of $\bma=\bmb=0$, the boundaries in our LPP under scaling limit converge to the processes in \eqref{E:blpp}, while the bulk to an Airy$_2$ process. In particular, the limiting distribution of our LPP, for which we have an explicit Fredholm determinant formula, also admits the following variational formula.
\begin{thm}[Variational Formula]\label{thm:var}
	For any $s\in\R$,
	\begin{equation}
		\Pb\left(\max_{t\in\R}\left\{\sqrt{2}\mathcal B_k(t)\Id_{t\geq 0}+\sqrt{2}\wt{\mathcal B}_\ell(-t)\Id_{t<0}+\mathcal A_2(t)-t^2\right\}\leq s\right)=F_{k,\ell,\mathbf{0},\mathbf{0}}(s),			
	\end{equation}
	where the right hand side is the generalized Baik--Rains distribution.
\end{thm}

As in~\cite{CFS16,CLW16}, one can consider also non-integrable boundaries. For instance, if the boundaries are replaced by generic i.i.d. random variables such that under scaling they converges (weakly) to the two-sided Brownian last passage percolation (and some control of the exit time of the geodesics as in the proof of Theorem~\ref{thm:var}), then Theorem~\ref{thm:var} still holds true.

This theorem is proved in Section~\ref{sec:var}. In the special case with $k = \ell = 1$, the result recovers Corollary 2.4 of~\cite{CFS16}.

\subsection{Inhomogeneous stationary LPP}

To state our final result, we turn to the model with the weights $\wh{\omega}$ given by \eqref{E:StThickBd}. As mentioned in the introduction, the increments of the associated last passage process $\wh{L}^{\bma, \bmb}$ are stationary in the following sense. For any $m, n \in \Z$ with $m > -\ell$ and $n > -k$,
\begin{align}
\label{E:IncSt}
\begin{split}
(\wh{L}^{\bma, \bmb}_{i, n}-\wh{L}^{\bma, \bmb}_{i-1, n}: i \in \Z \text{ and } i > -\ell+1) &\stackrel{\rm{(d)}}{=} (\wh{\omega}_{i, -k+1}: i \in \Z \text{ and } i > -\ell+1), \\
(\wh{L}^{\bma, \bmb}_{m, j}-\wh{L}^{\bma, \bmb}_{m, j-1}: j \in \Z \text{ and } j > -k+1) &\stackrel{\rm{(d)}}{=} (\wh{\omega}_{-\ell+1, j}: i \in \Z \text{ and } j > -k+1).
\end{split}
\end{align}
In other words, the last passage increments along any row (resp.\ column) have the same joint distribution as the weights on the lowest row $-k+1$ (resp.\ column $-\ell+1$). The stationary property in \eqref{E:IncSt} is obtained from Proposition 4.1 of~\cite{E16} by choosing the parameters there as
\begin{equation}
\begin{aligned}
a_i &=
\begin{cases}
\alpha_{i+1} = (16N)^{-1/3}x_{i+1} \quad &\text{ if } i < \ell, \\
1/2 \quad &\text{ if } i \ge \ell,
\end{cases} \\
b_j &= \begin{cases}
\beta_{j+1} = (16N)^{-1/3}y_{j+1} \quad &\text{ if } i < k, \\
1/2 \quad &\text{ if } j \ge k,
\end{cases}
\end{aligned}
\end{equation}
for $i, j \in \Z_{>0}$ and $z = 0$.

Our next result concerns the distribution function of \( \wh L^{\bma, \bmb}_{m,n} \). We will obtain this result for arbitrary parameters $\bma$ and $\bmb$ subject to the assumption that $\alpha_i + \beta_j \ge 0$ for all $i, j$.
The corresponding kernel is a slight modification of the one in Theorem~\ref{thm:expKernel}. Define
\begin{equation}\label{E:ehat}
	\widehat E(x,w;y,v):=E(x,w;y,v)\prod_{i=2}^\ell\frac{w-\alpha_i}{v-\alpha_i}\prod_{j=2}^k\frac{v+\beta_j}{w+\beta_j}.
\end{equation}
\begin{prop}\label{p:sta}
	For $m,n \in \Z_{>0}$ and $\bma \in \R^\ell$, $\bmb \in \R^k$ such that $\alpha_i+\beta_j\geq 0$ for all $(i,j)\in\llbracket\ell\rrbracket\times\llbracket k\rrbracket$, we have
	\begin{equation}
		\Pb\left(\wh L^{\bma, \bmb}_{m,n} \leq s\right) = \det\left(\Id - P_s \wh K_{m,n}^{\exp} P_s\right)_{L^2(\R)},
	\end{equation}
	where $\wh K_{m,n}^{\exp}$ is a $2\times 2$ matrix kernel with entries
	\begin{equation}
		\begin{aligned}
			\wh K_{m,n,11}^{\exp}(x, y) &:= \frac{-1}{(2\pi\I)^2} \oint_{\Gamma_{-1/2,-\bmb}} \dx w\oint_{\Gamma_{1/2,\bma_{2:\ell}}} \dx v  \,
			\frac{\wh E(x, w; y, v)}{v-w}  \frac{v + \beta_1}{w + \beta_1}, \\		
			K_{m,n,22}^{\exp}(x, y) &:= \frac{-1}{(2\pi\I)^2}\oint_{\Gamma_{-1/2,-\bmb_{2:k}}} \dx w \oint_{\Gamma_{1/2,\bma}} \dx v  \,
			\frac{\wh E(x, w; y, v)}{v-w} \frac{w - \alpha_1}{v - \alpha_1}, \\ 		
			K_{m,n,21}^{\exp}(x, y) &:= \frac{-1}{(2\pi\I)^2} \oint_{\Gamma_{-1/2,-\bmb_{2:k}}} \dx w \oint_{\Gamma_{1/2,\bma_{2:\ell}}} \dx v \,
			\frac{\wh E(x, w; y, v)}{v-w}   (v + \beta_1)   (w - \alpha_1), \\ 	
			K_{m,n,12}^{\exp}(x, y) &:= \frac{-\Id_{x>y}}{(2\pi\I)^2}
			\oint_{\Gamma_{-1/2,-\bmb}} \dx w\oint_{\Gamma_{1/2,\bma,w}} \dx v
			\,\frac{\wh E(x, w; y, v)}{v-w}\frac{1}{   (w + \beta_1)   (v - \alpha_1) }\\
			&+\frac{-\Id_{x\leq y}}{(2\pi\I)^2}
			\oint_{\Gamma_{1/2,\bma}} \dx v\oint_{\Gamma_{-1/2,-\bmb,v}} \dx w
			\,\frac{\wh E(x, w; y, v)}{v-w}\frac{1}{   (w + \beta_1)   (v - \alpha_1) },
		\end{aligned}		
	\end{equation}
	where the contours are chosen to be disjoint.
\end{prop}
This theorem is proved in Section~\ref{SectFiniteN}.

We now investigate last-passage percolation in the large-$N$ regime.
As in~\eqref{eq2.10}, we fix the endpoint and denote the associated last-passage time by $\wh L_{N,\tau}^{\bfx,\bfy}$.
The discussion above shows that, in the limit, only those $x_i$ and $y_j$ with $\tilde\alpha_i=\tilde\beta_j=0$ in~\eqref{eq2.11} contribute.
Accordingly, we choose the parameters as
\begin{equation}
	\bma = (16N)^{-1/3} \bfx, \quad
	\bmb = (16N)^{-1/3} \bfy,
\end{equation}
with $x_i + y_j > 0$ for all $(i,j) \in \llbracket \ell \rrbracket \times \llbracket k \rrbracket \setminus \{(1,1)\}$ and $x_1 + y_1 = 0$.  Define
\begin{equation}
	h(\xi,w;\zeta,v) := \frac{e^{v^3/3 - v\zeta}}{e^{w^3/3 - w\xi}} \frac{1}{v - w} \prod_{i=2}^{k} \frac{v + y_i}{w + y_i} \prod_{j=2}^\ell \frac{w - x_j}{v - x_j}.
\end{equation}
Let $\wh f_{ij}(\xi,w;\zeta,v)$ denote the integrand of $K_{ij}^{\bfx,\bfy}$ as defined in Theorem~\ref{thm:AC_conj}, with the term $\frac{e^{v^3/3 - v\zeta}}{e^{w^3/3 - w\xi}}  \frac{1}{v - w}$ replaced by $h(\xi,w;\zeta,v)$, and with $k = \ell = 1$.
\begin{prop}\label{P:gBR2}
	For all $s \in \R$, the following convergence holds:
	\begin{equation}
		\lim_{N \to \infty} \Pb \left(\wh L^{\bfx, \bfy}_{N,\tau} \leq 4N + (16N)^{1/3} s \right)
		= \det(\Id-P_s\wh K^{\bfx-\tau, \bfy+\tau}P_s)_{L^2(\R)},
	\end{equation}
	where $\wh K^{\bfx, \bfy}$ is a $2\times 2$ matrix kernel with entries
	\begin{equation}
		\begin{aligned}
			\wh K^{\bfx, \bfy}_{11}(\xi,\zeta)&:=\frac{1}{\left(2\pi \I\right)^2}\int_{\;\leftcontour{-\bfy}{}}\dx w   \int_{\;\rightcontour{w}{\bfx_{2:\ell}}}\dx v\wh f_{11}(\xi,w;\zeta,v) \\
			\wh K^{\bfx, \bfy}_{22}(\xi,\zeta)&:=\frac{1}{\left(2\pi \I\right)^2} \int_{\;\rightcontour{}{\bfx}}\dx v  \int_{\;\leftcontour{-\bfy_{2:k}}{ v}}\dx w\wh  f_{22}(\xi,w;\zeta,v) \\
			\wh K^{\bfx, \bfy}_{21}(\xi,\zeta)&:=\frac{1}{\left(2\pi \I\right)^2} \int_{\;\rightcontour{}{\bfx_{2:\ell} }}\dx v  \int_{\;\leftcontour{-\bfy_{2:k}}{ v}}\dx w\wh  f_{21}(\xi,w;\zeta,v) \\
			\wh K^{\bfx, \bfy}_{12}(\xi,\zeta)&:=\frac{\Id_{\xi>\zeta}}{\left(2\pi \I\right)^2}\int_{\;\rightcontour{}{-\bfy,\bfx}}\dx v \bigg[ \int_{\;\leftcontour{}{-\bfy, \bfx,v}}\dx w\wh  f_{12}(\xi,w;\zeta,v)+\oint_{ \Gamma_{-\bfy}} \dx w \wh f_{12}(\xi,w;\zeta,v)\bigg]\\
			&\quad+\frac{\Id_{\xi\leq \zeta}}{\left(2\pi \I\right)^2}
			\int_{\leftcontour{-\bfy,\bfx}{}}\dx w\bigg[\int_{\rightcontour {-\bfy,\bfx,w}{}}\dx v\wh  f_{12}(\xi,w;\zeta,v)- \oint_{\Gamma_{\bfx}} \dx v\wh  f_{12}(\xi,w;\zeta,v)\bigg].
		\end{aligned}
	\end{equation}
\end{prop}
\begin{proof}	
	The proof follows the same steps as that of Theorem~\ref{thm:AC_conj}, once Proposition~\ref{p:sta} is established.
	Therefore, we omit the details.In particular, observe that under our assumption
	$x_i+y_j>0$ for all $i\in\llbracket\ell\rrbracket$ and
	$j\in\llbracket k\rrbracket$, the contour associated with
	each entry is well-defined.
\end{proof}

\section{Proof of the finite-$N$ formula}\label{SectFiniteN}
In this section, we prove Theorem~\ref{thm:expKernel}. To derive the distribution function of the last passage time $\lpp_{m,n}$ under the setting of~\eqref{E:ThickBd}, we relate it to a Schur process. For this purpose it is more convenient to shift the picture to have the bottom-left point to be at $(1,1)$ instead of $(-\ell+1,-k+1)$, and place the thick boundaries on the north and east sides. More precisely, for fixed $m,n \in \mathbb{Z}_{\geq 1}$, let $(\omega_{i,j})_{1 \leq i \leq m+\ell,\ 1 \leq j \leq n+k}$ be a family of independent random variables distributed as
	\begin{equation} \label{E:Exp_Lpp2}
		\omega_{i,j} =
		\begin{cases}
			\Exp\left(\frac{1}{2} + \alpha_{m+\ell+1-i}\right), & \text{if } m+1 \leq i \leq m+\ell,\ 1 \leq j \leq n, \\
			\Exp\left(\frac{1}{2} + \beta_{n+k+1-j}\right), & \text{if } 1 \leq i \leq m,\ n+1 \leq j \leq n+k, \\
			\Exp(1), & \text{if } 1 \leq i \leq m,\ 1 \leq j \leq n, \\
			0, & \text{if } m+1 \leq i \leq m+\ell,\ n+1 \leq j \leq n+k,
		\end{cases}
	\end{equation}
	where $\bma$ and $\bmb$ are as in~\eqref{E:ThickBd}.

Then $\lpp_{m,n}$ under the setting of~\eqref{E:ThickBd} coincides with the LPP from $ (1,1)$ to $ (m+\ell, n+k)$ under the setting of~\eqref{E:Exp_Lpp2}, which we continue to denote by $\lpp_{m+\ell, n+k}$. This LPP is a standard limit of a geometric LPP, which is described by a Schur process with a well-known correlation kernel.

\subsection{Geometric LPP}
The geometric LPP is given as follows. Let $m,n \in \Z_{>0}$ be fixed. Choose vectors $ \boldsymbol{a} \in \R_{>0}^{\ell}$, $\boldsymbol{b} \in \R_{>0}^{k}$, and a parameter $q \in (0,1)$ such that
	\begin{equation}\label{asmp: glpp}
		0 < \sqrt{q}\, \boldsymbol{a},\, \sqrt{q}\, \boldsymbol{b} < 1.
	\end{equation}
	For $\gamma \in [0,1]$, we denote by $\mathrm{Geom}[\gamma]$ the geometric distribution with parameter $\gamma$. That is, if $X \sim \mathrm{Geom}[\gamma]$, then
	\begin{equation}
		\Pb(X = k) = (1 - \gamma) \gamma^k \quad \text{ for } k \in \mathbb{Z}_{\geq 0}.
	\end{equation}
For $ (i,j)\in\llbracket m+\ell\rrbracket\times\llbracket n+k\rrbracket$, the weights for the geometric LPP are then defined independently with the following marginals (see also Figure~\ref{F:glpp}):
	\begin{figure}
		\centering
		\includegraphics[height=7cm]{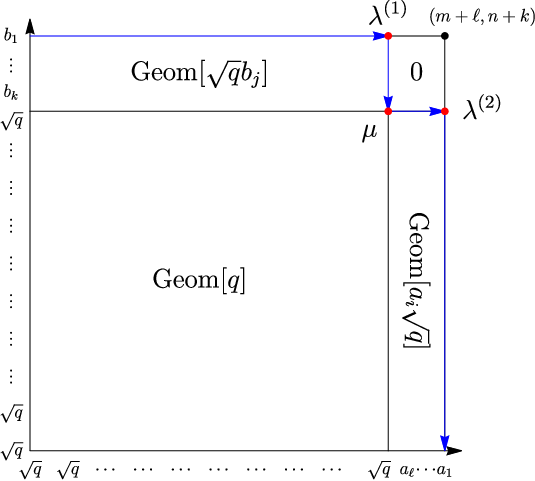}
		\caption{Illustration for geometric LPP. The weight parameters are the product of the numbers noted along the axis, except for the ones which are equal to zero. The latter random variables are however irrelevant for the values of the LPP at $(m,n+k)$ and $(m+\ell,n)$, related to the partitions $\lambda^{(1)}$ and $\lambda^{(2)}$, see \eqref{eq3.10}.}
\label{F:glpp}
	\end{figure}
	\begin{align}\label{E:PrelimitGeo}
		\omega_{(i, j)} =
		\begin{cases}
			\Geom[q], & \text{if } 1\leq i\leq m,\ 1\leq j\leq n, \\
			\Geom[\sqrt{q} a_{m+\ell+1-i}], & \text{if } m+1\leq i\leq m+\ell ,\ 1\leq j \le n, \\
			\Geom[\sqrt{q} b_{n+k+1-j}], & \text{if } 1\leq i \leq m,\ n+1\leq j\leq n+k, \\
			0, & \text{if } m+1\leq i\leq m+\ell,\ n+1\leq j\leq n+k.
		\end{cases}
	\end{align}
	Let $L_{(m,n)}^{\mathrm{geo}}$ denote the geometric last passage time from $ (1,1)$ to $ (m,n)$. Define
	\begin{equation}\label{E:Q}
		Q(x,w;\, y,v) := \left(\frac{1 - \sqrt{q}\, w}{1 - \sqrt{q}\, v} \right)^m \left(\frac{1 - \sqrt{q}/v}{1 - \sqrt{q}/w} \right)^n \frac{w^{x-1}}{v^y}.
	\end{equation}
	We start by giving an explicit formula for the distribution function of $L_{(m,n)}^{\mathrm{geo}}$.	
	\begin{prop}\label{P:geo}
		For the geometric LPP defined in~\eqref{E:PrelimitGeo} with $0<\bfa,\bfb<1$,
		\begin{equation}
			\Pb\left(L_{(m+\ell,n+k)}^{\mathrm{geo}}\leq s\right)=\det\left(\Id-P_sK_{m,n}^{{\rm geo}}P_{s}\right)_{\ell^2(\Z)}
		\end{equation}
where $P_s=\Id_{x>s}$ and $K_{m,n}^{{\rm geo}}$ is a $2 \times 2$ matrix kernel whose entries are given by
		\begin{align}
				K_{m,n,11}^{{\rm geo}}(x,y)&:=\frac{-1}{\left(2\pi\I\right)^2} \oint_{\Gamma_{\sqrt{q},\bfb}}\dx w \oint_{\Gamma_{1/\sqrt{q}}} \dx v  \frac{Q(x,w;y,v)}{v-w} \prod_{j=1}^k \frac{1-b_j/v}{1-b_j/w} ,\nonumber \\
				K_{m,n,22}^{{\rm geo}}(x,y)&:=\frac{-1}{\left(2\pi\I\right)^2} \oint_{\Gamma_{\sqrt{q}}}\dx w \oint_{\Gamma_{1/\sqrt{q},1/\bfa}} \dx v\frac{Q(x,w;y,v)}{v-w} \prod_{j=1}^\ell \frac{1-a_jw}{1-a_jv} ,\nonumber \\
				K_{m,n,21}^{{\rm geo}}(x,y)&:=\frac{-1}{\left(2\pi\I\right)^2} \oint_{\Gamma_{\sqrt{q}}}\dx w \oint_{\Gamma_{1/\sqrt{q} }}\dx v\frac{Q(x,w;y,v)}{v-w} \prod_{i=1}^\ell\left(1-a_iw\right) \prod_{j=1}^k \left(1-b_j/v\right),\label{E:geoKernel}\\
				K_{m,n,12}^{{\rm geo}}(x,y)&:=\frac{-1}{\left(2\pi\I\right)^2} \oint_{\Gamma_{\sqrt{q},\bfb}}\dx w \oint_{\Gamma_{1/\sqrt{q},1/\bfa}} \dx v\frac{Q(x,w;y,v)}{v-w}\frac{1}{\prod_{i=1}^\ell\left(1-a_iv\right) \prod_{j=1}^k \left(1-b_j/w\right)} \nonumber\\
&\quad+R_{m,n,12}^{{\rm geo}}(x,y). \nonumber
		\end{align}
		with
		\begin{equation}
			R_{m,n,12}^{{\rm geo}}(x,y):=\begin{cases}
				-\frac{1}{2\pi \I}\oint_{\Gamma_{\bfb}}\dx v \frac{ v^{x-1-y}}{\prod_{i=1}^\ell\left(1-a_iv\right) \prod_{j=1}^k \left(1-b_j/v\right)},\quad&{\rm if}\ x>y,\\
				\frac{1}{2\pi \I}\oint_{\Gamma_{1/\bfa}}\dx v \frac{ v^{x-1-y}}{\prod_{i=1}^\ell\left(1-a_iv\right) \prod_{j=1}^k \left(1-b_j/v\right)},\quad&{\rm if}\ x\leq y.\\
			\end{cases}
		\end{equation}	
	\end{prop}
	\begin{proof}
		Since $\omega_{i,j}=0$ for any $i> m$ and $j> n$, we have
		\begin{equation}
			L_{(m+\ell,n+k)}^{\mathrm{geo}}=\max\left\{L_{(m+\ell,n)}^{\mathrm{geo}},L_{(m,n+k)}^{\mathrm{geo}}\right\}.
		\end{equation}
		Hence, it suffices to compute the distribution function of the right-hand side. We achieve this by exploiting the connection to the Schur process~\cite{Ok01,OR01}. The relationship between geometric last passage percolation and the Schur process was first observed in~\cite{Jo00b}; see also~\cite{Jo03b} for further developments.
		
The LPP to the points $ (m, n+k)$, $ (m,n)$, and $ (m+\ell, n)$ are the first row of the partitions denoted by $\lambda^{(1)}$, $\mu$, and $\lambda^{(2)}$, whose joint distribution is given by the Schur process
\begin{equation}\label{eq3.10}
\varnothing \subset \lambda^{(1)}\supset \mu \subset \lambda^{(2)}\supset \varnothing
\end{equation}
with weight
\begin{equation}\label{E:schur}
s_{\lambda^{(1)}}(\overbrace{\sqrt{q},\ldots,\sqrt{q}}^{m\textrm{ times}}) s_{\lambda^{(1)}/\mu}(b_1,\ldots,b_k) s_{\lambda^{(2)}/\mu}(a_1,\ldots,a_{\ell}) s_{\lambda^{(2)}}(\overbrace{\sqrt{q},\ldots,\sqrt{q}}^{n\textrm{ times}}).
\end{equation}

		The Schur process $(\lambda^{(1)},\ \lambda^{(2)})$ is determinantal point process satisfying \cite{OR01} (see also~\cite{Jo03b} or Theorem~2.2 in~\cite{RB04} for the explicit expression below)
		\begin{equation}
			\Pb\left(L_{m,n+k}^{{\rm geo}}\leq s_1,L_{m+\ell,n}^{{\rm geo}}\leq s_2\right)=\det\left(\Id-\begin{pmatrix}
				P_{s_1}K_{m,n,11}^{{\rm geo}}P_{s_1} &P_{s_1}K_{m,n,12}^{{\rm geo}}P_{s_2}\\
				P_{s_2}K_{m,n,21}^{{\rm geo}}P_{s_1} &P_{s_2}K_{m,n,22}^{{\rm geo}}P_{s_2}\\
			\end{pmatrix}\right)_{L^2(\Z)}
		\end{equation}
		with
		\begin{equation}
			\begin{aligned}
				&K_{m,n,11}^{{\rm geo}}(x,y):=\frac{1}{\left(2\pi\I\right)^2} \oint_{}\oint_{\mathcal C_1} \dx v\dx w \frac{ Q(x,w;y,v)}{v-w} \prod_{j=1}^k\left(\frac{1-b_j/v}{1-b_j/w}\right),\\
				&K_{m,n,22}^{{\rm geo}}(x,y):=\frac{1}{\left(2\pi\I\right)^2} \oint_{}\oint_{\mathcal C_1} \dx v\dx w  \frac{Q(x,w;y,v)}{v-w} \prod_{j=1}^\ell\left(\frac{1-a_jw}{1-a_jv}\right),\\
				&K_{m,n,21}^{{\rm geo}}(x,y):=\frac{1}{\left(2\pi\I\right)^2} \oint_{}\oint_{\mathcal C_1} \dx v\dx w \frac{Q(x,w;y,v)}{v-w} \prod_{i=1}^\ell\left(1-a_iw\right) \prod_{j=1}^k \left(1-b_j/v\right),\\
				&K_{m,n,12}^{{\rm geo}}(x,y):=\frac{1}{\left(2\pi\I\right)^2}\oint_{}\oint_{\mathcal C_2} \dx v\dx w  \frac{Q(x,w;y,v)}{v-w}\frac{1}{\prod_{i=1}^\ell\left(1-a_iv\right) \prod_{j=1}^k \left(1-b_j/w\right)},
			\end{aligned}
		\end{equation}
        where the contours $\mathcal C_1$ and $\mathcal C_2$ are given in Figure~\ref{F:contourC}.
		
		Applying Cauchy's residue theorem on $K_{m,n,12}^{{\rm geo}}(x,y)$, we can deform the contour $\mathcal C_2$ to $\mathcal C_1$ leading to the extra term from the residue at $v=w$, namely
		\begin{equation}
			\begin{aligned}
				K_{m,n,12}^{{\rm geo}}(x,y)=&\frac{1}{\left(2\pi\I\right)^2}\oint_{}\oint_{\mathcal C_1} \dx v\dx w  \frac{Q(x,w;y,v)}{v-w}\frac{1}{\prod_{i=1}^\ell\left(1-a_iv\right) \prod_{j=1}^k \left(1-b_j/w\right)}\\
				&-\frac{1}{2\pi \I}\oint_{\Gamma_{0,\bfb}}\dx v\frac{1}{v^{y-x+1}}\frac{1}{\prod_{i=1}^\ell\left(1-a_iv\right) \prod_{j=1}^k \left(1-b_j/v\right)}.
			\end{aligned}
		\end{equation}
		The claimed representations are obtained by deforming these contours appropriately. For instance, for $K_{m,n,11}^{{\rm geo}}$, we can deform the contour for $v$ to include $1/\sqrt{q}$ since $v=\infty$ is not a pole.	
		\begin{figure}
			\centering
			\begin{subfigure}
				\centering
				\includegraphics[scale=0.8]
				{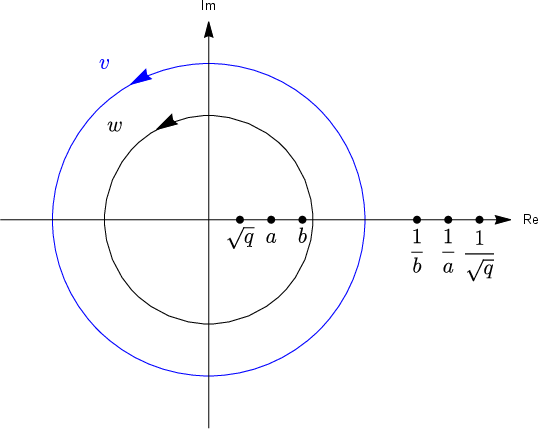}
			\end{subfigure}
			\begin{subfigure}
				\centering
				\includegraphics[scale=0.8]
				{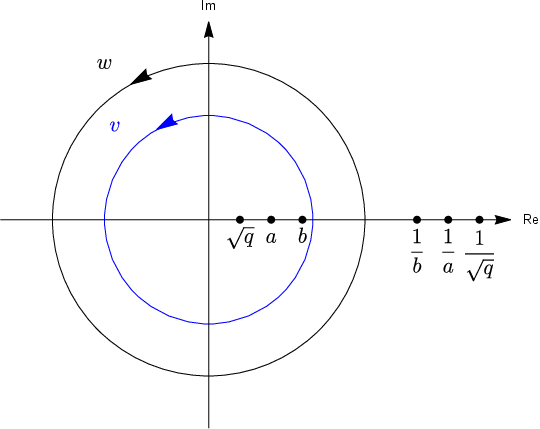}
			\end{subfigure}
			\caption{The contours $\mathcal{C}_1$ (left) and $\mathcal{C}_2$ (right) are depicted. Each contour encloses the points $\bfa$, $\bfb$, and $\sqrt{q}$, while excluding $\tfrac{1}{\bfa}$, $\tfrac{1}{\bfb}$, and $\tfrac{1}{\sqrt{q}}$. In $\mathcal{C}_1$, the $v$-contour encloses the $w$-contour, whereas in $\mathcal{C}_2$, the nesting order is reversed. In particular, under the assumption $0<\bfa,\bfb<1$, we have $\tfrac{1}{\bfa} > 1 > \bfa$ and $\tfrac{1}{\bfb} > 1 > \bfb$. }\label{F:contourC}
		\end{figure}
	\end{proof}

\subsection{Limit to exponential LPP}		
	As previously mentioned, the exponential LPP defined in~\eqref{E:ThickBd} can be viewed as the scaling limit of a certain geometric LPP: for $\varepsilon > 0$, define $q := 1 - \varepsilon$ and
	\begin{equation}\label{E:geoRate}
		\begin{aligned}
			a_i &= 1 - \varepsilon \alpha_i, \quad \text{for all } i \in \{ 1, \ldots, \ell\}, \\
			b_j &= 1 - \varepsilon \beta_j, \quad \text{for all } j \in \{ 1, \ldots, k\},
		\end{aligned}
	\end{equation}
	where $\bma$ and $\bmb$ are given in Theorem~\ref{thm:expKernel}. Substituting this into~\eqref{E:PrelimitGeo} and denote $L_{(i,j)\to(m,n)}^{{\rm geo},\varepsilon}$ as the last passage time from $ (i,j)$ to $ (m,n)$, we have
	\begin{equation}\label{E:geo_to_exp}
		\lim_{\varepsilon\to 0}\varepsilon L_{(i,j)\to(m,n)}^{{\rm geo},\varepsilon}\stackrel{d}{=}\lpp_{(i,j)\to(m,n)}.
	\end{equation}

    As often is the case, the limit of the correlation kernel is well-defined only after considering a proper conjugation. Thus, we introduce the conjugation matrix $C=(C_{i,j})_{1\leq i,j\leq 2}$ with entries
    \begin{equation}
		C_{i,j} := \begin{cases}
			1, & \text{if } (i,j)=(1,1)\textrm{ or }(i,j)=(2,2), \\
			(-1)^{\ell} \varepsilon^{k + \ell }, & \text{if } (i, j) = (1,2), \\
			(-1)^{\ell} \varepsilon^{-(k + \ell)}, & \text{if } (i, j) = (2,1).
		\end{cases}
	\end{equation}
	The limiting kernel is given by
	\begin{equation} \label{E:kernel_geo_to_exp}
		\kexp_{m,n,ij}(x, y) := \lim_{\varepsilon \to 0}\frac1\varepsilon C_{i,j}  \kgeo_{m,n,ij}(x/\varepsilon, y/\varepsilon).
	\end{equation}
An elementary computation gives the following limits.
	\begin{lem}\label{lem:limit}
		For any fixed $x,y\in\R$ and $w,v\in\C$, we have
		\begin{equation}
			\begin{aligned}
				&\lim_{\varepsilon\to0}\left(1-(1-\varepsilon \beta)(1+\varepsilon x)\right)/\varepsilon=\beta-x,\\
				&\lim_{\varepsilon\to0}\left(1-(1-\varepsilon \beta)/(1+\varepsilon x)\right)/\varepsilon=\beta+x,\\
				&\lim_{\varepsilon\to0} Q(x/\varepsilon,1+\varepsilon w;y/\varepsilon,1+\varepsilon v)=E(x,w;y,v),
			\end{aligned}
		\end{equation}
where $E$ is defined in~\eqref{E:P}.
	\end{lem}
	As previously discussed, the distribution of the last passage time from $(-\ell+1,-k+1)$ to $ (m,n)$ under the setting of~\eqref{E:ThickBd} coincides with that from $ (1,1)$ to $(m+\ell,n+k)$ under the setting of~\eqref{E:Exp_Lpp2}. Taking the limit as described in~\eqref{E:kernel_geo_to_exp} we obtain the following result.
	\begin{thm}\label{thm:hatExp}
		Under the Assumption~\ref{asmp: glpp}, and that $\bma, \bmb>0$,
		\begin{equation}
		\label{E:2}
			\Pb\left(\lpp_{m,n}\leq s\right)=\det\left(\Id-P_{s}K_{m,n}^{\exp}P_{s} \right)_{L^2(\R)},
		\end{equation}
		where $P_s(x)=\Id_{x>s}$ and $K_{m,n}^{\exp}$ is given in Theorem~\ref{thm:expKernel}.
	\end{thm}
	\begin{proof}
		Applying Lemma~\ref{lem:limit},~\eqref{E:kernel_geo_to_exp},~\eqref{E:geoRate} and change of variables $w\mapsto 1+\varepsilon w$, $v\mapsto 1+\varepsilon v$ on~\eqref{E:geoKernel}, we obtain that the kernel converges to $K_{m,n}^{\exp}$.

The only entry which is not straightforward of the claimed form is $K_{m,n,12}^{\exp}$. For this we first get
\begin{equation}
\label{E:1}
\begin{aligned}
K_{m,n,12}^{\exp}(x,y)&:=\frac{-1}{\left(2\pi\I\right)^2}\oint_{\Gamma_{-1/2,-\bmb}}\dx w\oint_{\Gamma_{1/2,\bma}}\dx v\frac{E(x,w;y,v)}{v-w}\frac{1}{\prod_{i=1}^{k}\left(w+\beta_i\right)\prod_{j=1}^{\ell}\left(v-\alpha_j\right)}\\
&\quad+R_{m,n}^{\exp}(x,y)
			\end{aligned}
		\end{equation}
		with
		\begin{equation}
			R_{m,n}^{\exp}(x,y):=\begin{cases}
				\frac{-1}{2\pi\I}\oint_{\Gamma_{-\bmb}}\dx w \frac{e^{w(x-y)}}{\prod_{i=1}^{k}\left(w+\beta_i\right)\prod_{j=1}^{\ell}\left(w-\alpha_i\right)},\quad&{\rm if}\ x>y,\\
				\frac{1}{2\pi\I}\oint_{\Gamma_{\bma}}\dx w \frac{e^{w(x-y)}}{\prod_{i=1}^{k}\left(w+\beta_i\right)\prod_{j=1}^{\ell}\left(w-\alpha_j\right)},\quad&{\rm if}\ x\leq y.
			\end{cases}
		\end{equation}
By deforming the contours in \eqref{E:1} so that either $w$ includes $v$ (for $x\leq y$) or $v$ includes $w$ (for $x>y$) the term $R_{m,n}^{\exp}$ is cancelled. This new representation is useful when considering later non-positive $\bmb,\bma$.

For the convergence of the Fredholm determinant, one just needs to add some a-priori bound on the kernel entries implying that we can take the $\varepsilon\to 0$ limit into the Fredholm series expansion. We refrain to give the details as similar estimates can be found for instance in Appendix~C.3 of~\cite{BFO19}.
	\end{proof}	
 	\begin{rem}
 		The restriction $\bma, \bmb > 0$ is essential in the proof of Theorem~\ref{thm:hatExp} because the application of Proposition~\ref{P:geo} requires that the parameters $\bfa$ and $\bfb$ defined in~\eqref{E:geoRate} satisfy $0 < \bfa, \bfb < 1$. Also, the expression for $K^{\exp}_{m,n,12}$ without the $R^{\exp}_{m,n}$ term is needed for the analytic continuation in the next section to the negative values of $\bfa,\bfb$.
 	\end{rem}

\subsection{Analytic continuation}

 	Our next goal is to prove Theorem~\ref{thm:expKernel} by extending the result of Theorem~\ref{thm:hatExp} to the setting where $\min\{{\bma}, {\bmb}\} \in (-1/2,0]$.
		
	By Proposition 5.1 of~\cite{FS05a}, the distribution function on the left-hand side of \eqref{E:2} is analytic in the parameters $\bma>-1/2$ and $\bmb>-1/2$. This is not surprising since for this range of parameters, the LPP model is well-defined. Therefore, extending the admissible range of $\bma$ and $ \bmb$  reduces to establishing the corresponding analytic continuation for the right-hand side of \eqref{E:2}.

	\begin{proof}[Proof of Theorem~\ref{thm:expKernel}]
	We need to show that for $i, j \in \{1, 2\}$ and $ \xi, \zeta \in \R $, the function $ K_{m,n,ij}^{\exp}(\xi,\zeta) $ defined in Theorem~\ref{thm:expKernel} is analytic with respect to $ \bma, \bmb>-\tfrac12$.\medskip
	
	\textbf{For $\kexp_{m,n,11}$ and $\kexp_{m,n,22}$.}	Since the analytic continuations of $\kexp_{m,n,11}$ and $\kexp_{m,n,22}$ proceed via essentially the same argument, we restrict our attention to the case of $\kexp_{m,n,11}$. We deform the contour $\Gamma_{-1/2, -\bmb}$ for the variable $w$ to $B_\delta$, where
	\begin{equation}
		\delta:=\max\{1,2|\beta_i|\mid i\in\llbracket k\rrbracket\}.
	\end{equation}
	In doing so, we cross a simple pole at $w = v$, whose residue is given by
	\begin{equation} \label{E:zeroResidue}
		\frac{1}{2\pi \I} \oint_{\Gamma_{1/2}} \dx v\, e^{xv - yv} = 0.
	\end{equation}
	Thus, by Cauchy's residue theorem, we obtain
	\begin{equation}\label{E:k11_ball}
		K_{m,n,11}^{\exp}(x,y)= \frac{-1}{(2\pi \I)^2} \oint_{B_\delta} \dx w \oint_{\Gamma_{1/2}} \dx v \, \frac{E(x,w; y,v)}{v-w} \prod_{i=1}^{\ell} \frac{v + \beta_j}{w + \beta_j}.
	\end{equation}
	Note that the right-hand side is manifestly analytic for any $\bmb \in (-1/2,0]$, since all poles lie strictly inside the contour $B_\delta$ (see also Lemma~B.2 of~\cite{BCFV14} for a general analysis statement of analyticity of an integral over a contour in the complex plane).\medskip

	\textbf{For $K_{m,n,21}^{\exp}$.} This case is straightforward as no contour deformation is required.\medskip

	\textbf{For $\kexp_{m,n,12}$.} Since the analytic continuations for $x > y$ and $x \leq y$ follow essentially the same argument, we focus on the case $x>y$. We deform the contour $\Gamma_{1/2, \bma}$ for the variable $v$ to $B_{3\delta_a}$, where
	\begin{equation}
		\delta_a:=\max\{1,|\alpha_i|,|\beta_j|\mid i\in\llbracket \ell\rrbracket,j\in\llbracket k\rrbracket\}.
	\end{equation}
	Next, we deform the contour $\Gamma_{-1/2, -\bmb}$ for the variable $w$ to $B_{2\delta_a}$. Both deformation do not introduce any additional poles. As a result, we obtain
	\begin{equation}\label{E:k12_analytic}
		K_{m,n,12}^{\exp}(x, y) = \frac{-1}{(2\pi \I)^2} \oint_{B_{2\delta_a}} \dx w\oint_{B_{3\delta_a}} \dx v  \, \frac{E(x, w; y, v)}{ \prod_{j=1}^{k} (w + \beta_j) \prod_{i=1}^{\ell} (v - \alpha_i)}\frac{1}{v-w}.
	\end{equation}
	The right-hand side is analytic in the parameters $\bma$ and $\bmb$, even when $\bma, \bmb \in (-1/2, 0]$. 	
\end{proof}
\begin{proof}[Proof of Proposition~\ref{p:sta}]
	Since the proof closely mirrors that of Theorem~\ref{thm:expKernel}, we only highlight the key differences here. Let $\wh L^{{\rm geo}}_{m,n}$ denote the geometric last passage time from $(1,1)$ to $(m,n)$ corresponding to the model~\eqref{E:StThickBd}. In contrast to~\eqref{E:schur}, the Schur process corresponding to this geometric LPP has the following weight:
	\begin{equation}
		s_{\lambda^{(1)}}(\overbrace{\sqrt{q},\ldots,\sqrt{q}}^{m\textrm{ times}},a_2,\ldots,a_{\ell})\, s_{\lambda^{(1)}/\mu}(b_1)\, s_{\lambda^{(2)}/\mu}(a_1)\, s_{\lambda^{(2)}}(\overbrace{\sqrt{q},\ldots,\sqrt{q}}^{n\textrm{ times}},b_2,\ldots,b_k).
	\end{equation}
	Accordingly, we define, in place of~\eqref{E:Q},
	\begin{equation}
		\wh Q(x,w;\, y,v) := \left(\frac{1 - \sqrt{q}\, w}{1 - \sqrt{q}\, v} \right)^m\left(\frac{1 - \sqrt{q}/v}{1 - \sqrt{q}/w} \right)^n\prod_{i=2}^\ell \left(\frac{1 - a_i\, w}{1 - a_i\, v} \right)\prod_{j=2}^k \left(\frac{1 - b_j/v}{1 - b_j/w} \right)\frac{w^{x-1}}{v^y}.
	\end{equation}
	Then the kernel for $\wh L^{{\rm geo}}_{m,n}$ is obtained from $K^{{\rm geo}}_{m,n}$ by replacing $Q(x,w;\, y,v)$ with $\wh Q(x,w;\, y,v)$ and taking $k=\ell=1$.
	The remainder of the proof proceeds exactly as in Theorem~\ref{thm:expKernel}.
\end{proof}

\section{Large-$N$ asymptotics}\label{sectAsymptotics}
In this section, we prove Theorem~\ref{thm:AC_conj}. Recall that we consider the end-point to be
\begin{equation}
(m,n)=(N - \tau(2N)^{2/3},\, N + \tau(2N)^{2/3})
\end{equation}
 and we also scale
		\begin{equation}
				\bma = \tilde\bma + (16N)^{-1/3} \bfx, \quad
				\bmb = \tilde\bmb + (16N)^{-1/3} \bfy,
		\end{equation}
where $\tilde\bma,\tilde\bmb\geq 0$.

	To obtain the limiting distribution
	\begin{equation}
		\lim_{N \to \infty} \Pb\left(L^{\tilde\bma, \tilde\bmb, \bfx, \bfy}_{N,\tau} \leq 4N + (16N)^{1/3} s \right)
	\end{equation}
   we need to show that the rescaled (and conjugated) kernel converges to the claimed limit kernel uniformly on compact sets, and then have a-priori bounds to extend the convergence to the level of Fredholm determinants.

	The rescaled kernel is given by
	\begin{equation}
		\kexpresc_{N,ij}(\xi,\zeta) :=  (16N)^{1/3} \kexp_{N - \tau(2N)^{2/3},\, N + \tau(2N)^{2/3},\, ij}(4N + (16N)^{1/3} \xi,4N + (16N)^{1/3} \zeta),
	\end{equation}
	Since the correlation kernel in Theorem~\ref{thm:expKernel} does not depend on the order of the parameters $\alpha_i,\beta_j$, neither does the limiting distribution. Therefore, without loss of generality, we assume that $\tilde\alpha_i = 0$ for all $i \in \llbracket \ell_0 \rrbracket$ for some $1 \leq \ell_0 \leq \ell$, and $\tilde\beta_j = 0$ for all $j \in \llbracket k_0 \rrbracket$ for some $1 \leq k_0 \leq k$. This means that $\alpha_i>0$ for all $i>\ell_0$, and $\beta_j>0$ for all $j>k_0$.
	
	\begin{lem}\label{lem:pt_cvg}
		For $i, j \in \{1, 2\}$ and uniformly for $\xi, \zeta$ in any compact set, we have
		\begin{equation}
			\lim_{N \to \infty} G_{N,ij} \, \kexpresc_{N,ij}(\xi, \zeta)
			= K_{ij}^{\bfx_{\llbracket \ell_0\rrbracket}+\tau,\bfy_{\llbracket k_0\rrbracket}-\tau}(\xi+\tau^2, \zeta+\tau^2),
		\end{equation}
		where the right-hand side is given in Definition~\ref{def:gBR} and $G_{N, ij}$ is a conjugation factor defined below.
	\end{lem}

We introduce some notations and contours for the proof of Lemma~\ref{lem:pt_cvg}. Define
	\begin{equation}
		\begin{aligned}
			&\bfr := \max\left\{ |x_i|,\ |y_j| \mid i \in\llbracket\ell_0\rrbracket,\ j \in \llbracket k_0\rrbracket \right\},\\
			&\bfR:=\max\left\{1, |\alpha_i|,\ |\beta_i|,\mid i \in\llbracket\ell\rrbracket,\ j \in \llbracket k\rrbracket \right\}.
		\end{aligned}
	\end{equation}
 We define a conjugation function as
	\begin{equation}\label{E:conjugation_for_limit}
	\begin{aligned}
		G_{N,ij} &:=
		\begin{cases}
			e^{-\tau(\xi-\zeta)}, & \text{if } (i,j)=(1,1)\textrm{ or }(i,j)=(2,2), \\
			e^{-\tau(\xi-\zeta)}\Phi_{N,\tilde\bma,\tilde\bmb}^{-1}	, & \text{if } (i,j) = (2,1), \\
			e^{-\tau(\xi-\zeta)}\Phi_{N,\tilde\bma,\tilde\bmb}	, & \text{if } (i,j) = (1,2).
		\end{cases} \\
		&= e^{-\tau (\xi-\zeta)} \Phi_{N,\tilde\bma,\tilde\bmb}^{j-i}
	\end{aligned}	
	\end{equation}
	with
	\begin{equation}
		\Phi_{N,\tilde\bma,\tilde\bmb}:=(16 N)^{-\frac{\ell_0+k_0}{3}} \prod_{i=\ell_0+1}^\ell (- \tilde\alpha_i) \prod_{j=k_0+1}^k \tilde\beta_j.
	\end{equation}
 We also define the contour (see Figure~\ref{F:deform_c12}),
	\begin{equation}
		\gamma := \mathcal{D}_1 \cup \mathcal{D}_2 \cup \bar{\mathcal{D}}_2 \cup \mathcal{D}_3
	\end{equation}
	where
	\begin{equation}\label{E:contourD}
		\begin{aligned}
			\mathcal{D}_1 &:= \left\{ 2 \bfr (16N)^{-1/3} e^{i\theta} \mid \theta \in [-2\pi/3, 2\pi/3] \right\}, \\
			\mathcal{D}_2 &:= \left\{ 2 \bfr (16N)^{-1/3} e^{2\pi\I/3}(1 - t) +2\bfR e^{2\pi\I/3} t \mid t \in [0,1] \right\}, \\
			\mathcal{D}_3 &:= \left\{ -\bfR + \sqrt{3}\bfR e^{i\theta} \mid \theta \in [\pi/2, 3\pi/2] \right\}.
		\end{aligned}
	\end{equation}
Similarly, we define the contour (see Figure~\ref{F:deform_c12}),
	\begin{equation}
		\Gamma := \mathcal{C}_1 \cup \mathcal{C}_2 \cup \bar{\mathcal{C}}_2 \cup \mathcal{C}_3
	\end{equation}
	where
	\begin{equation}\label{E:contourC}
		\begin{aligned}
			\mathcal{C}_1 &:= \left\{ 3 \bfr (16N)^{-1/3} e^{i\theta} \mid \theta \in [-\pi/3, \pi/3] \right\}, \\
			\mathcal{C}_2 &:= \left\{ 3 \bfr (16N)^{-1/3} e^{\pi\I/3}(1 - t) + 2\bfR e^{\pi\I/3} t \mid t \in [0,1] \right\}, \\
			\mathcal{C}_3 &:= \left\{ \bfR + \sqrt{3}\bfR e^{i\theta} \mid \theta \in [-\pi/2, \pi/2] \right\}.
		\end{aligned}
	\end{equation}
	
	\begin{figure}
		\centering
		\includegraphics[height=6cm]
		{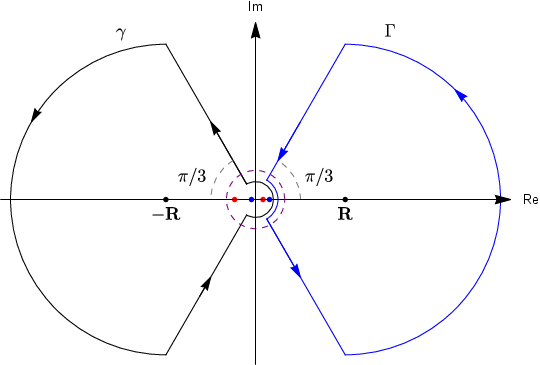}
		\caption{The blue (resp.\ red) points indicate the locations of $(16 N)^{-1/3} x_i$ for $i \in \llbracket \ell_0\rrbracket$ (resp.\ $(16 N)^{-1/3} y_j$ for $j \in \llbracket k_0\rrbracket$). Local modifications occur inside the dashed purple circle, whose radius is of order $\Or(N^{-1/3})$.}
		\label{F:deform_c12}
	\end{figure}

	\begin{proof}[Proof of Lemma~\ref{lem:pt_cvg}]
	We apply saddle-point analysis to derive the limiting kernel. The method is standard, and a similar scaling limit has been carried out in~\cite{BFP09}. See also Section 6.1~of~\cite{BF08} for a detailed explanation how to do the asymptotic analysis, which goes back to the strategy used in~\cite{GTW00}. Therefore, rather than presenting the full details of the saddle-point analysis, we highlight only the key ingredients. The main difference in our setting is the presence of extra finite products, some of them having poles. So when deforming the contours, one sometimes crosses these poles and needs to take care of the associated residues.
	
	Define
	\begin{equation}\label{E:xy}
			x:=4N+(16N)^{1/3}\xi,\quad
			y:=4N+(16N)^{1/3}\zeta.
	\end{equation} 	

		\textbf{For $\kexpresc_{N,11}$.} We deform the contour $\Gamma_{-1/2,-\bmb}$ for the variable $w$ in $\kexpresc_{N,11}$ to the contour $\gamma$. As we do not cross poles (for all $N$ large enough), we obtain
		\begin{equation}\label{E:sca11}
			\begin{aligned}
				\kexpresc_{N,11}(\xi,\zeta) =& -\frac{(16N)^{1/3} }{(2\pi\I)^2} \oint_{\gamma} \dx w \oint_{\Gamma_{1/2}} \dx v \,
				\frac{E(x, w; y, v)}{v - w}
				\prod_{i = 1}^{k} \frac{v + \beta_i}{w + \beta_i} .
			\end{aligned}
		\end{equation}
		 We deform the contour $\Gamma_{1/2}$ for the variable $v$ to $\Gamma$. Since this deformation does not cross any new poles, we obtain
		\begin{equation}\label{E:k11}
			\begin{aligned}
				\frac{-(16N)^{1/3} }{(2\pi\I)^2} \oint_{\gamma} \dx w \oint_{\Gamma} \dx v \,
			\frac{e^{4Nf(v) +\tau(2N)^{2/3}h(v)- (16N)^{1/3} v \zeta}}{e^{4Nf(w) +\tau(2N)^{2/3}h(w)- (16N)^{1/3} w \xi}}
			\frac{1}{v - w}
			\prod_{i = 1}^{k} \frac{v + \beta_i}{w + \beta_i}
			\end{aligned}
		\end{equation}
		with $f(v) := -v - \frac{1}{4} \log(1 - 2v) + \frac{1}{4} \log(1 + 2v)$ and $h(v):=\log(1-4v^2)$.
		
		We now proceed to identify steep descent paths for $f(v)$ and $-f(w)$. Observe that the function $f$ admits a degenerate critical point at $p_c = 0$, where both the first and second derivatives vanish, while the third derivative satisfies $f^{(3)}(p_c) = 8$. A Taylor expansion of $f$ and $h$ in a neighborhood of $p_c$ yields
		\begin{equation}
				f(v)=\frac{4v^3}{3}+\Or(v^4),\quad
				h(v)=-4v^2+\Or(v^4).
		\end{equation}
		
		It is easy to verify that $\gamma \setminus \mathcal{C}_1$ (resp.\ $\Gamma \setminus \mathcal{D}_1$) forms a descent path for the function $f(v)$ (resp.\ $-f(w)$)\footnote{In the sense that ${\rm Re}(f(v))$ is strictly decreasing along these curves, which can be rigorously verified by differentiating ${\rm Re}(f(v))$ with respect to the parameterizations of the respective curves. A similar computation appears in (4.44) and (4.45) of~\cite{BFP09}.}, see also the proof of Lemma~4.4 in~\cite{BFP09} for further details. As the path are steep descent, for any $\delta > 0$, the contribution of the double integral over $B_{\delta}^c \cap (\gamma \times \Gamma)$ is of order $\Or(e^{-\delta N})$.
Note that $\mathcal{C}_1$ and $\mathcal{D}_1$ are local modifications of order $\Or(N^{-1/3})$ to the steepest descent paths, so they do not affect the previous step.

		Also, we have
		\begin{equation}\label{Ec}
			\prod_{i=1}^{k}\frac{(16N)^{-1/3}v+\beta_i}{(16N)^{-1/3}w+\beta_i}=\prod_{i=1}^{k_0}\frac{v+y_i}{w+y_i}+\Or(N^{-1/3}).
		\end{equation}
Substituting $w\mapsto (16 N)^{-1/3}w$, $v\mapsto (16 N)^{-1/3}v$ in~\eqref{E:k11}, removing the higher order corrections in $f(v)$ and $h(v)$ and the error terms in \eqref{Ec} leads to an overall error terms of order $\Or(N^{-1/3})$. At this point we can extend the integrals to infinity with a correction only of order $\Or(e^{-\delta N})$.
Combining these steps we get
		\begin{equation}
			\begin{aligned}
				 &\eqref{E:k11}=\frac{1}{\left(2\pi\I\right)^2}\int_{\leftcontour{-\bfy_{\llbracket k_0\rrbracket}}{ }}\dx w \int_{\rightcontour{-\bfy_{\llbracket k_0\rrbracket},w}{}}\dx v\frac{e^{v^3/3-\tau v^2-v\zeta}}{e^{w^3/3-\tau w^2-w\xi}}\frac{ 1}{v-w}\prod_{i=1}^{k_0}\frac{v+y_i}{w+y_i}+\Or(N^{-1/3}),
			\end{aligned}
		\end{equation}
		where the two paths do not touch each other. Substituting $v\mapsto v+\tau$ and $w\mapsto  w+\tau$, taking $N$ to infinity, the right hand side is equal to
		\begin{equation}\label{E:fh}
			\frac{e^{\tau(\xi-\zeta)}}{\left(2\pi\I\right)^2}\int_{\leftcontour{-\tau-\bfy_{\llbracket k_0\rrbracket}}{}}\dx w \int_{\rightcontour{-\tau-\bfy_{\llbracket k_0\rrbracket},w}{}}\dx v\frac{e^{v^3/3 -v \left(\zeta +\tau ^2\right)}}{e^{w^3/3 -w \left(\xi +\tau ^2\right)}}\frac{ 1}{v-w}\prod_{i=1}^{k_0}\frac{v+y_i+\tau}{w+y_i+\tau},
		\end{equation}
 		Multiplying by the conjugation~\eqref{E:conjugation_for_limit} and taking $N\to\infty$ we obtain the limit kernel.

		\medskip

		\textbf{For $\kexpresc_{N,22}$.} The analysis for this kernel proceeds identically to that of $\kexpresc_{N,11}$ (if we exchange the roles of $(x,w)$ with $(y,v)$) and is therefore omitted. \medskip

		\textbf{For $\kexpresc_{N,21}$.} Note that for the factor inside the integral we have
		\begin{equation}
			\begin{aligned}
&(16N)^{\frac{k_0+\ell_0}{3}}\prod_{i=1}^{k}\left((16 N)^{-1/3}v+\beta_i\right)\prod_{j=1}^{\ell}\left((16 N)^{-1/3}w-\alpha_i\right) \\
				=& \prod_{i=k_0+1}^k\tilde\beta_i\prod_{j=\ell_0+1}^\ell(-\tilde\alpha_j)\prod_{i=1}^{k_0}\left(v+y_i\right)\prod_{j=1}^{\ell_0}\left(w-x_j\right)+\Or(N^{-1/3}).
			\end{aligned}
		\end{equation}
		After applying a similar saddle point analysis as above, we obtain
		\begin{equation}
			\begin{aligned}
				&   \lim_{N\to\infty}G_{N,21} \kexpresc_{N,21}(\xi,\zeta)\\
				=&\frac{1}{\left(2\pi\I\right)^2}\int_{\leftcontour{}{}} \dx w\int_{\rightcontour{w}{}}\dx v \frac{e^{v^3/3-v(\zeta+\tau^2)}}{e^{w^3/3-w(\xi+\tau^2)}}\frac{ 1}{v-w}\prod_{i=1}^{k_0}\left(v+y_i+\tau\right)\prod_{j=1}^{\ell_0}\left(w-x_j+\tau\right).
			\end{aligned}
		\end{equation}

\medskip
		\textbf{For $\kexpresc_{N,12}$.} We consider the case $\xi\leq \zeta$. The procedure for the case $\xi>\zeta$ is identical to this, and hence we omit the proof.
We denote the limit of the integrand by
\begin{equation}
	g_{12}(\xi,w;\zeta,v):=\frac{e^{v^3/3 - v\left(\zeta+\tau^2\right)}}{e^{w^3/3 - w\left(\xi+\tau^2\right)}}  \frac{1}{v-w} \frac{1}{\prod_{\substack{i=1}}^{k_0}\left(w+y_i+\tau\right)\prod_{\substack{  j=1}}^{\ell_0} \left(v-x_j+\tau\right)}.
\end{equation}

Using Cauchy's residue theorem, we have
		 \begin{equation}\label{E:decomp}
			\begin{aligned}
				\kexpresc_{N,12}(\xi,\zeta)=&-\frac{(16N)^{1/3}}{\left(2\pi\I\right)^2}\oint_{\gamma}\dx w\oint_{\Gamma}\dx v\frac{E(x, w; y, v)}{ \prod_{i=1}^{k} (w + \beta_i) \prod_{j=1}^{\ell} (v - {\alpha}_j) }\frac{1}{v-w}\\
				&+\frac{(16N)^{1/3}}{2\pi\I}\oint_{\Gamma}\dx v\frac{e^{v(x-y)}}{ \prod_{i=1}^{k} (v + \beta_i) \prod_{j=1}^{\ell} (v - {\alpha}_j) }\\
				&-\frac{(16N)^{1/3}}{\left(2\pi \I\right)^2}\oint_{\Gamma_{\bma_{\llbracket \ell_0\rrbracket}}}\dx v\oint_{\gamma}\dx w\frac{E(x, w; y, v)}{ \prod_{i=1}^{k} (w + \beta_i) \prod_{j=1}^{\ell} (v - {\alpha}_j)}\frac{1}{v-w},
			\end{aligned}
		\end{equation}
		where the contours for $v$ and $w$ are chosen so that they are disjoint. We now proceed to compute the limit of the terms appearing on the right-hand side of~\eqref{E:decomp}.
		
		\textit{First term in~\eqref{E:decomp}.} We denote by $K_{N,12}^1$ the first term in~\eqref{E:decomp}. Note that
		\begin{equation}
			\begin{aligned} &(16 N)^{-\frac{k_0+\ell_0}{3}}\frac{1}{\prod_{i=1}^{k}\left((16 N)^{-1/3}w+\beta_i\right)\prod_{j=1}^{\ell}\left((16 N)^{-1/3}v-\alpha_j\right)}\\
				=&\frac{1}{\prod_{i=1}^{k_0}\left(w+y_i\right)\prod_{j=1}^{\ell_0}\left(v-x_j\right)}\frac{1}{\prod_{i=k_0+1}^k\tilde\beta_i\prod_{j=\ell_0+1}^\ell\left(-\tilde\alpha_j\right)}+\Or(N^{-1/3})
			\end{aligned}
		\end{equation}
		Applying the same procedure as the one for $\kexpresc_{N,11}$, we have
		\begin{equation}
\begin{aligned}
&			\lim_{N\to\infty}G_{N,12}K_{N,12}^1(\xi,\zeta)\\
&=\frac{1}{\left(2\pi\I\right)^2}\int_{\leftcontour{-\tau+\bfx_{\llbracket\ell_0\rrbracket},-\tau-\bfy_{\llbracket k_0\rrbracket} }{}}\dx w\int_{\rightcontour{-\tau+\bfx_{\llbracket\ell_0\rrbracket},-\tau-\bfy_{\llbracket k_0\rrbracket},w}{ }}\dx v \,g_{12}(\xi,w;\zeta,v).
\end{aligned}
		\end{equation}
		
		\textit{Second term in~\eqref{E:decomp}.} We denote by $K_{N,12}^2$ the second term in~\eqref{E:decomp}. We claim that for any $\xi\leq \zeta$, it holds
		\begin{equation}\label{E:rmdZero}
			(16 N)^{-\frac{k_0+\ell_0}{3}}K_{N,12}^2(\xi,\zeta)=\Or(N^{-1/3}).
		\end{equation}
		Let $N_0 \in \Z_{>0}$. We define a modified contour $\Gamma(N_0)$, which coincides with $\Gamma$ except that in the definition of each sub-contour $\mathcal{C}_i$ in~\eqref{E:contourC}, the parameter $N$ is replaced by $N_0$. In particular, we can choose $N_0$ appropriately so that $\Gamma\left(N_0\right)$ encloses all points $\alpha_j$ for $j \geq \ell_0+1$, but do not include the points $\alpha_j:=(16 N)^{-1 / 3} x_j$ for all $j \in \llbracket \ell_0 \rrbracket$ and also not the poles $\beta_i$ for $i \in \llbracket k \rrbracket$ (since we assume $\tilde{\boldsymbol{\beta}} \geq 0$ ). This is possible for all $N$ large enough. This gives
		\begin{equation}
			K_{N,12}^2(\xi,\zeta)=\frac{(16N)^{1/3}}{2\pi\I}\oint_{\Gamma(N_0)}\dx v\frac{e^{v(\xi-\zeta)}}{ \prod_{i=1}^{k} (v + \beta_i) \prod_{j=1}^{\ell} (v - {\alpha}_j)}.
		\end{equation}
		For $v\in \Gamma(N_0)$, we have ${\rm Re}(v)>0$ and
		\begin{equation}
			\min\{|v + \beta_i|,|v - {\alpha}_j|\}\geq\varepsilon>0,
		\end{equation}
		for some $\varepsilon$ independent of $N$, $\xi$ and $\zeta$. Together with $\xi\leq \zeta$, we have
		\begin{equation}\label{E:uni}
			|K_{N,12}^2(\xi,\zeta)|\leq C(16N)^{1/3}\varepsilon^{-(k+\ell+2)}.
		\end{equation}
		Then~\eqref{E:rmdZero} follows from the assumption $\min\{k_0,\ell_0\}\geq 1$.
		
		\textit{Third term in~\eqref{E:decomp}.} We by denote $K_{N,12}^3$ the third term in~\eqref{E:decomp}. We claim that for any $\xi\leq \zeta$, it holds
		\begin{equation}\label{E:k12rm}
			\lim_{N\to\infty} G_{N,12}K_{N,12}^3(\xi,\zeta)=\frac{-1}{\left(2\pi\I\right)^2}\oint_{\Gamma_{-\tau+\bfx_{\llbracket\ell_0\rrbracket}}}\dx v\int_{\leftcontour{ -\tau-\bfy_{\llbracket k_0\rrbracket},-\tau+\bfx_{\llbracket\ell_0\rrbracket},v}{}}\dx w g_{12}(\xi,w;\zeta,v).
		\end{equation}
        For the integral over $v$, we just take $\Gamma_{\bma_{\llbracket\ell_0\rrbracket}}=(16 N)^{-1/3} \Gamma_{-\tau+\bfx_{\llbracket\ell_0\rrbracket}}$. After the change of variable $(w,v)\to (16 N)^{-1/3} (w,v)$, the integral over $v$ is over a bounded region, with an integrand uniformly bounded in $N$, so that we can take the limit inside the integral. For the integral over $w$, we proceed as in $\kexpresc_{N,11}$ to get the result.
\end{proof}
	Next, we need to upgrade the pointwise convergence to convergence in trace class. We introduce a further conjugation which is of order one for $\xi,\zeta$ bounded, so this does not changes the results derived so-far:
	\begin{equation}
		T_{ij}:=\begin{cases}
			e^{3\bfR\left(\zeta-\xi\right)},&{\rm if}\ (i,j)=(1,1),\\
			e^{3\bfR\left(\xi-\zeta\right)},&{\rm if}\ (i,j)=(2,2),\\
			e^{3\bfR\left(\xi+\zeta\right)},&{\rm if}\ (i,j)=(2,1),\\
			e^{-3\bfR\left(\xi+\zeta\right)},&{\rm if}\ (i,j)=(1,2).
		\end{cases}
	\end{equation}

	\begin{lem}\label{lem:trace}
		For any $s>0$ and $N$ large enough, there exists a constant $\kappa>0$ such that
		\begin{equation}
			\left|  T_{ij} G_{N,ij}\kexpresc_{N,ij}(\xi ,\zeta)\right|\leq C e^{-\kappa (\xi+\zeta)},\quad\forall\, \xi,\zeta\geq s.
		\end{equation}
Combining this with Lemma~\ref{lem:pt_cvg}, for which the estimates on the error terms holds true uniformly for $\xi,\zeta$ in a bounded set, then for any $s_0\in\R$ fixed, then the same estimate holds true for all $\xi,\zeta\geq s_0$ for some new constant $C=C(s_0)$.
	\end{lem}
		\begin{proof}

 Since the arguments for all four entries are largely analogous, we focus on the most intricate case, namely $\kexpresc_{N,12}$. For the remaining kernels, we only outline the key components. As before, we assume $\xi \leq \zeta$, and it remains to analyze each term on the right-hand side of~\eqref{E:decomp}.

\textit{First term in~\eqref{E:decomp}.} Substituting $w\mapsto (16N)^{-1/3}w$ and $v\mapsto (16N)^{-1/3}v$, we have
\begin{equation}
	\begin{aligned}
		K_{N,12}^1(\xi,\zeta)&=	\frac{1}{\left(2\pi\I\right)^2}\oint_{(16N)^{1/3}\gamma}\dx w\oint_{(16N)^{1/3}\Gamma}\dx v  \frac{F_N((16N)^{-1/3}v,(16N)^{-1/3}w)e^{w\xi-v\zeta}}{v-w}\\
	\end{aligned}
\end{equation}
with
\begin{equation}
	F_N(v,w):=\frac{e^{4Nf(v) +\tau(2N)^{2/3}h(v)}}{e^{4Nf(w) +\tau(2N)^{2/3}h(w)}}\frac{1}{\prod_{i=1}^{k} (w + \beta_i) \prod_{j=1}^{\ell} (v - {\alpha}_j)}.
\end{equation}
Note that by the definition of $\gamma$ and $\Gamma$, there exists a constant $\varepsilon>0$ independent of $N,\xi,\zeta$ such that
\begin{equation}
	\min\left\{|v-w|,|w+y_i|,|v-x_j|\big| i\in\llbracket k_0\rrbracket,j\in\llbracket\ell_0\rrbracket,v\in(16N)^{1/3}\gamma,\ w\in(16N)^{1/3}\Gamma\right\}\geq \varepsilon.
\end{equation}
Also note that
\begin{multline}
\bigg|\frac{G_{N,12}}{\prod_{i=1}^{k} ((16N)^{-1/3}w + \beta_i) \prod_{j=1}^{\ell} ((16N)^{-1/3}v-{\alpha}_j)} \\ \qquad-\frac{1}{\prod_{i=1}^{k_0}(w+y_i)\prod_{j=1}^{\ell_0}\left(v-x_j\right)}\bigg|=\Or(N^{-1/3})
\end{multline}
uniformly for $v\in(16N)^{1/3}\gamma,\ w\in(16N)^{1/3}\Gamma$. Furthermore, for $w\in(16N)^{1/3}\gamma$, we have
\begin{equation}
	{\rm Re}(w)\leq \cos\left(\pi\I/3\right)2\bfr=\bfr.
\end{equation}
Hence, there exists a constant $\kappa>0$ such that
\begin{equation}\label{E:w}
	\left|e^{w\xi-3\bfR \xi}\right|\leq C e^{-\kappa\xi},\quad\forall w\in(16N)^{1/3}\gamma.
\end{equation}
On the other hand, for all $v\in (16N)^{1/3}\Gamma$, we have ${\rm Re}(v)>0$, hence, we have
\begin{equation}
	\left|e^{-v\zeta-3\bfR \zeta}\right|\leq C e^{-\kappa\zeta},\quad\forall v\in(16N)^{1/3}\Gamma.
\end{equation}
From the saddle point analysis made for $\xi=\zeta=0$, we know that the kernel converges to its limit kernel. With respect to this case, here we have the extra factor $e^{w\xi-v\zeta}$ and the conjugation. Thus combining the arguments above and applying the same saddle point analysis, we obtain
\begin{equation}
	\left|  T_{12} G_{N,12}K_1(\xi ,\zeta)\right|\leq C e^{-\kappa (\xi+\zeta)}.
\end{equation}

\textit{Second term in~\eqref{E:decomp}.} As we saw in~\eqref{E:uni}, this term is uniformly bounded for any $\xi,\zeta$ in bounded set. Moreover, for $v\in\Gamma$, we have
\begin{equation}
	{\rm Re}(v)\leq \bfR+\sqrt{3}\bfR\leq 3\bfR.
\end{equation}
Thus, after multiplying it with conjugation function $e^{-3\bfR(\xi+\zeta)}$, the result will also decay exponentially in $\xi$ and $\zeta$.

\textit{Third term in~\eqref{E:decomp}.} This case can be handled via essentially same method as the one for the first term in~\eqref{E:decomp}. Indeed, substituting $w\mapsto (16N)^{-1/3}w$ and $v\mapsto (16N)^{-1/3}v$, we have
\begin{equation}
		K_{N,12}^3(\xi,\zeta)=-\frac{ 1}{\left(2\pi \I\right)^2}\oint_{\Gamma_{\bfx_{\llbracket \ell_0\rrbracket}}}\dx v\oint_{(16N)^{1/3}\gamma}\dx w\frac{F_N((16N)^{-1/3}v,(16N)^{-1/3}w)e^{w\xi-v\zeta}}{v-w}.
\end{equation}
Note that we can choose contour $\Gamma_{\bfx_{\llbracket \ell_0\rrbracket}}$ such that for all $v\in \Gamma_{\bfx_{\llbracket \ell_0\rrbracket}}$, we have
\begin{equation}
	{\rm Re}(v)>-2\bfr.
\end{equation}
Then we can apply the same method as the one for first term in~\eqref{E:decomp} to obtain the exponential decay.

\textbf{For $\kexpresc_{N,11}$. } We modify the contour $\Gamma$ by replacing $3\bfr$ in the definitions of $\mathcal{C}_1$ and $\mathcal{C}_2$ with $8\bfR$, and denote the resulting contour by $\hat\Gamma$. This modification is still of local order $\Or(N^{-1/3})$ and thus does not affect the saddle point analysis. Moreover, for $v \in \hat\Gamma$, we have
\begin{equation}
	{\rm Re}(v) \geq 4\bfR,
\end{equation}
so there exists a constant $\kappa > 0$ such that
\begin{equation}
	\left| e^{-v\zeta + 3\bfR\zeta} \right| \leq C e^{-\kappa \zeta}.
\end{equation}
Note also that~\eqref{E:w} remains valid. Hence, we obtain the desired exponential decay.

\textbf{For $\kexpresc_{N,22}$.} The result follows by an argument analogous to that for $\kexpresc_{N,11}$.

\textbf{For $\kexpresc_{N,21}$.} In addition to modifying the contour $\Gamma$ to $\hat\Gamma$, we also replace the contour $\gamma$ with $\hat\gamma$, defined as the reflection of $\hat\Gamma$ across the imaginary axis. With this choice, we have
\begin{equation}
	{\rm Re}(w) \leq -4\bfR,\quad {\rm Re}(v) > 4\bfR, \quad \forall\, w \in \hat\gamma,\ v \in \hat\Gamma.
\end{equation}
Consequently, there exists a constant $\kappa > 0$ such that
\begin{equation}
	\max\left\{ \left| e^{-v\zeta + 3\bfR\zeta} \right|, \left| e^{w\xi + 3\bfR\xi} \right| \right\} \leq C e^{-\kappa \zeta}.
\end{equation}
This completes the proof.
\end{proof}
		
Now we have all the ingredients to prove Theorem~\ref{E:AC_Conj}.
\begin{proof}[Proof of Theorem~\ref{E:AC_Conj}]
Lemma~\ref{lem:pt_cvg} implies that for $\xi, \zeta$ in a bounded set, we have
\begin{equation}
	\lim_{N \to \infty} G_{N,ij} \, \kexpresc_{N,ij}(\xi, \zeta)
	= K_{ij}^{\bfx_{\llbracket \ell_0\rrbracket}-\tau,\bfy_{\llbracket k_0\rrbracket}+\tau}(\xi+\tau^2, \zeta+\tau^2),
\end{equation}
for $i,j \in \{1,2\}$. In combination with Lemma~\ref{lem:trace}, this pointwise convergence allows us to get the convergence of the Fredholm determinant as well. In short, the Fredholm determinant defined through the Fredholm series converges if we can show that we can take the $N\to\infty$ limit inside the series. This is done by dominated convergence.

Indeed, we have
\begin{equation}\label{eq4.60}
\begin{aligned}
&\Pb \left(L^{\tilde\bma, \tilde\bmb, \bfx, \bfy}_{N,\tau} \leq 4N + (16N)^{1/3} s \right) \\
&=
\sum_{n\geq 0} \frac{(-1)^n}{n!} \sum_{k_1,\ldots,k_n=1}^2 \int_s^\infty d\xi_1  s \int_s^\infty d\xi_n \det[\kexpresc_{N,k_i k_j}(\xi_i,\xi_j)]_{1\leq i,j\leq n}\\
&=\sum_{n\geq 0} \frac{(-1)^n}{n!} \sum_{k_1,\ldots,k_n=1}^2\int_s^\infty d\xi_1  s \int_s^\infty d\xi_n \det[ T_{k_i k_j} G_{N,k_i k_j}\kexpresc_{N,k_i k_j}(\xi_i,\xi_j)]_{1\leq i,j\leq n},
\end{aligned}
\end{equation}
where in the last step we introduced the conjugations, that do not change the value of the determinant.
By Lemma~\ref{lem:trace}, $|T_{k_i k_j} G_{N,k_i k_j}\kexpresc_{N,k_i k_j}(\xi_i,\xi_j)|\leq C e^{-\kappa(\xi_i+\xi_j)}$. Using this and Hadamard's bound, saying that for a $n\times n$ matrix $A$ with $|A_{i,j}|\leq 1$, $|\det(A)|\leq n^{n/2}$, we get that for all $N$ large enough,
\begin{equation}
|\eqref{eq4.60}|\leq \sum_{n\geq 0} \frac{C^n n^{n/2}}{n!} \sum_{k_1,\ldots,k_n=1}^2 \Big(\int_s^\infty d\xi_1 e^{-2\kappa \xi_1}\Big)^n<\infty.
\end{equation}
Thus by dominated convergence we can take the limit inside the sums / integrals, which leads to
\begin{equation}
\begin{aligned}
&\lim_{N\to\infty} \Pb \left(L^{\tilde\bma, \tilde\bmb, \bfx, \bfy}_{N,\tau} \leq 4N + (16N)^{1/3} s \right)\\
&=\sum_{n\geq 0} \frac{(-1)^n}{n!} \sum_{k_1,\ldots,k_n=1}^2 \int_s^\infty d\xi_1  s \int_s^\infty d\xi_n \det[ T_{k_i k_j} K_{k_i k_j}^{\bfx_{\llbracket \ell_0\rrbracket}-\tau,\bfy_{\llbracket k_0\rrbracket}+\tau}(\xi_i+\tau^2, \xi_j+\tau^2)]_{1\leq i,j\leq n}\\
&=\sum_{n\geq 0} \frac{(-1)^n}{n!} \sum_{k_1,\ldots,k_n=1}^2 \int_s^\infty d\xi_1  s \int_s^\infty d\xi_n \det[ K_{k_i k_j}^{\bfx_{\llbracket \ell_0\rrbracket}-\tau,\bfy_{\llbracket k_0\rrbracket}+\tau}(\xi_i+\tau^2, \xi_j+\tau^2)]_{1\leq i,j\leq n}\\
&=F_{|\mathcal{I}(\tilde\bma)|,\, |\mathcal{I}(\tilde\bmb)|,\, \bfx_{\mathcal{I}(\tilde\bma)} - \tau,\, \bfy_{\mathcal{I}(\tilde\bmb)} + \tau}(s + \tau^2). \qedhere
\end{aligned}
\end{equation}
\end{proof}
		
\section{New representation for $F_{{\rm BR},\tau}$}\label{sec:BR}

\subsection{Proof of Corollary~\ref{cor:toBR}}\label{subsectBR}
Since $F_{{\rm BR},\tau}(s)=F_{{\rm BR},-\tau}(s)$ for all $\tau,s\in\R$, we may, without loss of generality, assume that $\tau\geq 0$. By~\eqref{E:gT}, it suffices to express the kernel in~\eqref{E:GBR} with
\begin{equation}\label{E:s}
	\begin{aligned}
		&k = \ell = 1,\\
		&\bfx = -\bfy = \tau
	\end{aligned}
\end{equation}
in terms of the Airy kernel and the Airy function. Since the procedure is identical for all four kernel entries, we only treat the entry
$K_{12}^{\bfx,\bfy}$. Assuming $\xi \leq \zeta$, we have
\begin{equation} \label{e52}
	\begin{aligned}
K_{12}^{\bfx,\bfy}(\xi,\zeta)=\frac{1}{\left(2\pi \I\right)^2}
\int_{\leftcontour{\tau}{}}\dx w\bigg[\int_{\rightcontour {\tau,w}{}}\dx v f_{12}(\xi,w;\zeta,v)- \oint_{\Gamma_{\tau}} \dx v f_{12}(\xi,w;\zeta,v)\bigg],
\end{aligned}
\end{equation}
where
\begin{equation}
f_{12}(\xi,w;\zeta,v):=\frac{e^{v^3/3 - v\zeta}}{e^{w^3/3 - w\xi}}  \frac{1}{v-w}\frac{1}{ \left(w-\tau\right) \left(v-\tau\right)}.
\end{equation}
Note that we have
\begin{equation}\label{e54}
	\begin{aligned}
		&-\frac{1}{\left(2\pi \I\right)^2}
		\int_{\leftcontour{\tau}{}}\dx w  \oint_{\Gamma_{\tau}} \dx v f_{12}(\xi,w;\zeta,v) =e^{\tau^3/3 - \tau\zeta}\frac{1}{2\pi \I}
		\int_{\leftcontour{\tau}{}}\dx w    \frac{e^{-w^3/3 + w\xi}}{ \left(w-\tau\right)^2 }.
	\end{aligned}
\end{equation}
By changing the contour of the variable $w$ from $\leftcontour{\tau}{}$ to $\leftcontour{}{\tau}$ and applying Cauchy's residue theorem, we obtain
\begin{equation}\label{e55}
	\begin{aligned}
		&\frac{1}{\left(2\pi \I\right)^2}
		\int_{\leftcontour{\tau}{}}\dx w\int_{\rightcontour {\tau,w}{}}\dx v f_{12}(\xi,w;\zeta,v)\\
		=&\frac{1}{\left(2\pi \I\right)^2}
		\int_{\leftcontour{}{\tau}}\dx w \int_{\rightcontour {\tau,w}{}}\dx v \frac{e^{v^3/3 - v\zeta}}{e^{w^3/3 - w\xi}}  \frac{1}{v-w}\frac{1}{ \left(w-\tau\right) \left(v-\tau\right)} +e^{-\frac{\tau^3}{3}+\tau\xi}\frac{1}{ 2\pi \I }
		\int_{\rightcontour {\tau}{}}\dx v\frac{e^{v^3/3 - v\zeta}}{  \left(v-\tau\right)^2}.
	\end{aligned}
\end{equation}
Substituting~\eqref{e54} and~\eqref{e55} into~\eqref{e52}, we obtain
\begin{equation}\label{E:airy}
	\begin{aligned}
		K_{12}^{\bfx,\bfy}(\xi,\zeta)=&\frac{1}{\left(2\pi \I\right)^2}
		\int_{\leftcontour{}{\tau}}\dx w \int_{\rightcontour {\tau,w}{}}\dx v \frac{e^{v^3/3 - v\zeta}}{e^{w^3/3 - w\xi}}  \frac{1}{v-w}\frac{1}{ \left(w-\tau\right) \left(v-\tau\right)}\\
		&+e^{\tau^3/3 - \tau\zeta}\frac{1}{2\pi \I}
		\int_{\leftcontour{\tau}{}}\dx w    \frac{e^{-w^3/3 + w\xi}}{ \left(w-\tau\right)^2 }+e^{-\frac{\tau^3}{3}+\tau\xi}\frac{1}{ 2\pi \I }
		\int_{\rightcontour {\tau}{}}\dx v\frac{e^{v^3/3 - v\zeta}}{  \left(v-\tau\right)^2}.
	\end{aligned}
\end{equation}
The key ingredients are the integral representation of Airy function, i.e.,
\begin{equation}\label{E:ai}
	\Ai(\xi)=\frac{1}{2\pi\I}\int_{\rightcontour{0}{}}\dx ve^{\frac{v^3}{3}-v\xi}=\frac{1}{2\pi\I}\int_{\leftcontour{}{0}}\dx ve^{-\frac{v^3}{3}+v\xi}.
\end{equation}

\textit{First term in~\eqref{E:airy}.} For $w\in\leftcontour{}{\tau}$ and $v\in\rightcontour {\tau,w}{}$, we have
\begin{equation}\label{E:r}
	\begin{aligned}
		&\frac{1}{w-\tau}=-\int_0^\infty\dx\gamma e^{(w-\tau)\gamma},\quad\frac{1}{v-\tau}=\int_0^\infty\dx\lambda e^{-(v-\tau)\lambda}.
	\end{aligned}
\end{equation}
Together with~\eqref{E:ai} and definition of Airy kernel~\eqref{E:airk}, the first term in~\eqref{E:airy} is equal to
\begin{equation}
	-\int_0^\infty\dx\lambda\int_0^\infty\dx\gamma\kai(\xi+\gamma,\zeta+\lambda)e^{\tau(\lambda-\gamma)}.
\end{equation}

\textit{Second term in~\eqref{E:airy}.} Applying Cauchy's residue theorem, it is equal to
\begin{equation}\label{E:sc}
	e^{\tau^3/3 - \tau\zeta}\frac{1}{2\pi \I}
	\int_{\leftcontour{}{\tau}}\dx w    \frac{e^{-w^3/3 + w\xi}}{ \left(w-\tau\right)^2 }-e^{\tau(\xi-\zeta)}(\tau^2-\xi)
\end{equation}
Applying~\eqref{E:r} and~\eqref{E:ai}, we have
\begin{equation}
\begin{aligned}
		\eqref{E:sc}&=e^{\tau^3/3 - \tau\zeta}\int_0^{\infty}\dx\lambda\int_0^\infty\dx\gamma \Ai(\xi+\lambda+\gamma)e^{-\tau(\lambda+\gamma)}-e^{\tau(\xi-\zeta)}(\tau^2-\xi)\\
		&=e^{\tau^3/3 - \tau\zeta}\int_0^{\infty}\dx x  \Ai(\xi+x )e^{-\tau x}x-e^{\tau(\xi-\zeta)}(\tau^2-\xi),
\end{aligned}
\end{equation}
where in the last step we use the substitution $x=\lambda+\gamma$ and interchange the order of integration.

\textit{Third term in~\eqref{E:airy}.} Applying~\eqref{E:r} and~\eqref{E:ai}, we can rewrite it as
\begin{equation}
	e^{-\frac{\tau^3}{3}+\tau\xi}\int_0^\infty\dx\lambda\Ai(\zeta+\lambda)\lambda e^{\tau\lambda}.
\end{equation}

If $\xi > \zeta$, the procedure is identical to the case $\xi \leq \zeta$, except that the term \mbox{$e^{\tau(\xi - \zeta)}(\tau^2 - \xi)$} is replaced by $e^{\tau(\xi - \zeta)}(\tau^2 - \zeta)$. By combining both cases, we obtain the desired formula for $K_{12}^{{\rm BR},\tau}$ in~\eqref{E:brtau} .

\subsection{Analytic proof for $F_{{\rm BR},0}$}
		In this section, we give a direct analytic proof for Corollary~\ref{cor:toBR} with $\tau=0$, that is,
		\begin{equation}\label{E:toshow}
			F_{{\rm BR},0}(t) = \det(\Id-P_{t} K^{{\rm BR},0} P_{t})_{L^2(\R)}
		\end{equation}
where $F_{{\rm BR},0}(t)$ is defined by its classical form in \eqref{E:classical_BR} below.

		With $\tau=0$, we have
		\begin{equation}
			K_{12}^{{\rm BR},0}(\xi,\zeta)=L(\xi,\zeta)+\min\{\xi,\zeta\},
		\end{equation}
		where
		\begin{equation}\label{E:L}
			\begin{aligned}
				L(\xi, \zeta) := &-\int_0^\infty \mathrm{d}x \int_0^\infty \dx y\, \kai(x+\xi, y+\zeta)\\
				&+ \int_0^\infty \mathrm{d}x \int_0^\infty \dx y\, \left( \Ai(\xi+x+y) + \Ai(\zeta+x+y) \right).\\
			\end{aligned}
		\end{equation}
		An elementary computation gives
		\begin{equation}\label{E:kernel_K}
			\begin{pmatrix}
				K_{11}^{{\rm BR},0} &K_{12}^{{\rm BR},0}\\
				K_{21}^{{\rm BR},0} &K_{22}^{{\rm BR},0}
			\end{pmatrix}=\begin{pmatrix}
			\partial_2^2 L(\xi, \zeta) & L(\xi, \zeta) + \min\{\xi, \zeta\} \\
			 \partial_1^2 \partial_2^2 L(\xi, \zeta)& \partial_1^2 L(\xi, \zeta)
			\end{pmatrix}=:K^\top.
		\end{equation}
		 Applying $\det\left(\Id-K^\top\right)=\det\left(\Id-K\right)$ for a general kernel $K$, we can rewrite the right hand side of~\eqref{E:toshow} as
		\begin{equation}\label{E:newRep}
			\hat{F}(t) := \det\left(\Id - P_t K P_t\right)_{L^2(\R)}.
		\end{equation}
	
		\begin{prop}\label{cor:old_and_new}
			For all $t \in \R$, we have $\hat{F}(t) = F_{{\rm BR},0}(t)$.
		\end{prop}
        The remainder of the section deals with the proof of this statement.

		\subsubsection{Decomposition of kernel}
		The starting point of the proof is the following decomposition of kernel:
		\begin{lem}[Lemma 3.9.37 of \cite{AGZ10}]\label{lem:decomposition_of_kernel}
			Let $a, b, c, d, e, \sigma$, and $w$ be functions on $\R^2$ such that $d = \sigma + w$, and assume that all Fredholm determinants appearing below are well defined. Then,
			\begin{equation}\label{E:dright}
				\det\left(\Id -
				\begin{pmatrix}
					a & b \\
					c + e & d
				\end{pmatrix}
				\right)_{L^2(\R)}
				= \det\left(\Id - X \hat{R} \right)_{L^2(\R)}   \det\left(\Id - \sigma\right)_{L^2(\R)},
			\end{equation}
			where
			\begin{align}\label{E:components}
				X &=
				\begin{pmatrix}
					a + b e & (a + b e) b \\
					c + d e & w + (c + d e) b
				\end{pmatrix},
				\qquad
				\hat{R} =
				\begin{pmatrix}
					1 & 0 \\
					0 & (\Id - \sigma)^{-1}
				\end{pmatrix},
			\end{align}
			and
			\begin{equation}
				be(x, y) := \int_0^\infty \mathrm{d}z\, b(x, z)\, e(z, y).
			\end{equation}
		\end{lem}
		Note that
		\begin{equation}
			\partial_2^2L(x,y)=K^{{\rm BR},0}_{11}=\kai(x, y) + \left( 1 - \int_0^\infty \mathrm{d}\gamma\, \Ai(x + \gamma) \right) \Ai(y).
		\end{equation}
		Hence, we may apply Lemma~\ref{lem:decomposition_of_kernel} to the kernel $K$ in~\eqref{E:kernel_K}, with the following identification:
		\begin{equation}\label{E:new}
			\begin{aligned}
				\left[\begin{array}{rr}
					a & b \\
					c & d
				\end{array}\right]
				&:= \left[\begin{array}{rr}
					P_t \partial_1^2 L P_t & P_t \partial_1^2 \partial_2^2 L P_t \\
					P_t L P_t & P_t \partial_2^2 L P_t
				\end{array}\right], \qquad\qquad\quad
				e(x, y) := P_t \min\{x, y\} P_t, \\
				w(x, y) &:= P_t\left(1 - \int_0^\infty \mathrm{d}\gamma\, \Ai(x + \gamma)\right) \Ai(y)P_t \qquad \sigma(x, y) := P_t\kai(x, y)P_t.
			\end{aligned}
		\end{equation}
		Let $\rho_t := (\Id - \sigma)^{-1}$. To proceed, we compute the components of the right-hand side of~\eqref{E:dright} under the setting~\eqref{E:new}. Using integration by parts and the exponential decay of the kernel $L$ and its derivative, a straightforward computation yields the following lemma.	
		\begin{lem}\label{lem:decomposition}
		Let $\hat F(t)$ be defined as in~\eqref{E:newRep}. Then we have

			\begin{equation}\label{E:hatF}
				\hat F(t)=\det\left(\Id - X \hat{R} \right)_{L^2(\R)} \, F_{{\rm GUE}}(t),
			\end{equation}
			where
			\begin{equation}\label{E:hatr}
				\hat{R} =
				\begin{pmatrix}
					1 & 0 \\
					0 & \rho_t
				\end{pmatrix}
			\end{equation}
			and
			\begin{equation}\label{E:Xdcomp}
				X(\xi,\zeta) := X_\phi(\xi) \, X_\psi(\zeta)
			\end{equation}
			with
			\begin{equation}
				X_\phi(\xi) :=
				\begin{bmatrix}
					P_t\phi_1(\xi) & 0 \\
					P_t\phi_2(\xi) & P_t\phi_3(\xi)
				\end{bmatrix}, \quad
				X_\psi(\zeta) :=
				\begin{bmatrix}
					1 & P_t\psi_1(\zeta) \\
					0 & P_t\psi_2(\zeta)
				\end{bmatrix}.
			\end{equation}
			The functions $\phi_i$ and $\psi_i$ are defined by
			\begin{align}
				\phi_1(x) &= t \partial_1 \kai(x,t)
				+ \int_0^\infty \!\dx \gamma \partial_1\kai(x, t+\gamma) + \Ai(x), \label{E:phi_1} \\
				\phi_2(x) &= -t \int_0^\infty \!\dx \gamma \, K_{\mathrm{Ai}}( x+\gamma,t)
				+ t \int_0^\infty \!\dx \gamma \, \Ai(\gamma + t) + L(x,t), \label{E:phi_2} \\
				\phi_3(x) &= 1 - \int_0^\infty \!\dx \gamma \, \Ai(\gamma+x), \\
				\psi_1(x) &= \partial_1 \kai(x, t), \qquad \psi_2(x) = \Ai(x).
			\end{align}
			 	
		\end{lem}
	
		\subsubsection{Link to Painlev\'e transcendent}
			We begin by recalling the definition of the Baik–Rains distribution $F_{{\rm BR},0}(t)$. Through the whole section, we choose a fixed $t\in\R$. We adapt the notation from~\cite{AGZ10}. For two functions $a, b : \R \to \R$, we define
		\begin{equation}\label{E:inner}
			\langle a \mid b \rangle_t
			:= \int_0^\infty \dx x \int_0^\infty \dx y \,
			a(x+t)\, \rho_t(x+t, y+t)\, b(y+t).
		\end{equation}
		Note that $\langle a \mid b \rangle_t = \langle b \mid a \rangle_t$, since $\kai$ is symmetric.	With a slight abuse of notation, for two functions $\varphi, \psi : \R^2 \to \R$, we set
		\begin{equation}
			\langle \varphi \mid \psi \rangle_t
			:= \langle \varphi(\, \,, t) \mid \psi(\, \,, t) \rangle_t.
		\end{equation}
		Let $f(x):=\Ai(x)$ and $g(x, t) := \kai(x, t)$, we define $q(t) := \Ai(t) + \langle f \mid g \rangle_t$ and $\bfx(t):=e^{-\int_0^\infty \dx x\, q(x + t)}$. The Baik–Rains distribution $F_0(t)$ admits the following representation (see (2.20) in~\cite{BR00}):
		\begin{equation}\label{E:classical_BR}
			F_{{\rm BR},0}(t) = \frac{\mathrm{d}}{\mathrm{d}t} \left( F_{{\rm GUE}}(t) \Upsilon(t)\right)=F_{{\rm GUE}}'(t)\Upsilon(t) +\bfx(t)^2F_{{\rm GUE}}(t).
		\end{equation}
		with $\Upsilon(t) := \int_{-\infty}^t \mathrm{d}s\, \bfx(s)^2$ It was shown in Proposition~A.1 of~\cite{FS05a} that
		\begin{equation}\label{E:upsilon}
			\Upsilon(t)= t + \Psi(t) - \langle (1 + F) \mid (F + H) \rangle_t,
		\end{equation}
		where
		\begin{equation}\label{E:notation_1}
			\begin{aligned}
				\Psi(x) &:= \int_0^\infty \mathrm{d}\gamma \int_0^\infty \mathrm{d}\lambda\, \Ai(x + \gamma + \lambda), \\
				F(x) &:= -\int_0^\infty \mathrm{d}\gamma\, \Ai(x + \gamma), \\
				H(x, t) &:= \int_0^\infty \mathrm{d}\gamma\, \kai(x, t + \gamma).
			\end{aligned}
		\end{equation}
		Let $R(x, y; t)$ denote the resolvent of the operator $P_t \kai P_t$, i.e., $1 + R(x, y; t) = \left( \Id - P_t \kai P_t \right)^{-1}(x, y)$, then we have (see Equation (3.8.16) in~\cite{AGZ10})
		 \begin{equation}
			F_{{\rm GUE}}'(t) = R(t, t; t)\, F_{{\rm GUE}}(t)
		\end{equation}
		Using the fundamental identity between the resolvent and the kernel (Equation (3.4.20) in~\cite{AGZ10}), we have $R(t, t; t) = g(t, t) + \langle g \mid g \rangle_t.$ Substituting this into~\eqref{E:classical_BR}, we conclude that
		\begin{equation}
			F_0(t) = \left( \Upsilon(t)\left( g(t, t) + \langle g \mid g \rangle_t \right) + \bfx(t)^2 \right) F_{{\rm GUE}}(t).
		\end{equation}
		Comparing with~\eqref{E:hatF}, to prove Proposition~\ref{cor:old_and_new}, it suffices to show that
		\begin{equation}\label{E:toShow}
			\det\left(\Id - X \hat{R} \right) = \Upsilon(t) \left( g(t,t) + \langle g \mid g \rangle_t \right) + \bfx(t)^2.
		\end{equation}

		\begin{proof}[Proof of Proposition~\ref{cor:old_and_new}]
			Using the cyclic property of the determinant and~\eqref{E:Xdcomp}, we obtain
			\begin{equation}\label{E:determinant_finite_rank}
				\begin{aligned}
					\det&\left(\Id - X \hat{R} \right)
					= \det\left(\Id - X_\psi \hat{R} X_\phi \right) \\
					&= \left(1 - \int_0^\infty \!\dx x \, \phi_1(x+t) - \langle \psi_1 \mid \phi_2 \rangle_t \right)
					\left(1 - \langle \psi_2 \mid \phi_3 \rangle_t \right)
					- \langle \psi_1 \mid \phi_3 \rangle_t \, \langle \psi_2 \mid \phi_2 \rangle_t,
				\end{aligned}
			\end{equation}
			where in the last step we used the explicit rank-two structure of the operator $X_\psi \hat{R} X_\phi$. By definition, we have
			\begin{equation}\label{E:rewrite}
				\begin{aligned}
					&\int_0^\infty\dx x\phi_1(x+t)=-tg(t)+G(t)-F(t)\\
					&\phi_2(x)=tG(x,t)-tF(t)+\Phi(x,t)+\Psi(x)+\Psi(t)\\
					&\phi_3(x)=1+F(x),\qquad \psi_1(x)=\psi(x,t),\qquad \psi_2(x)=f(x),
				\end{aligned}
			\end{equation}
			where
			\begin{equation}
				\begin{aligned}
					G(x,t)&:=-\int_0^\infty\dx\gamma\kai(x+\gamma,t),\qquad\qquad		\Psi(x):=\int_0^\infty\dx\lambda\int_0^\infty\dx \gamma\Ai(x+\lambda+\gamma),\\
					\Phi(x,t)&:=-\int_0^\infty\dx\lambda\int_0^\infty\dx \gamma\kai(x+\lambda,t+\gamma),\qquad
					\psi(x,t):=\partial_1\kai(x,t).
				\end{aligned}
			\end{equation}
			By~\eqref{E:rewrite}, we can express $\langle \psi_i \mid \phi_j \rangle_t$ with $i \in \{1,2\}$ and $j \in \{2,3\}$ in terms of the following quantities:
			\begin{align}
				&\inner{f\mid G }_t,\quad \inner{f\mid 1}_t ,\quad\inner{f\mid F}_t\label{E:toCompute1},\\
				&\inner{\psi \mid G }_t,\quad \inner{\psi \mid 1},\quad \inner{\psi \mid \Phi }_t,\quad\inner{\psi \mid\Psi},\quad\inner{\psi \mid F}_t\label{E:toCompute3}\\							&\inner{f\mid \Phi }_t,\quad\inner{f\mid\Psi}_t.\label{E:toCompute2}
			\end{align}
			The main task of the proof is to express the above quantities in the form $\inner{x \mid y}_t$ with $x, y \in \{F, H, 1\}$, as well as $\bfx(t)$ and the quantity $g(t,t) + \inner{g \mid g}_t$. The quantities in~\eqref{E:toCompute1} are already known in the literature (see Equation (3.9.95) in~\cite{AGZ10})
			\begin{equation}\label{E:f_painleve}
				\begin{aligned}
					&\inner{f\mid G}_t=\frac{\bfx+\bfx^{-1}}{2} F(t)-\frac{\bfx-\bfx^{-1}}{2},\\
					&\inner{f\mid 1}_t=-\frac{\bfx-\bfx^{-1}}{2},\qquad \inner{f\mid F}_t=1-\frac{\bfx+\bfx^{-1}}{2},
				\end{aligned}
			\end{equation}
			where we write $\bfx$ instead of $\bfx(t)$ to lighten the notation.	For~\eqref{E:toCompute3} and~\eqref{E:toCompute2}, we make use of the following identities from (3.9.93) and (3.9.94) of~\cite{AGZ10}:
			\begin{align}
				\frac{\dx }{\dx t}\langle \phi_1 \mid \phi_2 \rangle_t &= \langle \mathbf{D} \phi_1 \mid \phi_2 \rangle_t + \langle \phi_1 \mid \mathbf{D} \phi_2 \rangle_t - \langle \phi_1 \mid f \rangle_t \langle f \mid \phi_2 \rangle_t, \label{E:dt_bracket} \\
				\langle \phi_1' \mid \phi_2 \rangle_t + \langle \phi_1 \mid \phi_2' \rangle_t &= -\left( \phi_1^- + \langle g \mid \phi_1 \rangle_t \right) \left( \phi_2^- + \langle g \mid \phi_2 \rangle_t \right) + \langle f \mid \phi_1 \rangle_t \langle f \mid \phi_2 \rangle_t. \label{E:integration_by_part_bracket}
			\end{align}
			
			Here, for any sufficiently smooth function $\phi = \phi(x, t)$, we adopt the following notation (see page 178 of~\cite{AGZ10}) :
			\begin{equation}\label{E:derivative_notation}
				\begin{aligned}
					\phi'(x; t) &:= \frac{\partial \phi}{\partial x}(x; t), \qquad
					\phi^{-}:=\phi^{-}(t)=\phi(t ; t),\\
					\mathbf{D} \phi(x; t) &:= \left( \frac{\partial}{\partial x} + \frac{\partial}{\partial y} \right) \phi(x; t).
				\end{aligned}
			\end{equation}
			With a slight abuse of notation, if $\phi : \R \to \R$, we continue to write $\mathbf{D} \phi(x) := \phi'(x)$ and $\phi^-:=\phi(t)$.
			
			\textbf{For~\eqref{E:toCompute3}}. We apply~\eqref{E:integration_by_part_bracket} and the identities
			\begin{equation}
				\Psi' = F, \quad F' = f, \quad \Phi' = H, \quad G' = g, \quad g' = \psi,
			\end{equation}
			we have
			\begin{equation}\label{E:psi_mid}
				\begin{aligned}
					&\langle \psi\mid G\rangle_t=-\left(G^-+\langle g\mid G\rangle_t\right)\left(g^-+\langle g\mid g\rangle_t\right)+\langle f\mid g\rangle_t\langle f\mid G\rangle_t-\langle g\mid g\rangle_t,\\
					&\langle \psi\mid F\rangle_t=-\left(g^-+\langle g\mid g\rangle_t\right)\left(F^-+\langle g\mid F\rangle_t\right)+\langle f\mid g\rangle_t\langle f\mid F\rangle_t-\langle f\mid g\rangle_t,\\
					&\langle \psi\mid \Phi\rangle_t=-\left(g^-+\langle g\mid g\rangle_t\right)\left(\Phi^-+\langle g\mid\Phi\rangle_t\right)+\langle f\mid g\rangle_t\langle f\mid \Phi\rangle_t-\langle g\mid H\rangle_t\\
					&\langle \psi\mid\Psi\rangle_t=-(\Psi^-+\langle g\mid\Psi\rangle_t)\left(g^-+\langle g\mid g\rangle_t\right)+\langle f\mid\Psi\rangle_t\langle f\mid g\rangle_t-\langle g\mid F\rangle_t,\\
					&\langle \psi\mid 1\rangle_t=-(g^-+\langle g\mid g\rangle_t)\left(1+\langle g\mid 1\rangle_t\right)+\langle f\mid g\rangle_t\langle f\mid 1\rangle_t.\\
				\end{aligned}
			\end{equation}
			Hence, in addition to~\eqref{E:f_painleve}, it remains to compute
			\begin{align}
				&\inner{g\mid G}_t,\quad \inner{g\mid 1}_t,\quad\inner{g\mid F}_t,\label{E:toCompute4}\\
				&\inner{g\mid \Phi}_t,\quad \inner{f\mid \Phi}_t,\label{E:toCompute5}\\
				&\inner{g\mid\Psi}_t,\quad \inner{f\mid\Psi}_t,\label{E:toCompute6}\\
				&\inner{g\mid H}_t.\label{E:toCompute7}
			\end{align}
			The quantities in~\eqref{E:toCompute4} are already known~\cite[(3.9.95)]{AGZ10}
			\begin{equation}\label{E:g_painleve}
				\begin{aligned}
					\inner{g\mid G}_t&=\frac{\bfx+\bfx^{-1}}{2}-\frac{\bfx-\bfx^{-1}}{2} F^{-}+\left(F^-\right)^2 / 2-1\\
					\inner{g\mid 1}_t&=\frac{\bfx+\bfx^{-1}}{2}-1,\qquad \inner{g\mid F}_t=\frac{\bfx-\bfx^{-1}}{2}-F(t).
				\end{aligned}
			\end{equation}
			For~\eqref{E:toCompute5}, using the identities $F' = f$, $\Phi' = H$, together with~\eqref{E:integration_by_part_bracket}, we obtain
			\begin{equation}
				\begin{aligned}
					\langle f\mid \Phi\rangle_t+\langle F\mid H\rangle_t&=-\left(\Phi^-+\langle g\mid\Phi\rangle_t\right)\left(F^-+\langle g\mid F\rangle_t\right)+\langle f\mid F\rangle_t\langle f\mid\Phi\rangle_t,\\
					\langle H\mid 1\rangle_t&=-\left(\Phi^-+\langle g\mid\Phi\rangle_t\right)\left(1+\langle g\mid 1\rangle_t\right)+\langle f\mid \Phi\rangle_t\langle f\mid1\rangle_t,
				\end{aligned}
			\end{equation}
			solving this linear system and applying the identities in~\eqref{E:g_painleve}, we obtain
			\begin{equation}\label{E:midPhi}
				\begin{aligned}
					&\langle f\mid\Phi\rangle_t=-\frac{\langle F\mid H\rangle_t(\mathbf x+\mathbf x^{-1})-\langle H\mid 1\rangle_t\left(\mathbf x-\mathbf x^{-1}\right)}{2},\\
					&\langle g\mid\Phi\rangle_t=-\frac{\langle F\mid H\rangle_t\left(\mathbf x^{-1}-\mathbf x\right)+\langle H\mid 1\rangle_t\left(\mathbf x+\mathbf x^{-1}\right)+2\Phi^-}{2}.
				\end{aligned}
			\end{equation}	
			For~\eqref{E:toCompute6}, applying the same procedure to both $\langle f \mid \Psi \rangle_t$ and $\langle g \mid \Psi \rangle_t$, using the identity $\Psi' = F$, we obtain
			\begin{equation}\label{E:midPsi}
				\begin{aligned}
					&\langle f\mid\Psi\rangle_t=-\frac{\langle F\mid F\rangle_t(\mathbf x+\mathbf x^{-1})-\langle F\mid 1\rangle_t\left(\mathbf x-\mathbf x^{-1}\right)}{2},\\
					&\langle g\mid\Psi\rangle_t=-\frac{\langle F\mid F\rangle_t\left(\mathbf x^{-1}-\mathbf x\right)+\langle F\mid 1\rangle_t\left(\mathbf x+\mathbf x^{-1}\right)+2\Psi^-}{2}.\\
				\end{aligned}
			\end{equation}
			For~\eqref{E:toCompute7}, we define
			\begin{equation}
				\hat g(x,t):=\frac{\dx}{\dx x}H(x,t)
			\end{equation}
			Hence by~\eqref{E:integration_by_part_bracket},
			\begin{equation}\label{E:Hf_Hg}
				\begin{aligned}
					\langle \hat g\mid F\rangle_t+\langle H\mid f\rangle_t&=-\left(H^-+\langle g\mid H\rangle_t\right)\left(F^-+\langle g\mid F\rangle_t\right)+\langle f\mid F\rangle_t\langle f\mid H\rangle_t,\\
					\langle \hat g\mid 1\rangle_t&=-\left(H^-+\langle g\mid H\rangle_t\right)\left(1+\langle g\mid 1\rangle_t\right)+\langle f\mid H\rangle_t\langle f\mid 1\rangle_t.
				\end{aligned}
			\end{equation}	
			On the other hand, we have $\hat g(x,t)=F^-f(x)+g(x,t)$, together with~\eqref{E:f_painleve} and~\eqref{E:g_painleve}, we have
			\begin{equation}
				\begin{aligned}
					\langle \hat g\mid 1\rangle_t&=\langle g\mid 1\rangle_t+F^-\langle f\mid 1\rangle_t=\frac{\bfx+\bfx^{-1}}{2}-1+F^-\left(-\frac{\bfx-\bfx^{-1}}{2}\right),\\
					\langle \hat g\mid F\rangle_t&=\langle g\mid F\rangle_t+F^-\langle f\mid F\rangle_t=\frac{\bfx-\bfx^{-1}}{2}-F^{-}+F^-\left(	1-\frac{\bfx+\bfx^{-1}}{2}\right)
				\end{aligned}
			\end{equation}
			Plugging this back to~\eqref{E:Hf_Hg} and solving the linear equation, we obtain
			\begin{equation}\label{E:midH}
					\langle f\mid H\rangle_t=-\frac{\mathbf x-\mathbf x^{-1}-2F^-}{2},\qquad
					\langle g\mid H\rangle_t=-\frac{-\mathbf x-\mathbf x^{-1}+2+2H^-}{2}.
			\end{equation}
			Substituting~\eqref{E:g_painleve},~\eqref{E:midPhi},~\eqref{E:midPsi}, and~\eqref{E:midH} into~\eqref{E:psi_mid}, we obtain the desired representation of~\eqref{E:toCompute2}, whose explicit expression is omitted here due to its complexity. On the other hand, the expressions required for~\eqref{E:toCompute3} are already contained in~\eqref{E:midPhi} and~\eqref{E:midPsi}. Therefore, in conjunction with~\eqref{E:f_painleve}, we obtain the desired representations for all quantities in~\eqref{E:toCompute1},~\eqref{E:toCompute3}, and~\eqref{E:toCompute2}. Plugging this and
			\begin{equation}
				\int_0^\infty\dx x\phi_1(x+t)=-tg(t)+G(t)-F(t)
			\end{equation}
			back to~\eqref{E:determinant_finite_rank}, we then obtain
			\begin{equation}\label{E:final}
				\begin{aligned}
					\det\left(\Id - X \hat{R} \right) =&\left(g^-+\langle g\mid g\rangle_t\right)\Upsilon(t)\\
					&+\mathbf x\left(g^-+\langle g\mid g\rangle_t\right)\left(-G^--H^-+\frac12 \left(F^-\right)^2+G^-\right)+\mathbf x^2.\\
				\end{aligned}
			\end{equation}
			By definition and symmetry of Airy kernel, we have
			\begin{equation}
				G^-=-H^-.
			\end{equation}
			On the other hand, a straightforward calculation shows that
			\begin{equation}
				\int_0^\infty \dx x\kai(x+t,t)=\frac12\left(\int_0^\infty \dx x\Ai(x+t)\right)^2,
			\end{equation}
			which implies $\frac12 \left(F^-\right)^2=-G^-$.  This completes the proof of~\eqref{E:toShow}.
		\end{proof}
		\begin{rem}
			As already noted in Remark~3.9.42 of~\cite{AGZ10}, the evaluations of the
			determinants in~\eqref{E:final} are lengthy, but can be carried out using symbolic
			computations in \textit{Mathematica}. This is the approach we adopt, and the
			corresponding code is provided in the BonnData repository~\cite{FK2/5PUB1P_2025}.
		\end{rem}

\section{Proof of the variational formula}\label{sec:var}
In this section, we prove the variational formula in Theorem~\ref{thm:var}. By Theorem~\ref{thm:AC_conj}, for fixed $k, \ell \in \Z_{>0}$, we have
\begin{equation}\label{E:lhs}
	\lim_{N \to \infty} \Pb \left( L^{0,0,0,0}_{N,0} \leq 4N + (16N)^{1/3} s \right)
	= F_{k,\ell,0,0}(s).
\end{equation}
Note that $L^{0,0,0,0}_{N,0}$ corresponds to the case $\bma = \bmb = 0$ in~\eqref{E:ThickBd}. Therefore, to prove Theorem~\ref{thm:var}, it suffices to show that
\begin{equation}\label{E:var}
	{\rm l.h.s.\ of }~\eqref{E:lhs} = \Pb \left( \max_{t \in \R} \left\{
	\sqrt{2}\mathcal{B}_k(t) \Id_{t \geq 0}
	+ \sqrt{2} \wt{\mathcal{B}}_\ell(-t) \Id_{t < 0}
	+ \mathcal{A}_2(t) - t^2
	\right\} \leq s \right).
\end{equation}
For notational convenience, we write $L^{0,0,0,0}_{N,0}$ as $L_N$, and denote $L^{0,0}_{(i,j)\to(m,n)}$ by $L_{(i,j)\to(m,n)}$. Moreover, when $(i,j) = (-\ell+1, -k+1)$, we further abbreviate it as $L_{(m,n)}$.

In the setting of~\eqref{E:ThickBd}, we have
\begin{equation}\label{E:LN}
	L_{N} = \max\left\{ L_{(-\ell+1,\,1)\to(N,N)},\; L_{(1,\,-k+1)\to(N,N)} \right\}.
\end{equation}
For $\bfa, \bfb \in \Z^2$ with $(-\ell+1, -k+1) \le \bfa \le \bfb$, we denote by $\Gamma_{\bfa \to \bfb}$ the geodesic from $\bfa$ to $\bfb$, i.e., the path $\pi$ connecting $\bfa$ and $\bfb$ such that
\begin{equation}
	L_{\bfa \to \bfb} = \sum_{(r,s)\in\pi} \omega_{r,s}.
\end{equation}
We introduce the \textit{exit point} of the geodesic $\Gamma_{(-\ell+1, -k+1) \to (N, N)}$ from the thick boundary:
\begin{equation}
	Z^{\mathbf{N}} :=
	\begin{cases}
		\max\left\{ i \mid (i,0) \in \Gamma_{(1,-k+1) \to (N,N)} \right\}, & \text{if } L_N = L_{(1,-k+1) \to (N,N)}, \\[1ex]
		-\max\left\{j \mid (0,j) \in \Gamma_{(-\ell+1,1) \to (N,N)} \right\}, & \text{if } L_N = L_{(-\ell+1,1) \to (N,N)}.
	\end{cases}
\end{equation}
In view of \eqref{E:LN}, the geodesic $\Gamma_{(-\ell+1, -k+1) \to (N, N)}$ either moves horizontally until reaching the point $(1, -k+1)$ and then follows $\Gamma_{(1, -k+1) \to (N, N)}$, or moves vertically until reaching the point $(-\ell+1, 1)$ and then follows $\Gamma_{(-\ell+1, 1) \to (N, N)}$. Thus, $Z^{\mathbf{N}}$ marks the point where the geodesic $\Gamma_{(-\ell+1, -k+1) \to (N, N)}$ exits the thick boundary, which consists of the first $k$ rows and the first $\ell$ columns.

The proof of~\eqref{E:var} relies on two key ingredients: the convergence of exponential LPP on thin rectangles, and control over the exit point.
The former is already established in the literature (see, e.g.,~\cite{ABC12,Bar01,GW91}). The control of the exit time has also been studied in the literature under various regimes (see, e.g.,~\cite{BHA20,EJS23,SS20}). In our setting, we obtain the following result.
\begin{prop}\label{p:exit}
	For any $u\in\R_{\ge 0}$ and $N$ large enough, there exists $C,\kappa>0$ independent of $N$ and $u$ such that
	\begin{equation}
		\Pb\left(|Z^{\mathbf N}|\geq uN^{2/3}\right)\leq C  e^{-\kappa u}.
	\end{equation}
\end{prop}
The proof of Proposition~\ref{p:exit} is deferred to the end of this section. We are now in a position to prove Theorem~\ref{thm:var}.

\begin{proof}[Proof of Theorem~\ref{thm:var}]
	As mentioned earlier, it suffices to prove~\eqref{E:var}. For $M,N>0$, we define
	\begin{equation}\label{E:dmn}
		D_{M,N} := \left\{ |Z^{\mathbf N}| \leq 2^{5/3} M N^{2/3} \right\}.
	\end{equation}
	By Proposition~\ref{p:exit}, there exist constants $C,\kappa>0$ such that
	\begin{equation}\label{E:eps}
		\Pb\left(D_{M,N}\right)>1-Ce^{-\kappa M}
	\end{equation}
	for all $N$ large enough. For $u\in\R$, we define
	\begin{equation}
		\begin{aligned}
			L_N^-(u)&:=\Id_{u\leq 0}\left(L_{(-\ell+1,1)\to(0,-2^{5/3}N^{2/3}u)}+L_{(1,-2^{5/3}N^{2/3}u)\to(N,N)}\right)\\
			L_N^+(u)&:=\Id_{u>0}\left(L_{(1,-k+1)\to(2^{5/3}N^{2/3}u,0)}+L_{(2^{5/3}N^{2/3}u,1)\to(N,N)}\right).
		\end{aligned}
	\end{equation}	
	Then by~\eqref{E:LN}, under the event $D_{M,N}$, we have
	\begin{equation}\label{E:lndecomp}
		L_N=\max_{|u|\leq M}\left(L_N^-(u), L_N^+(u)\right).
	\end{equation}
	Note that we have
	\begin{equation}\label{E:LNM}
		\begin{aligned}
			&\frac{L_N^-(u)-4N}{\left(16N\right)^{1/3}}=S_{N,1}(u)+S_{N,2}(u)
		\end{aligned}
	\end{equation}
	with
	\begin{equation}
		\begin{aligned}
			S_{N,1}(u)&:=\Id_{u\leq 0}\frac{L_{(-\ell+1,1)\to(0,-2^{5/3}N^{2/3}u)}-2^{8/3}N^{2/3}u}{\left(16N\right)^{1/3}}\\
			S_{N,2}(u)&:=\Id_{u\leq 0}\frac{L_{(1,-2^{5/3}N^{2/3}u)\to(N,N)}-(4N-2^{8/3}N^{2/3}u)}{\left(16N\right)^{1/3}}.
		\end{aligned}
	\end{equation}
	By Theorem~\ref{thm:strip}, as $N \to \infty$, the process $S_{N,1}(u)$ converges in distribution to $\Id_{u \leq 0}\sqrt{2}\wt{\mathcal{B}}_\ell(u)$. On the other hand, under the setting~\eqref{E:ThickBd}, the term
	\begin{equation}
		L_{(1, -2^{5/3}N^{2/3}u) \to (N,N)}-(4N - 2^{8/3}N^{2/3}u)
	\end{equation}
	can be interpreted as a last passage time without boundary conditions. Hence, as $N \to \infty$, the process $S_{N,2}(u)$ converges in distribution to $\Id_{u \leq 0} (\mathcal{A}_2(u) - u^2)$ (see, e.g.,~\cite{Jo03,PS02}). Plugging this into~\eqref{E:LNM}, we obtain
	\begin{equation}
		\lim_{N \to \infty} \frac{L_N^-(u) - 4N}{(16N)^{1/3}} = \Id_{\{u \leq 0\}} \left( \sqrt{2} \wt{\mathcal{B}}_\ell(-u) + \mathcal{A}_2(u) - u^2 \right),
	\end{equation}
	where the limit is understood as convergence in distribution of processes. Similarly, one also has
	\begin{equation}
		\lim_{N \to \infty} \frac{L_N^+(u) - 4N}{(16N)^{1/3}} = \Id_{\{u > 0\}} \left( \sqrt{2} \mathcal{B}_k(u) + \mathcal{A}_2(u) - u^2 \right).
	\end{equation}
	Since $\kappa$ in~\eqref{E:eps} is independent of $N$, we obtain
	\begin{equation}
		\left| \text{l.h.s.}~\eqref{E:lhs} - \Pb\left( \max_{|u| \leq M} \left( \mathbf{1}_{\{u \leq 0\}} \sqrt{2} \wt{\mathcal{B}}_\ell(-u) + \mathbf{1}_{\{u > 0\}} \sqrt{2} \mathcal{B}_k(u) + \mathcal{A}_2(u) - u^2 \right) \leq s \right) \right| \leq Ce^{-\kappa M}
	\end{equation}
	so that \eqref{E:var} then follows by letting $M \to \infty$.
\end{proof}

It remains to prove Proposition~\ref{p:exit}. We follow the method and notation developed in~\cite{BHA20}.
\begin{proof}[Proof of Proposition~\ref{p:exit}] For fixed $u,N>0$ with $N$ large enough, we have $\{|Z^{\mathbf N}|\geq uN^{2/3}\}=A\cup B$, where
	\begin{equation}
		\begin{aligned}
			A&:=\left\{\max_{x\geq uN^{2/3}}\left(L_{ (0,x)}+L_{(1,x)\to(N,N)}\right)= L_N\right\}\\
			B&:=\left\{\max_{x\geq uN^{2/3}}\left(L_{ (x,0)}+L_{(x,1)\to(N,N)}\right)= L_N\right\}
		\end{aligned}
	\end{equation}
	It suffices to bound the probabilities of the events $A$ and $B$. Since the procedures for the two events are similar, we focus on the event $A$. Note that
	\begin{equation}\label{E:sum}
		\Pb\left(A\right)\leq\sum_{r\geq u}^\infty\Pb\left(\max_{x\in [rN^{2/3},(r+1)N^{2/3}]}\left(L_{ (0,x)}+L_{(1,x)\to(N,N)}\right)\geq L_N\right).
	\end{equation}
	For each $r\geq u$, the summand above is bounded by
	\begin{equation}\label{E:split}
		\begin{aligned}
			& \Pb\left(\max_{x\in [rN^{2/3},(r+1)N^{2/3}]}\left(L_{ (0,x)}+L_{(1,x)\to(N,N)}\right)\geq \alpha\right)+\Pb\left(L_N\leq\alpha\right),
		\end{aligned}
	\end{equation}
	with $\alpha:=4N-r^2N^{1/3}/20$. By definition of LPP, we have $\mbox{$L_N\geq L_{(1,1)\to (N,N)}$}$. Under the setting of~\eqref{E:ThickBd}, the quantity $L_{(1,1) \to (N,N)}$ corresponds to last passage percolation without boundary conditions. Therefore, by Theorem~\ref{thm:deviation_ptp}, there exists a constant $\kappa > 0$ such that
	\begin{equation}\label{E:lm}
		\Pb\left(L_N\leq\alpha\right)\leq \Pb\left(L_{(1,1)\to (N,N)}\leq\alpha\right)\leq C  e^{-\kappa r^6}.
	\end{equation}
	Next, we bound the first term in \eqref{E:split}. We denote it by $P_1$. Note that
	\begin{equation}\label{E:rep}
		\begin{aligned}
			P_1&=\Pb\left(\max_{\gamma\in[r,r+1]}\left\{L_{ (0,\gamma N^{2/3})}+L_{(1,\gamma N^{2/3})\to(N,N)}-4N\right\}+ \tfrac{r^2N^{1/3}}{4}\geq \tfrac{r^2N^{1/3}}{5}\right)\\
			&\leq\Pb\left(\max_{\gamma\in[r,r+1]}\left\{L_{ (0,\gamma N^{2/3})}+L_{(1,\gamma N^{2/3})\to(N,N)}-4N+\tfrac{\gamma^2N^{1/3}}{4}\right\}\geq \tfrac{r^2N^{1/3}}{5}\right).
		\end{aligned}
	\end{equation} 	
	Note that for any $\gamma\in[r,r+1]$, we have
	\begin{equation}
		f(N-\gamma N^{2/3},N-1)=4N-2\gamma N^{2/3}-\tfrac{\gamma^2 N^{1/3}}{4}+\Or(1),
	\end{equation}
	where $f(m,n):=(\sqrt{m}+\sqrt{n})^2$. Thus, for $N$ large enough, we can bound the probability in the second line of~\eqref{E:rep} by
	\begin{equation}\label{E:la}
		\begin{aligned}
			&\Pb\left(\max_{\gamma\in[r,r+1]}\left\{L_{(1,\gamma N^{2/3})\to(N,N)}-f(N-\gamma N^{2/3},N-1)\right\}\geq\tfrac{r^2N^{1/3}}{10}\right)\\
			+ \,&\Pb\left(\max_{\gamma\in[r,r+1]}\left\{L_{ (0,\gamma N^{2/3})}-2\gamma N^{2/3}\right\}\geq\tfrac{r^2N^{1/3}}{10}\right)\\
		\end{aligned}
	\end{equation}
	
	\textit{First term in~\eqref{E:la}.} We denote the first term by $H_1$. By Lemma~\ref{lem:b}, there exists constant $\kappa$ independent of $r$ such that
	\begin{equation}\label{E:h1}
		H_1\leq e^{-\kappa |r|^3}.
	\end{equation}
	
	\textit{Second term in~\eqref{E:la}.} We denote the second term by $H_2$. Let $\gamma \in [r, r+1]$. First note that under the setting of~\eqref{E:ThickBd}, we have
	\begin{equation}
		L_{(0, \gamma N^{2/3})} = L_{(-\ell+1,1) \to (0, \gamma N^{2/3})}.
	\end{equation}
	Since $\ell$ is a fixed finite constant, for sufficiently large $N$, the second term in~\eqref{E:la} can be rewritten as follows
	\begin{equation}\label{E:ma}
		\mathbb{P}\left(
		\max_{\gamma \in [r, r+1]}
		\left\{
		L_{(-\ell+1,1) \to (0, \gamma N^{2/3})}
		- 2(\gamma N^{2/3} + \ell)
		\right\}
		\geq \tfrac{r^2 N^{1/3}}{10}
		\right).
	\end{equation}
	Note that the sequence $(M_n)_n$ is a submartingale, where
	\begin{equation}
		M_n := L_{(-\ell+1,1) \to (0,n)} - 2(n + \ell - 1).
	\end{equation}
	Indeed, by the definition of last passage percolation (LPP), we have
	\begin{equation}
		L_{(-\ell+1,1) \to (0,n+1)} \geq L_{(-\ell+1,1) \to (0,n)} + \omega_{0,n+1}.
	\end{equation}
	Together with the fact that $\E(\omega_{0,n+1}) = 2$, we obtain $\E(M_{n+1} \mid M_n) \geq M_n.$ Define now the function $g(x) := e^{\frac{x}{\sqrt{(r+1)\ell}}}$, which is convex and increasing. By Jensen's inequality, it follows that the sequence $(g(M_n))_n$ is also a submartingale. Applying Doob's maximal inequality, we obtain
	\begin{equation}
		\begin{aligned}
			&H_2=\Pb\left(\max_{\gamma \in [r, r+1]}M_{\gamma N^{2/3}}\geq \tfrac{r^2 N^{1/3}}{10}\right)
			\leq e^{-\tfrac{r^2}{10\sqrt{\left(r+1\right)\ell}}} \E\left(\exp\left(\tfrac{M_{(r+1)N^{2/3}}}{\sqrt{\left(r+1\right)\ell}N^{1/3}}\right)\right).
		\end{aligned}
	\end{equation}
	In Section~\ref{sec:emb}, we prove the boundedness of the above expectation by showing that the upper tail of
		$\tfrac{M_{(r+1)N^{2/3}}}{\sqrt{(r+1)\ell}\,N^{1/3}}$
		has (super-)exponential decay.
Hence, we have $H_2 \leq C e^{-Cr}$ for some constant $C > 0$. Plugging this and~\eqref{E:h1} into~\eqref{E:rep}, we obtain $P_1 \leq C e^{-Cr}$. Substituting this and~\eqref{E:lm} into~\eqref{E:split}, we conclude that $\Pb(A) \leq C e^{-Cr}$. This completes the proof.			
\end{proof}

		\appendix
        \section{Classical expression for the Baik--Rains distribution}\label{AppBR}
        The Baik--Rains distribution can be expressed in terms of solutions of Painlev\'e II equation~\cite{BR00} (reported also in Appendix~A of~\cite{FS05a}) or in terms of kernels and Fredholm determinants, see~\cite{FS05a}. For completeness we report this last expression here. Let
		\begin{equation}
			K_{\Ai, s}(x, y)=\int_{\R_{+}} \dx\lambda \Ai(\lambda+x+s) \Ai(\lambda+y+s) .
		\end{equation}
		and
	\begin{equation}
		\begin{aligned}
			& \widehat{\Phi}_{\tau, s}(x)=\int_{\R_{-}} \dx z e^{\tau z} K_{\Ai, s}(z, x) e^{\tau s}, \\
			& \widehat{\Psi}_{\tau, s}(y)=\int_{\R_{-}} \dx z e^{\tau z} \Ai(y+z+s), \\
			& \rho_s(x, y)=\left(\Id-P_0 K_{\Ai, s} P_0\right)^{-1}(x, y),
		\end{aligned}
	\end{equation}
	and the scaling function $g$ by	
	\begin{equation}
		g(s, \tau)=e^{-\frac{1}{3} \tau^3} \left[\int_{\R_{\le 0}^2} \dx x \dx y e^{\tau(x+y)} \Ai(x+y+s)+\int_{\R_{\ge 0}^2} \dx x \dx y \widehat{\Phi}_{\tau, s}(x) \rho_s(x, y) \widehat{\Psi}_{\tau, s}(y)\right].
	\end{equation}
	Then we have
	\begin{equation}
		F_\tau(s)=\frac{\partial}{\partial s}\left(F_{{\rm GUE}}\left(s+\tau^2\right) g\left(s+\tau^2, \tau\right)\right) .
	\end{equation}
As the Baik--Rains distribution for $\tau$ and $-\tau$ is the same, the above expressions are the ones for $\tau>0$, so that all the integrals converges. For $\tau=0$ one has to modify the integrals on $\R_{\le 0}$ to $\R_{\ge 0}$ using for instance the identity $\int_\R \Ai(x)=1$. The expression which is obtained can be found for instance in~\cite{BFP09} (for the $m=1$ case).

		\section{Some known bounds on LPP}
		In this section, we collect some known results regarding LPP. Since we only consider exponential random variables, we state the results directly in the exponential setting, without presenting the more general version for arbitrary weights.
		
		Our setup is as follows: for $i,j \in \Z$, let $\hat\omega_{i,j}~\Exp(1)$ be i.i.d.~exponential random variables, and denote by $\hat L_{m,n}$ the last passage time from $(1,1)$ to $(m,n)$ under this setting. For any fixed $p \in \Z_{>0}$, define the process
		\begin{equation}
			\mathcal{L}_N(u) = \frac{\hat L_{p,uN} - uN}{\sqrt{N}}.
		\end{equation}

		\begin{thm}[Theorem 3.2 of\cite{GW91}]\label{thm:strip}
			As $N \to \infty$, the process $\mathcal{L}_N(u)$ converges in distribution to the Brownian LPP process $\mathcal{B}_p(u)$ defined in~\eqref{E:blpp}.
		\end{thm}
		
		Define now $f(m,n):=(\sqrt{m}+\sqrt{n})^2$ which is the law of large number approximation of the LPP to $(m,n)$.
		\begin{thm}[Theorem 4.1 of~\cite{BGZ19}; Theorem 2 of~\cite{LR09}]\label{thm:deviation_ptp}
			For each $\psi>1$ There exists $C, c>0$ depending on $\psi$ such that for all $m, n, r \geq 1$ with $\psi^{-1}<\frac{m}{n}<\psi$ and all $x>0$ we have the following:
			\begin{align}
				&\Pb\left(\hat L_{m, n}- f(m,n)\geq x n^{1 / 3}\right) \leq C e^{-c \min \left\{x^{3 / 2}, x n^{1 / 3}\right\}}\\
				&\Pb\left(\hat L_{m, n}-f(m,n) \leq-x n^{1 / 3}\right) \leq C e^{-c x^3}.
			\end{align}
		\end{thm}
		
		 \begin{lem}[Lemma 4.6 of~\cite{BHA20}]\label{lem:b}
		 	For any fixed $r \in \R$, there exist constants $N_0, C, c > 0$ such that for all $N > N_0$, we have
		 	\begin{equation}
		 		\Pb\left(
		 		\max_{x \in  [ rN^{2/3}, (r+1)N^{2/3}  ]}
		 		\left\{ L_{(x,1) \to (N,N)} - f(N - x, N - 1) \right\}
		 		\geq \tfrac{r^2 N^{1/3}}{16}
		 		\right)
		 		\leq C e^{-c |r|^3}.
		 	\end{equation}
		 \end{lem}
		
		 \section{Exponential moment bound}\label{sec:emb}
         In Lemma~\ref{lemUpperTailThinRectangle} we provide a bound which implies that
		 \begin{equation}\label{11}
		 	\mathbb{E}\left(\exp \left(\frac{M_{(r+1) N^{2 / 3}}}{\sqrt{(r+1) \ell} N^{1 / 3}}\right)\right)<C,
		 \end{equation}
		 for some finite constant $C > 0$ (independent of $N$ and $r$) where
		 \begin{equation}
		 	M_n:=L_{(-\ell+1,1) \rightarrow(0, n)}-2(n+\ell-1) .
		 \end{equation}
		 By definition of our model, all the weights on the points along the right-up paths connecting $(-\ell+1,1)$ and $(0, n)$ are distributed as $\operatorname{Exp}\left(\frac{1}{2}\right)$. This motivates the following lemma.
		 \begin{lem}\label{lemUpperTailThinRectangle}
		 	Consider a last passage percolation model on $\mathbb{Z}^2$ in which the weights $\omega_p \sim \mathrm{Exp}(\tfrac12)$ for all $p \in \mathbb{Z}^2$. For $p_1,p_2 \in \mathbb{Z}^2$, we denote by $L_{p_1,p_2}$ the corresponding last passage time.	
		 	For fixed $\ell\in\Z_{>0}$, there exists a constant $C$ such that, for all $r,S\in \mathbb{R}_+$ and $N$ sufficiently large, it holds
		 	\begin{equation}\label{c3}
		 		\mathbb{P}\!\left(
		 		L_{(1,1),\,((r+1)N^{2/3},\,\ell)}
		 		> 2(r+1)N^{2/3} + S\sqrt{r+1}\,N^{1/3}
		 		\right)
		 		\leq
		 		C \exp\!\left(-\tfrac{1}{48}\min\{S^2,\, S N^{1/6}\}\right).
		 	\end{equation}
		 	
		 \end{lem}
		 We then have
		 \begin{equation}
		 	M_n \stackrel{d}{=} L_{(1,1),(n,\ell)} - 2(n+\ell-1).
		 \end{equation}
		 Hence, to prove~\eqref{11}, it suffices to establish an exponential decay for
		 \begin{equation}
		 	\mathbb{P}\!\left(
		 	L_{(1,1),\,((r+1)N^{2/3},\,\ell)}
		 	> 2(r+1)N^{2/3} + 2(\ell-1)
		 	+ S\sqrt{(r+1)\ell}\,N^{1/3}
		 	\right)
		 \end{equation}
		 for $S \geq S_0 \geq 0$.
		 Note that the term $2(\ell-1)$ is irrelevant and can always be absorbed by shifting $S$ by an amount of order $\mathcal{O}(N^{-1/3})$.
		 Moreover, once such a bound is obtained for some fixed $S_0$, the same estimate holds for all $S_0 \geq 0$ after adjusting the prefactor.
		 Consequently,~\eqref{c3} indeed implies~\eqref{11}.

		 \begin{proof}[Proof of Lemma~\ref{lemUpperTailThinRectangle}]
		 	For $n,\ell\in\Z$, we define
		 	\begin{equation}
		 		L:=L_{(1,1),(n,\ell)},
		 	\end{equation}
		 	whose distribution is given by
		 	\begin{equation}
		 		\Pb(L\leq a)=\det(\Id-K)_{L^2((a,\infty))},
		 	\end{equation}
		 	where the kernel\footnote{The kernel can be obtained by a few manipulations from the work~\cite{Jo00b}, or more easily from Theorem 3 of~\cite{BP07}.} is given by
		 	\begin{equation}\label{1.5}
		 		K(x,y)=\frac{1}{(2\pi\I)^2}\oint_{\Gamma_0}\dx w \oint_{\Gamma_{1/2}}\dx z \frac{(1-2 w)^n}{(1-2 z)^n}\frac{z^\ell}{w^\ell} \frac{e^{w y-z x}}{w-z}.
		 	\end{equation}
		 	We want to have a bound for
		 	\begin{equation}
		 		n=(r+1)N^{2/3},\quad a=2n+S \sqrt{r+1}N^{1/3}
		 	\end{equation}
		 	for $S\geq S_0\geq 0$.
		 	Since $\ell$ is fixed and we consider large $N$, we can assume that $n=(r+1)N^{2/3}>\ell$. Also, for $n>\ell$ and $x\geq 0$ there is no pole at $\infty$, so that the contour $\Gamma_{1/2}$ for variable $z$ in \eqref{1.5} can be deformed to be $\{R-\I y,y\in\R\}$ with $R>|w|$.
		 	
		 	Denote $\e=(r+1)^{-1/2}N^{-1/3}$, so that $n=\e^{-2}$, and consider the following change of variables in the kernel
		 	\begin{equation}
		 		w=\e W,\quad z=\e Z,\quad x=2\e^{-2}+\xi \e^{-1},\quad y=2\e^{-2}+\zeta \e^{-1}.
		 	\end{equation}
            We consider the rescaled and conjugated kernel $K^{\rm resc,\e}(\xi,\zeta):=\e^{-1}e^{(\xi-\zeta)/16} K(x,y)$ with the above change of variables. After the change of variables $(w,z)\mapsto (\e W,\e Z)$ we obtain
		 	\begin{equation}\label{re}
		 		K^{\rm resc,\e}(\xi,\zeta)=\frac{e^{(\xi-\zeta)/16}}{(2\pi\I)^2}\oint_{\Gamma_0} \dx W \int_{\rho-\I \R} \dx Z \frac{Z^\ell}{W^\ell} \frac{e^{\zeta W-\xi Z}}{W-Z} \frac{e^{\Phi_{\e}(W)}}{e^{\Phi_{\e}(Z)}},
		 	\end{equation}
		 	where
		 	\begin{equation}
		 		\Phi_\e(Z)=\e^{-2}\ln(1-2\e Z)+2\e^{-1}Z.
		 	\end{equation}
		 	Since we only need to get the bound for large $\xi,\zeta$, we assume without loss of generality that $\xi,\zeta\geq 1$. We choose the following paths for the integrals:
		 	\begin{equation}\label{wcontour}
		 		W=\{\tfrac{1}{16} e^{\I\phi},\phi\in [-\pi,\pi)\}\textrm{ and }Z=\{\rho(\xi)-\I y,y\in\R\},
		 	\end{equation}
		 	with $\rho(\xi)=\tfrac16\xi\Id_{1\leq \xi\leq \e^{-1/2}}+\tfrac16 \e^{-1/2}\Id_{\xi\geq \e^{-1/2}}$. Note that by our choice, $\rho(\xi)$ is at most $\tfrac16\e^{-1/2}$.
		 	
		 	We first show that the integral in $W$ is bounded by a constant uniformly for all $Z\in\rho-\I\R.$ Applying Taylor expansion, we obtain
		 	\begin{equation}\label{1.12}
		 		\Phi_\e(W)=-2 W^2+\Or(\e W^3).
		 	\end{equation}
		 	Consider now an arbitrary $Z\in\rho-\I \R$, then we have
		 	\begin{equation}
		 		\begin{aligned}
		 			&\left|\oint_{\Gamma_0} \dx W \frac{e^{\zeta W+\Phi_\varepsilon(W)}}{W^\ell(W-Z)}\right|\leq \oint_{\Gamma_0} \dx W \frac{e^{\zeta \Re(W)+\Re(\Phi_\varepsilon(W))}}{\left|W^\ell(W-Z)\right|}
		 		\end{aligned}
		 	\end{equation}
		 	By \eqref{wcontour}, we have $e^{\zeta \Re(W)}=e^{\zeta/16}$, $|W^{-\ell}|\leq 16^\ell$  and $1/|W-Z|$ is bounded by a constant independent of $\e$. Hence, we have
		 	\begin{equation}\label{1.14}
		 		\left|\oint_{\Gamma_0} \dx W \frac{e^{\zeta W+\Phi_\varepsilon(W)}}{W^\ell(W-Z)}\right|\leq C' e^{\zeta/16}\oint_{\Gamma_0} \dx We^{\Re(\Phi_\varepsilon(W))}\stackrel{\eqref{1.12}}{\leq } Ce^{\zeta/16},\\
		 	\end{equation}
		 	for some constants $C,C'$ independent of $\e$.
		 	
		 	Finally we need to bound the integrand in the $Z=\rho-\I y$ variable. We consider the cases of small and large values of $|y|$ separately.
		 	
		 	(a) For $|y|\leq \e^{-1/2}$, we have $|\e Z|=\Or(\sqrt{\e})$ since both $y$ and $\rho$ are $\Or(\e^{-1/2})$. Applying Taylor expansion, we obtain
		 	\begin{equation}
		 		-\Phi_\e(Z)=2 Z^2(1+\Or(\e Z)),\textrm{ so that }\Re(-\Phi_\e(Z))\leq \tfrac52 (\rho^2-y^2)
		 	\end{equation}
		 	uniformly for all $\e>0$ small enough.
		 	
		 	(b) For $|y|>\e^{-1/2}$, we have
		 	\begin{equation}\label{1.16}
		 		\Re(-\Phi_\e(Z))=-\e^{-2}\ln((1-2\e\rho))-2\e^{-1}\rho-\tfrac12\e^{-2}\ln\left(1+\frac{4\e^2 y^2}{(1-2\e\rho)^2}\right)
		 	\end{equation}
		 	As $\rho\e=\Or(\sqrt{\e})$, by Taylor expansion in the first term we get
		 	\begin{equation}\label{1.17}
		 		\Re(-\Phi_\e(Z))\leq 2\rho^2(1+\Or(\e\rho))-\tfrac12\e^{-2}\ln(1+4\e^2 y^2)\leq \frac52\rho^2-\tfrac12\e^{-2}\ln(1+4\e^2 y^2)
		 	\end{equation}
		 	uniformly for all $\e>0$ small enough. Plugging $|Z|^\ell\leq 2^{\ell/2} (\rho^\ell+y^\ell)$, \eqref{1.14} and \eqref{1.17} into \eqref{re}, then for some constants $C_1,C_2>0$ we have
		 	\begin{equation}
		 		\begin{aligned}
		 			|K^{\rm resc,\e}(\xi,\zeta)| &\leq C_1 e^{\xi(1/16-\rho)} e^{\frac52 \rho^2} \int_\R \dx y (\rho^\ell+y^\ell) \left(e^{-y^2}\Id_{|y|\leq \e^{-1/2}}+\frac{\Id_{|y|>\e^{-1/2}}}{(1+4\e^2 y^2)^{\e^{-2}/2}}\right)\\
		 			&\leq  C_2 e^{3\rho^2+\xi/16-\rho\xi}
		 		\end{aligned}
		 	\end{equation}
		 	for all $\e>0$ small enough.
		 	
		 	For $1\leq\xi\leq\e^{-1/2}$, we have $\rho=\tfrac16\xi$ and therefore
		 	\begin{equation}
		 		3 \rho^2+\frac{1}{16} \xi-\rho \xi=-\frac{1}{12} \xi^2+\frac{1}{16} \xi\leq -\frac{1}{48} \xi^2.
		 	\end{equation}
		 	For $\xi\geq\e^{-1/2}$, we have $\rho=\tfrac16\e^{-1/2}$, which then implies
		 	\begin{equation}
		 		3 \rho^2+\frac{1}{16} \xi-\rho \xi=\frac{\e}{12}+\xi\left(\frac{1}{16}-\frac{\e^{-1/2}}{6}\right)\leq -\frac{1}{12} \xi\e^{-1/2}
		 	\end{equation}
		 	for all $\xi\geq\e^{-1/2}\geq 3$. To summarize, for all $\xi\geq 1$ (and $\e>0$ small enough) we get
		 	\begin{equation}
		 		3\rho^2+\tfrac1{16}\xi-\rho\xi\leq -\tfrac1{48}\xi^2\Id_{1\leq \xi\leq \e^{-1/2}}-\tfrac1{12}\xi\e^{-1/2}\Id_{\xi>\e^{-1/2}}=:\Psi_\e(\xi).
		 	\end{equation}
		 	
		 	Thus we have for some constant $C_3>0$
		 	\begin{equation}
		 		\begin{aligned}
		 			&\Pb(L_{(1,1),((r+1)N^{2/3},\ell)}>2(r+1)N^{2/3}+S \sqrt{r+1}N^{1/3})\\
		 			&=\sum_{m\geq 1} \frac{(-1)^{m-1}}{m!} \int_{S}^\infty \dx \xi_1\cdots\int_S^\infty \dx \xi_m \det(K^{\rm resc,\e}(\xi_i,\xi_j))_{1\leq i,j\leq m}\\
		 			&\leq \sum_{m\geq 1} \frac{C_2}{m!} \int_S^\infty \dx \xi_1\cdots\int_S^\infty \dx \xi_m\, m^{m/2}\prod_{k=1}^m e^{\Psi_\e(\xi_k)} \\
		 			&\leq C_3 e^{\Psi_\e(S)}=C_3 e^{-\tfrac1{48}S^2}\Id_{1\leq S\leq \e^{-1/2}} + C_3 e^{-\tfrac1{12}S\e^{-1/2}} \Id_{S>\e^{-1/2}}
		 		\end{aligned}
		 	\end{equation}
		 	where we used Hadamard bound that says that for a $m\times m$ matrix $A$ with $|A_{i,j}|\leq 1$, $|\det(A)|\leq m^{m/2}$. Replacing $\e^{-1}=\sqrt{r+1}N^{1/3}$ and considering that $r+1\geq 1$, we get
		 	\begin{equation}
		 		\Pb(L_{(1,1),((r+1)N^{2/3},\ell)}>2(r+1)N^{2/3}+S \sqrt{r+1}N^{1/3})\leq C_3 e^{-\tfrac1{48}\min\{S^2,S N^{1/6}\}}
		 	\end{equation}
		 \end{proof}
		
\section{Well-definedness of the Fredholm determinant}\label{Sec:Well}
Let us show that the Fredholm determinant in Theorem~\ref{thm:expKernel} is well defined by considering a conjugated kernel and proving that each entry of the kernel exhibits exponential decay, uniformly for all $\bma,\bmb>-\tfrac12+\e$ for an arbitrarily small $\e>0$. We define
\begin{equation}
	\tilde{\mathbf{R}} := \frac{\varepsilon}{2} - \min\Bigl\{0, \min_{j\in\llbracket k\rrbracket} \beta_j, \min_{i\in\llbracket \ell\rrbracket} \alpha_i \Bigr\}.
\end{equation}
Since $\bma,\bmb > -\frac{1}{2} + \varepsilon$, we have
\begin{equation}
	\tilde{\mathbf{R}} \le \frac{\varepsilon}{2} + \frac{1}{2} - \varepsilon = \frac{1}{2} - \frac{\varepsilon}{2}.
\end{equation}
Moreover, by definition,
\begin{equation}
	\begin{aligned}
		\tilde{\mathbf{R}} - \max_{j\in\llbracket k\rrbracket} \{-\beta_j\}
		&= \tilde{\mathbf{R}} + \min_{j\in\llbracket k\rrbracket} \beta_j \ge \frac{\varepsilon}{2},\\
		\tilde{\mathbf{R}} + \min_{i\in\llbracket \ell\rrbracket} \alpha_i
		&\ge \frac{\varepsilon}{2}.
	\end{aligned}
\end{equation}

Hence, we can choose the contours $\Gamma_{-1/2},\Gamma_{-1/2,-\bmb},\Gamma_{1/2},\Gamma_{\bma}$, and $\Gamma_{1/2,\bma}$ such that
\begin{align}
	&\tilde\bfR-\Re(w)>\tfrac{\e}{3},\quad\forall w\in\Gamma_{-1/2,-\bmb},\quad \Re(v)-\tilde \bfR>\tfrac\e3,\quad\forall v\in\Gamma_{1/2},\label{2.5}\\
	&\Re(w)+\tilde\bfR<-\tfrac\e3,\quad\forall w\in\Gamma_{-1/2},\quad \Re(v)+\tilde\bfR>\tfrac\e3,\quad\forall v\in\Gamma_{1/2,\bma},\label{2.6}\\
	&\Re(w)+\tilde\bfR<-\tfrac\e3,\quad\forall w\in\Gamma_{-1/2},\quad \Re(v)-\tilde\bfR>\tfrac\e3,\quad\forall v\in\Gamma_{1/2},\label{2.7}\\
	&\tilde\bfR-\Re(w)>\tfrac\e3,\quad\forall w\in\Gamma_{-1/2,-\bmb},\quad \Re(v)+\tilde\bfR>\tfrac\e3,\quad\forall v\in\Gamma_{\bma}\label{2.8}.
\end{align}

Next, we define the conjugation function.		
\begin{equation}
	T_{ij}(x,y):=\begin{cases}
		e^{\tilde{\bfR}(y-x)}, & \text{if } (i,j)=(1,1),\\
		e^{\tilde{\bfR}(x-y)}, & \text{if } (i,j)=(2,2),\\
		e^{\tilde{\bfR}(x+y)}, & \text{if } (i,j)=(2,1),\\
		e^{-\tilde{\bfR}(x+y)}, & \text{if } (i,j)=(1,2).
	\end{cases}
\end{equation}
\begin{lem}
	Let $K_{m,n,ij}^{\rm exp}$ be defined as in Theorem~\ref{thm:expKernel} with $\bma,\bmb>-\tfrac12+\e$. Then we have
	\begin{equation}
		|K_{m,n,ij}^{\rm exp}(x,y)T_{ij}(x,y)|\leq C\,e^{-\varepsilon(x+y)/3}
	\end{equation}
	for some constant $C\in\mathbb{R}_+$.
	
\end{lem}
\begin{proof}			
	\textbf{For $K_{m,n,11}^{\exp}(x,y)$.}Note that we can choose the contours $\Gamma_{-1/2,-\bmb}$ and $\Gamma_{1/2}$ such that
	\begin{equation}\label{up}
		\left|\left(\frac{1/2-w}{1/2-v}\right)^m \left(\frac{1/2+v}{1/2+w}\right)^n \frac{1}{v-w} \prod_{j=1}^{k} \frac{v+\beta_j}{w+\beta_j}\right| \leq C
	\end{equation}
	for some constant $C$, uniformly for all $w \in \Gamma_{-1/2,-\bmb}$ and $v \in \Gamma_{1/2}$. Hence, we obtain
	\begin{equation}
		\begin{aligned}
			|K_{m,n,11}^{\rm exp}(x,y) T_{11}(x,y)| &\leq C \oint_{\Gamma_{-1/2,-\bmb}} \mathrm{d} w \oint_{\Gamma_{1/2}} \mathrm{d} v \,
			e^{-(\tilde{\bfR}-\Re(w))x} e^{-(\Re(v)-\tilde{\bfR})y} \\
			&\stackrel{\eqref{2.5}}{\leq} C e^{-\varepsilon(x+y)/3}.
		\end{aligned}
	\end{equation}	
	\textbf{For $K_{m,n,22}^{\exp}(x,y)$.} Note that we can choose the contours $\Gamma_{-1/2}$ and $\Gamma_{1/2,\bma}$ such that
	\begin{equation}
		\left|	\left(\frac{1 / 2-w}{1 / 2-v}\right)^m\left(\frac{1/2 + v}{1/2 + w} \right)^{n} \frac{1}{v-w} \prod_{i=1}^{\ell} \frac{w - \alpha_i}{v - \alpha_i}\right|\leq C
	\end{equation}
	for some constant $C$, uniformly for all $w\in \Gamma_{-1/2}$ and $v\in\Gamma_{1/2,\bma}$. Therefore
	\begin{equation}
		\begin{aligned}
			|K_{m,n,22}^{\exp}(x,y)T_{22}(x,y)| &\leq C\oint_{\Gamma_{-1/2}} \dx w\oint_{\Gamma_{1/2,\bma}} \dx v  \,
			e^{(\tilde\bfR+\Re(w))x}e^{-(\Re(v)+\tilde\bfR)y}\\
			&\stackrel{\eqref{2.6}}{\leq} Ce^{-\e(x+y)/3}.
		\end{aligned}
	\end{equation}  	
	\textbf{For $K_{m,n,21}^{\exp}(x,y)$.} Note that we can choose the contours $\Gamma_{-1/2}$ and $\Gamma_{1/2}$ such that
	\begin{equation}
		\left|	\left(\frac{1 / 2-w}{1 / 2-v}\right)^m\left(\frac{1/2 + v}{1/2 + w} \right)^{n} \frac{1}{v-w} \prod_{j=1}^{k} (v + \beta_j) \prod_{i=1}^{\ell} (w - \alpha_i)\right|\leq C
	\end{equation}
	for some constant $C$, uniformly for all $w\in \Gamma_{-1/2}$ and $v\in\Gamma_{1/2}$. Thus, we obtain
	\begin{equation}
		\begin{aligned}
			|K_{m,n,21}^{\exp}(x,y)T_{21}(x,y)|&\leq C\oint_{\Gamma_{-1/2}} \dx w\oint_{\Gamma_{1/2}} \dx v  \,
			e^{(\Re(w)+\tilde\bfR)x}e^{-(\Re(v)-\tilde\bfR)y}\\
			&\stackrel{\eqref{2.7}}{\leq}Ce^{-\e(x+y)/3}.
		\end{aligned}
	\end{equation}
	\textbf{For $K_{m,n,12}^{\exp}(x, y)$.} Without loss of generality, we consider $x>y$, which implies
	\begin{equation}
		\begin{aligned}
			K_{m,n,12}^{\rm exp}(x, y) &:= \frac{-1}{(2\pi \mathrm{i})^2}
			\oint_{\Gamma_{-1/2,-\bmb}} \hspace{-1em} \mathrm{d} w \oint_{\Gamma_{1/2,\bma,w}} \hspace{-1em} \mathrm{d} v \,
			\frac{E(x, w; y, v)}{v-w} \frac{1}{\prod_{j=1}^{k} (w + \beta_j) \prod_{i=1}^{\ell} (v - \alpha_i)}.
		\end{aligned}
	\end{equation}
	Applying Cauchy's residue theorem, we obtain
	\begin{equation}
		\begin{aligned}
			K_{m,n,12}^{\exp}(x, y) &:= \frac{-1}{(2\pi\I)^2}
			\oint_{\Gamma_{-1/2,-\bmb}} \hspace{-1em}\dx w\oint_{\Gamma_{1/2,\bma,w}} \hspace{-1em}\dx v
			\,\frac{E(x, w; y, v)}{v-w}\frac{1}{ \prod_{j=1}^{k} (w + \beta_j) \prod_{i=1}^{\ell} (v - \alpha_i) }\\
			&=\frac{-1}{(2\pi\I)^2}
			\oint_{\Gamma_{-1/2,-\bmb}} \hspace{-1em}\dx w\oint_{\Gamma_{1/2}} \dx v
			\,\frac{E(x, w; y, v)}{v-w}\frac{1}{ \prod_{j=1}^{k} (w + \beta_j) \prod_{i=1}^{\ell} (v - \alpha_i) }\\
			&+\frac{-1}{(2\pi\I)^2}
			\oint_{\Gamma_{-1/2,-\bmb}} \hspace{-1em}\dx w\oint_{\Gamma_{\bma}} \dx v
			\,\frac{E(x, w; y, v)}{v-w}\frac{1}{ \prod_{j=1}^{k} (w + \beta_j) \prod_{i=1}^{\ell} (v - \alpha_i) }\\
			&+\underbrace{\frac{-1}{(2\pi\I)^2}\oint_{\Gamma_{-\bmb}} \dx w\frac{e^{w(x-y)}}{ \prod_{j=1}^{k} (w + \beta_j) \prod_{i=1}^{\ell} (w - \alpha_i) }}_{=:K(x,y)}.
		\end{aligned}
	\end{equation}
	For the first two terms on the right hand side of the second equation, we can do the same as we did before and then apply \eqref{2.8} to obtain that the absolute value of the sum of those two is bounded by $Ce^{-\e(x+y)/2}$. As for the last term, we can choose $\Gamma_{-\bmb}$ such that
	\begin{equation}
		\left|\frac{1}{ \prod_{j=1}^{k} (w + \beta_j) \prod_{i=1}^{\ell} (w - \alpha_i) }\right|\leq C.
	\end{equation}
	Thus, we have
	\begin{equation}
		\begin{aligned}
			\left|K(x,y)T_{12}(x,y)\right|\leq &\oint_{\Gamma_{-\bmb}} \dx w e^{w(x-y)} e^{-\tilde\bfR(x+y)}\\
			=&C\sum_{\gamma\in\{\beta_i\mid i\in\llbracket k\rrbracket\}}\oint_{\Gamma_{-\gamma}} \dx w e^{w(x-y)} e^{-\tilde\bfR(x+y)}.\\
		\end{aligned}
	\end{equation}
	For $\gamma\in\{\beta_i\mid i\in\llbracket k\rrbracket\}$, if $\gamma\geq0$, then we can choose $\Gamma_{-\gamma}$ such that $\Re(w)<0$ for all $w\in\Gamma_{-\gamma}$, together with the assumption $x>y$ and $\tilde\bfR>\tfrac\e3$  (for $\e$ small), we have
	\begin{equation}
		\left|\oint_{\Gamma_{-\gamma}} \dx w e^{w(x-y)} e^{-\tilde\bfR(x+y)}\right|\leq Ce^{-\e(x+y)/3}.
	\end{equation}
	If $0>\gamma>-\tfrac12$, then we have
	\begin{equation}
		\begin{aligned}
			&\oint_{\Gamma_{-\gamma}} \dx w e^{(\Re(w)-\tilde\bfR)x} e^{-(\Re(w)+\tilde\bfR)y}\\
		\end{aligned}
	\end{equation}
	By definition of $\tilde\bfR$, we can choose $\Gamma_{-\gamma}$ such that $\Re(w)-\tilde\bfR<-\tfrac\e3$ and $\Re(w)+\tilde\bfR>\tfrac\e3$ for all $w\in\Gamma_{-\gamma}$, which then implies
	\begin{equation}
		\left|\oint_{\Gamma_{-\gamma}} \dx w e^{w(x-y)} e^{-\tilde\bfR(x+y)}\right|\leq Ce^{-\e(x+y)/3}.
	\end{equation}
	To conclude, also for this entry of the kernel, we have
	\begin{equation}
		|K^{\exp}_{m,n,12}(x,y)T_{12}(x,y)|\leq Ce^{-\e(x+y)/3}.
	\end{equation}

\end{proof}

\end{document}